\documentclass[a4paper,12pt,final]{amsart}
\usepackage[utf8]{inputenc}
\usepackage[T1]{fontenc}
\usepackage[UKenglish]{babel}
\usepackage[a4paper,margin=28mm]{geometry}
\usepackage{verbatim}
\allowdisplaybreaks[4]

\usepackage{times}
\usepackage{dsfont,mathrsfs}
\usepackage{amsmath}
\usepackage{amsthm}
\usepackage{amssymb}
\usepackage{amsfonts}
\usepackage{latexsym}
\usepackage{booktabs}
\usepackage{graphicx}
\usepackage{epstopdf}
\usepackage{float}
\usepackage{subfigure}

\usepackage{xcolor}

\usepackage{ulem}
\usepackage{bm}

\usepackage{url}

\makeatletter
\@namedef{subjclassname@2020}{\textup{2020} Mathematics Subject Classification}
\makeatother

\newtheorem{theorem}{Theorem}[section]
\newtheorem{lemma}[theorem]{Lemma}
\newtheorem{proposition}[theorem]{Proposition}

\newtheorem{assumption}[theorem]{Assumption}

\theoremstyle{definition}

\newtheorem{example}[theorem]{Example}
\newtheorem{remark}[theorem]{Remark}
\numberwithin{equation}{section}
\renewcommand{\labelenumi}{\roman{enumi})}
\renewcommand\theenumi\labelenumi
\renewcommand{\leq}{\leqslant}
\renewcommand{\le}{\leqslant}
\renewcommand{\geq}{\geqslant}
\renewcommand{\ge}{\geqslant}

\newcommand{\tl}{\tilde}
\newcommand{\Be}{\begin{equation}}
\newcommand{\Ees}{\end{equation*}}
\newcommand{\Bes}{\begin{equation*}}
\newcommand{\Ee}{\end{equation}}

\newcommand{\R}{\mathbb{R}}
\newcommand{\E}{\mathbb{E}}
\newcommand{\e}{\varepsilon}

\newcommand{\PP}{\mathbb{P}}

\newcommand{\mcl}{\mathcal}

\newcommand{\bbR}{\mathbb{R}}

\newcommand{\bfa}{\mathbf{a}}
\newcommand{\bfb}{\mathbf{b}}
\newcommand{\bfc}{\mathbf{c}}

\newcommand{\bfp}{\mathbf{p}}
\newcommand{\bfr}{\mathbf{r}}
\newcommand{\bfu}{\mathbf{u}}
\newcommand{\bfv}{\mathbf{v}}

\newcommand{\bfx}{\mathbf{x}}
\newcommand{\bfy}{\mathbf{y}}
\newcommand{\bfz}{\mathbf{z}}

\newcommand{\bfA}{\mathbf{A}}
\newcommand{\bfB}{\mathbf{B}}

\newcommand{\bfI}{\mathbf{I}}
\newcommand{\bfJ}{\mathbf{J}}
\newcommand{\bfQ}{\mathbf{Q}}

\newcommand{\bfM}{\mathbf{M}}
\newcommand{\bfN}{\mathbf{N}}

\newcommand{\bfX}{\mathbf{X}}
\newcommand{\bfY}{\mathbf{Y}}
\newcommand{\bfZ}{\mathbf{Z}}

\newcommand{\bmGamma}{\bm{\Gamma}}
\newcommand{\bmsigma}{\bm{\sigma}}
\newcommand{\bmtheta}{\bm{\theta}}
\newcommand{\rme}{\mathrm{e}}

\newcommand{\bmell}{\bm{\ell}}

\newcommand{\dif}{\mathop{}\!\mathrm{d}}

\newcommand{\I}{\mathds{1}}

\newcommand{\opP}{\operatorname{P}}
\newcommand{\opQ}{\operatorname{Q}}

\usepackage{marginnote}
\usepackage[notref,notcite]{showkeys}

\usepackage[graphicx]{realboxes}

\usepackage[colorlinks,linkcolor=red,filecolor=blue,urlcolor=cyan,citecolor=green,pagebackref]{hyperref}

\begin{document}

\title[Unbiased approximation of ergodic measure of piecewise $\alpha$-stable OU processes]
{Unbiased approximation of the ergodic measure for piecewise $\alpha$-stable Ornstein-Uhlenbeck processes arising in queueing networks}

\author[X. Jin]{Xinghu Jin}
\address{Xinghu Jin: School of Mathematics, Hefei University of Technology, Hefei, Anhui, 230601, China.}
\email{xinghujin@hfut.edu.cn}

\author[G. Pang]{Guodong Pang}
\address{Guodong Pang: Department of Computational Applied Mathematics and Operations Research,
George R. Brown School of Engineering,
Rice University,
Houston, TX 77005}
\email{gdpang@rice.edu}

\author[Y. Wang]{Yu Wang}
\address{Yu Wang: 1. Department of Mathematics, Faculty of Science and Technology, University of Macau, Taipa, Macau, 999078, China; \ \
2. UM Zhuhai Research Institute, Zhuhai, Guangdong, 519000, China.}
\email{yc17447@um.edu.mo}

\author[L. Xu]{Lihu Xu}
\address{Lihu Xu: 1. Department of Mathematics, Faculty of Science and Technology, University of Macau, Taipa, Macau, 999078, China; \ \
2. UM Zhuhai Research Institute, Zhuhai, Guangdong, 519000, China.}
\email{lihuxu@umac.mo}



\begin{abstract}
Piecewise $\alpha$-stable Ornstein-Uhlenbeck (OU) processes arising in queue networks usually do not have an explicit dissipation, which makes the related numerical methods such as Euler-Maruyama (EM) scheme more difficult to analyze.
We develop an EM scheme with decreasing step size $\Lambda=(\eta_n)_{n\in \mathbb{N}}$ to approximate their ergodic measures. This approximation does not have a bias and has a rate $\eta^{1/\alpha}_n$ in Wasserstein-1 distance.
We show by the classical OU process that our convergence rate is optimal.
We further prove the central limit theorem (CLT) and moderate derivation principle (MDP) for the empirical measure of these piecewise $\alpha$-stable Ornstein-Uhlenbeck processes.
In addition, we use the Sinkhorn--Knopp algorithm to compute the Wasserstein-1 distance and conduct simulations for several concrete examples.
\end{abstract}

\maketitle

 \tableofcontents\thispagestyle{plain}

\section{Introduction}\label{sec:Intro}

A large family of diffusion-scaled queueing processes in stochastic networks in heavy traffic converge to a universal limit, which is governed by the following stochastic differential equation (SDE):
\begin{equation}\label{e:SDE}
	\dif  \bfX_t = g( \bfX_t) \dif t+ \bmsigma \dif  \bfZ_t, \quad  \bfX_0 = \bfx\in\R^d,
\end{equation}
where $ \bfX_t$ is $\R^d$-valued for each $t \ge 0$, the coefficient $\bmsigma = [\bmsigma_1, \dotsc, \bmsigma_d] $ is a non-singular constant matrix in $\R^{d\times d}$ with each {column vector} $\bmsigma_i \in \R^d$, and $g: \mathbb{R}^d\to \R^d $ is a Lipschitz function defined as
\begin{equation}\label{e:g}
	\begin{aligned}
		g( \bfx) &= \bmell- \bfM ( \bfx-\langle {\rm \bf e}_d, \bfx \rangle^+  \bfv)
		-\langle {\rm \bf e}_d, \bfx \rangle^+  \bmGamma  \bfv \\
		&=
		\begin{cases}
			\bmell-( \bfM + ( \bmGamma- \bfM ) \bfv{\rm \bf e}_d^{\prime}) \bfx,   & \text{ \quad if } \  {\rm \bf e}_d^{\prime}  \bfx >0,  \\
			\bmell- \bfM  \bfx,   & \text{ \quad if }\  {\rm \bf e}_d^{\prime}  \bfx \leq 0,
		\end{cases}
	\end{aligned}
\end{equation}
where $\bmell\in \R^d$, the vector $\bfv\in \R^d_+$ satisfies $\langle {\rm \bf e}_d, \bfv \rangle={\rm \bf e}_d^{\prime} \bfv=1$ with ${\rm \bf e}_d=(1,\cdots,1)^{\prime} \in \R^d$, the matrix $ \bfM \in \R^{d\times d}$ is a non-singular $\mathrm{M}$-matrix (see, its definition below and more details in \cite[Chapter 6]{Berman1994Nonnegative}), the term $\langle {\rm \bf e}_d, \bfx \rangle^+$ means the positive part of $\langle {\rm \bf e}_d, \bfx \rangle$, and $ \bmGamma={\rm diag}(\gamma_1,\cdots,\gamma_d)$ with $\gamma_i\in \R_+, i=1,\cdots,d$.  Here, we use notation $\bfz^\prime$ for  the transpose of any $\bfz\in \R^d$. Furthermore, the $\alpha$ stable noise $\bfZ_t$ is one of the below processes:
\begin{itemize}
	\item [(i)] For $\alpha=2$, the noise $\bfZ_t$ is the standard $d$ dimensional Brownian motion;
	\item [(ii)] For $\alpha\in (1,2)$, the noise $\mathbf{Z}_t = (Z_t^1, \dots, Z_t^d)^{\prime}$ is the cylindrical  stable noise, and $Z_t^i, i=1,\dotsc,d$ are i.i.d. symmetric $\alpha$-stable processes in dimension $1$ with L\'evy measure $\nu_1(\dif z)=c_{\alpha}|z|^{-1-\alpha} \dif z$, where  $c_{\alpha}= \alpha 2^{\alpha-1}\pi^{-1/2} \Gamma ({( 1+\alpha ) }/{2} )/ \Gamma (1-{\alpha}/{2} )$;
	\item [(iii)] For $\alpha \in (1,2)$, the noise $\bfZ_t$ is the $d$ dimensional rotationally symmetric $\alpha$ stable noise with L\'{e}vy measure $\nu(\dif \bfz)=C_{d,\alpha}|\bfz|^{-\alpha-d} \dif \bfz$, where the constant $C_{d,\alpha}$ is given by
	\begin{equation*}
		C_{d,\alpha} =  \alpha2^{\alpha-1}\pi^{-d/2} \frac{\Gamma\big({( d+\alpha) }/{2} \big)}{\Gamma(1- {\alpha}/{2})} .
	\end{equation*}
\end{itemize}

The cases in (i) and (ii) arise from multiclass parallel server networks and many-server queues with phase-type service times with light-tailed and heavy-tailed arrival processes, respectively. We also consider the case (iii) which have applications in data science (see, e.g.,  \cite{chen2021solving}, \cite{raj2023algorithmic}) and mathematical finance (see, e.g., \cite{tankov2003financial}).

The SDE \eqref{e:SDE} is often called piecewise Ornstein Uhlenbeck (OU) process and has been widely used to model queueing networks (see, e.g.,  \cite{Arapostathis2019Uniform,
	Arapostathis2019Ergodicity,Braverman2017Stein2,
	Dieker2013Positive}). It is often hard to find steady states for queueing networks directly, whereas the invariant measure of the limiting SDE can be used as an approximation in the appropriate scaling regimes. Therefore, it is highly desirable to figure out the ergodic measure of SDE \eqref{e:SDE} and design an algorithm to generate random variables with this distribution.

Throughout this paper, we assume that the matrices $ \bfM $, $ \bmGamma$ and the vector $ \bfv$ satisfy the following assumption:
\begin{assumption}\label{assump:M}
	One of the following conditions holds:
	\begin{itemize}
		\item [(i)] $ \bfM  \bfv\geq  \bmGamma  \bfv \gneqq {\bf 0}$;
		\item [(ii)] $ \bfM ={\rm diag}(m_1,\cdots,m_d)$ with $m_i>0$, $i=1,\cdots,d$ and $ \bmGamma  \bfv \neq {\bf 0}$,
	\end{itemize}
	where the notations $\geq$ and $\gneqq$ are defined in \eqref{e:>>=} and \eqref{e:>Neq} respectively.
\end{assumption}

Under Assumption \ref{assump:M}, the process $( \bfX_t)_{t\geq 0}$ is exponentially ergodic with the ergodic measure $\mu$ (see, \cite[Theorem 3.5]{Arapostathis2019Ergodicity} or Lemma
\ref{lem:eeSDEs} below). We shall use the following Euler-Maruyama (EM) scheme to approximate the ergodic measure $\mu$. Let $\Lambda=(\eta_n)_{n\in \mathbb{N}}$ be the step size, for $n\in \mathbb{N}_0=\mathbb{N} \cup \{0\}$, and
\begin{equation}\label{e:EM2}
	\bfY_{t_{n+1}} = \bfY_{t_n} + \eta_{n+1} g( \bfY_{t_n})+ \bmsigma  {\Delta} \bfZ_{\eta_{n+1}}, \quad   \bfY_0 = \bfX_0 =   \bfx,
\end{equation}
where $t_n= \sum_{i=1}^n \eta_i$, $t_0=0$ and $\Delta\bfZ_{\eta_{n+1}}=\bfZ_{t_{n+1}}-\bfZ_{t_{n}}$. Then $( \bfY_{t_n})_{n\in \mathbb{N}_0}$ is a non-homogeneous Markov chain.

\subsection{Related works}

{\it  Diffusion approximations in stochastic networks.}
The piecewise Ornstein--Uhlenbeck diffusion in \eqref{e:SDE} is the scaling limit of a large family of stochastic networks and queueing systems.
In the Halfin-Whitt regime, such diffusions are established for multiclass $GI/Ph/n$ and $GI/Ph/n+GI$ queues in \cite{DHT2010,PR2000}, while their ergodic properties are studied in \cite{Dieker2013Positive} and the related numerical analysis of the stationary distribution are studied in \cite{DH2013,DH2012}.
In \cite{Arapostathis2019Ergodicity,pang2010queueing}, the $\alpha$-stable (controlled) limiting diffusions are established for
$G/M/n(+M)$ and multiclass $G/M/n(+M)$ queues, respectively, when the arrival processes are heavy-tailed  (such as regularly varying interarrival times) and converge to an $\alpha$-stable motion or independent  $\alpha$-stable motions, and the ergodicity properties of such processes are studied in \cite{Arapostathis2019Uniform,Arapostathis2019Ergodicity}. In fact, for (multiclass) $G/Ph/n$ and $G/Ph/n+GI$ queues, when the arrival processes are heavy-tailed and converge to an $\alpha$-stable motion, one also obtains the limiting process as given in  \eqref{e:SDE} (the proof follows from slight modifications of  \cite{DHT2010,pang2010queueing,PR2000}).

Due to the non-Markovian and network complexities, the stationary distribution and ergodic properties of these queueing networks/systems are impossible to be derived directly. Therefore, their (jump) diffusion approximations become useful since the related techniques of studying invariant measures and ergodic properties have been well developed.
We refer the reader to the recent studies on the relevant ergodic properties in these networks and queueing systems in  \cite{Aghajani2020The,Arapostathis2021On,
	Arapostathis2019Uniform,
	Arapostathis2019Ergodicity,Braverman2017Stein2,
	Braverman2017Stein,Dieker2013Positive,
	Gurvich2014Diffusion}.
In this paper, we focus on the approximations of the ergodic measures of the limiting diffusions. This is in the same spirit with the Euler schemes for constrained diffusions as scaling limits of stochastic networks in \cite{BCR1} and for piecewise Ornstein-Uhlenbeck diffusions as scaling limits of $GI/Ph/n+GI$ queues and parallel server networks in \cite{Jin2024An}.

{\it EM scheme.} Recently, Fang et al. \cite{Fang2019Multivariate} used the EM scheme to approximate ergodic measures for diffusions with help of Malliavin calculus and Stein's method and obtained a convergence rate in Wasserstein-1 distance. There are also some references concerning on the computation of ergodic measures for SDEs, see \cite{Benaim2017Ergodicity,Panloup2008Computation,
	Panloup2008Recursive} and references therein, but most of them are asymptotic type. Strong convergence for EM scheme can be found in \cite{Kuhn2019Strong} and references therein.  Jin et al. \cite{Jin2024An} used the EM scheme to approximate the ergodic measure for the piecewise OU process driven by Brownian motion and considered the related asymptotics for the original process.

{\it Central limit theorem (CLT) and Moderate deviation principle (MDP) for empirical measures.} It follows from \cite[Theorem 3.5]{Arapostathis2019Ergodicity} (or see Lemma  \ref{lem:eeSDEs} below) that the  SDE \eqref{e:SDE} is exponentially ergodic with ergodic measure $\mu$, which implies that the Birkhoff ergodic theorem holds true for the empirical measure	$\mathcal{E}_T$ (see the definition in \eqref{e:ETX} below) of the process $(\bfX_t)_{t\geq 0}$ (i.e., $\lim_{T\to \infty}\mathcal{E}_T=\mu$ almost surely); see, for instance, \cite{da1996ergodicity}. It is natural to consider its CLT and MDP with respect to $\mathcal{E}_T$. The recent work \cite{Jin2024An} considered the CLT and MDP for the SDE in \eqref{e:SDE} with $\alpha=2$, it is natural to ask whether the CLT and MDP hold true for the case $\alpha \in (1,2)$.

{\it Sinkhorn-Knopp algorithm.} The Sinkhorn-Knopp algorithm is widely used in approximating the Wasserstein type distances, see \cite{Chizat2020Faster,Cuturi2013Sinkhorn,Khamlicha2023Optimal,Liao2022Fast,
	Luise2018Differential,Xie2020A} and references therein.
Cuturi \cite{Cuturi2013Sinkhorn} introduced the {Sinkhorn distance} to study the optimal transportation problem between two probability measures, and used this distance with the entropy constraint to develop algorithms with a  cheap cost. Such Sinkhorn distance with constraints can be computed by the Sinkhorn and Knopp's fixed point iteration. In addition, Luise et al. \cite{Luise2018Differential} gave regularity for the regularized Sinkhorn distance and a convergence rate for the approximation in Wasserstein distance. We refer the  reader to \cite{Knight2008The} and references therein for more details about the Sinkhorn-Knopp algorithm.

\subsection{Our contributions and main results}

Jin et al. \cite{Jin2024An} develops an EM scheme  for approximating the invariant measure of piecewise Ornstein-Uhlenbeck diffusions driven by Brownian motion.
However, the numerical scheme in \cite{Jin2024An} has a constant step size, which leads to a bias in the approximation, see \cite[Theorem 1]{Jin2024An}. In order to remove this bias, we will use a numerical scheme with decreasing step sizes, which converges to the ergodic measure exactly as the time tends to infinity.

The very recent works \cite{Chen2023Approximation,Chen2023Approximation2} provide an approximation to the ergodic measure of SDEs driven by $\alpha$-stable process and get convergence rates in Wasserstein-1 distance. However,  they have to assume that the drift $g\in \mathcal{C}^2(\R^d,\R^d)$ satisfies the following dissipation condition: $\langle \bfx - \bfy, g(\bfx)-g(\bfy) \rangle \leq -c|\bfx-\bfy|^2+K$ with some positive constants $c$ and $K$.
Most of the works on approximating the ergodic measures of SDEs need a similar assumption, see \cite{Huang2018The,Kuhn2019Strong,
	Pamen2017Strong,Panloup2008Computation,
	Panloup2008Recursive,Protter1997The}. However, the piecewise linear drift $g$ in \eqref{e:g} does not satisfy such a dissipation condition, see \eqref{e:NonDis} below, hence the known approximation methods do not work. Borrowing the mollification technique in \cite{Jin2024An} and the intermediate matrix ${\bf Q}$ proposed in \cite{Arapostathis2019Ergodicity}, we are able to establish a desirable  dissipation property for the drift $g$, and surprisingly find that a straightforward strong approximation works.

Note that the SDE \eqref{e:SDE} driven by $\alpha$-stable noise does not have finite second moment when $\alpha \in (1,2)$, one may suspect whether the CLT and MDP will hold for the empirical measure $\mathcal{E}_T$. Our second main result confirms the both asymptotic properties are true. In their proofs, Stein's method with nonlocal operator plays an important role,  the related analysis is much more delicate than the one in the case $\alpha=2$ and has its own independent interest.

Besides proving the theorems about the approximation error and the CLT and MDP, we use the recently developed the Sinkhorn-Knopp algorithm to verify the convergence of our numerical scheme in Wasserstein-1 distance.  This algorithm works very well for both the light-tailed models and the $\alpha$-stable heavy tailed ones.

\vskip 3mm

Before presenting our main results, we impose some assumptions on the step size sequence $\Lambda=(\eta_n)_{n\in \mathbb{N}}$. Conveniently, for some constant $\beta \in (0,1]$, define $\omega$ as below:
\begin{equation}\label{e:Q}
	\omega =  \lim_{k \to \infty}  \left\{ \frac{\eta_k^{\beta}-\eta^{\beta}_{k+1}}
	{\eta_{k+1}^{1+\beta}} \right\}.
\end{equation}

\begin{assumption}\label{assump-2}
	The positive decreasing step size sequence $\Lambda=(\eta_n)_{n\in \mathbb{N}}$ satisfies the following conditions:
	\begin{itemize}
		\item [(i)] $\eta_1<1$, $\sum_{n=1}^{\infty} \eta_n =\infty$ and  $\sum_{n=1}^{\infty} \eta_n^2 < \infty$;
		\item [(ii)] The above $\omega$ satisfies $\omega<\alpha \beta \theta$, where the constant $\theta$ is defined  in \eqref{e:lambda} below.
	\end{itemize}
\end{assumption}

Our first main result is the following theorem about the EM scheme, which provides a nonasymptotic estimate for the error between the ergodic measures of the SDE and its EM scheme, whose proofs will be given in Section \ref{sec:ProofThm1} below.

\begin{theorem}\label{thm:XY-dis}
	For $\alpha\in(1,2]$, under Assumptions \ref{assump:M} and \ref{assump-2}, let $\mu$ be the ergodic measure for the process $( \bfX_t)_{t\geq 0}$ of SDE \eqref{e:SDE}, and let the EM scheme  $( \bfY_{t_n})_{n\in \mathbb{N}_0}$ be given by \eqref{e:EM2}. Then there exists a positive constant $C$ independent of $(\eta_n)_{n\in \mathbb{N}}$ such that for any $n\in \mathbb{N}$ and $\bfx\in \R^d$, one has
	\begin{equation*}
		\mathcal{W}_1 (\mu, \mathcal{L}( \bfY_{t_n}^{\bfx})) \leq C(1+|\bfx|)\eta_n^{ {1}/{\alpha} },
	\end{equation*}
	where $ \bfY_{t_n}^{\bfx}$ means {the marginal random variable} of process $( \bfY_{t_n})_{n\in\mathbb{N}_0}$ at time $t_n$ with $ \bfY_0=\bfx$  and $ \mathcal{W}_1 $ is the Wasserstein-1 distance defined in \eqref{e:W1} below.
\end{theorem}

Our second set of main results includes the CLT and MDP for the long-term behavior of $(\bfX_t)_{t\geq 0}$ with $\alpha\in (1,2)$.
Recall that the CLT and MDP for the process $(\bfX_t)_{t\geq 0}$ with $\alpha=2$ has been shown in \cite[Theorems 2 and 3]{Jin2024An}.
For any $\bfx\in \R^d$ and $T>0$, the empirical measure $\mathcal{E}_t^{\bfx}$ of $(\bfX_t)_{t\geq 0}$ is defined by
\begin{equation}\label{e:ETX}
	\mathcal{E}_T^{\bfx}(A) = \frac{1}{T} \int_0^T \delta_{\bfX_s^{\bfx}}(A) \dif s, \quad A \in \mathcal{B}(\R^d),
\end{equation}
where $\mathcal{B}(\R^d)$ is the collection of Borel sets on $\R^d$, and $\delta_{\bfy}(\cdot)$ is the delta measure defined as $\delta_{\bfy}(A)=1$ if $y\in A$ and $\delta_{\bfy}(A)=0$ otherwise. It is easy to check that for any bounded Borel measurable function $h\in \mathcal{B}_b(\R^d,\R)$,
\begin{equation*}
	\mathcal{E}_T^{\bfx}(h) = \frac{1}{T} \int_0^T h(\bfX_s^{\bfx}) \dif s.
\end{equation*}

\begin{theorem}[CLT]\label{thm:SDECLT}
	For $\alpha \in (1,2)$ and under Assumption \ref{assump:M}, let $\mu$ be the ergodic measure for the process $(\bfX_t)_{t\geq 0}$. For any $h\in \mathcal{B}_b(\R^d,\R)$ and $\bfx\in \R^d$, the term $\sqrt{t}[\mathcal{E}_t^\bfx(h)-\mu(h)]$ converges weakly to Gaussian distribution
	$\mathcal{N}(0,\mathcal{V}(f_h))$ as $t\to \infty$, where $\mathcal{V}(f_h)$ is given as following:
	\begin{equation}\label{e:Vf}
		\mathcal{V}(f_h) =
		\left\{
		\begin{aligned}
			&\int_{\R^d} \int_{\R^d \setminus \{\bf 0\}} \big[f_h(\bfx+{\bf \bmsigma z})-f_h(\bfx) \big]^2 \nu(\dif \bfz) \mu(\dif \bfx), \quad \text{ rotationally  symmetric  noise,} \\
			&\sum_{i=1}^d \int_{\R^d} \int_{\R \setminus \{0\}} \big[f_h(\bfx+ \bmsigma_i z_i)-f_h(\bfx) \big]^2 \nu_1(\dif z_i) \mu(\dif \bfx), \quad \text{ cylindrical stable noise,}
		\end{aligned}
		\right.
	\end{equation}
	where $\bmsigma = [\bmsigma_1, \dotsc, \bmsigma_d]$ and $f_h$ is the solution to the Stein's equation \eqref{e:steinbb} below.
\end{theorem}

\begin{theorem}[MDP]\label{thm:SDEMDP}
	For $\alpha \in (1,2)$ and under Assumption \ref{assump:M}, let $\mu$ be the ergodic measure for the process $(\bfX_t)_{t\geq 0}$.
	For any $h\in \mathcal{B}_b(\R^d,\R)$, $\bfx\in \R^d$, and any measurable set $A\subseteq \R$,
	\begin{equation*}
		\begin{aligned}
			-\inf_{z\in \mathring{A}} \frac{z^2}{2\mathcal{V}(f_h)}
			&\leq \liminf_{t\to \infty} \frac{1}{a_t^2} \log
			\PP\left( \frac{\sqrt{t}}{a_t}[\mathcal{E}_t^{\bfx}(h)-\mu(h)]\in A  \right) \\
			&\leq \limsup_{t\to \infty} \frac{1}{a_t^2} \log
			\PP\left( \frac{\sqrt{t}}{a_t}[\mathcal{E}_t^{\bfx}(h)-\mu(h)]\in A  \right)
			\ \leq \  -\inf_{z\in \bar{A}} \frac{z^2}{2\mathcal{V}(f_h)},
		\end{aligned}
	\end{equation*}
	where $\bar{A}$ and $\mathring{A}$ are the closure and interior of the set $A$, respectively; the function $a_t$ satisfies $a_t\to \infty$ and $t^{-1/2}a_t\to 0$ as $t\to \infty$, and $\mathcal{V}(f_h)$ is defined in \eqref{e:Vf}.
\end{theorem}

\subsection{Notations}
Throughout this paper, the following notations will be used. We use normal font for scalars (e.g. $a, A, \dotsc $) and boldface for vectors and matrices (e.g. $ \bfx, \bfy, \bfA, \bfB, \dotsc $). For any $a\in \R$, define $a^+ =\max\{a,0\}$.
For $x,y \in \R$, $x \vee y = \max\{x,y\}, x \wedge y =\min\{x,y\}$. And
we use $\R^d$ (and $\R^d_{+}$), $d\geq 1$, to denote real-valued $d$-dimensional (non-negative) vectors, and write $\R$ (and $\R_+$) for $d=1$. Additionally, $\R^d_0=\R^d \setminus \{ \bf{0} \}$.

For any vector $\bfu = (u_1,\dotsc, u_d)^{\prime}$ and $\bfv=(v_1, \dotsc, v_d)^{\prime}$ in $\R^d$, let $\langle \bfu, \bfv \rangle = \bfu^\prime \bfv = \sum_{i=1}^d u_i v_i $ be the inner product. This induces the Euclidean norm $|\bfu| = \langle \bfu,\bfu \rangle^{1/2}$. Besides, denote $\bfu ./ \bfv = (u_1/v_1, \dotsc, u_d/v_d)^{\prime}$. On the other hand, for any vector $\bfz = (z_1, \dotsc, z_d)^{\prime} \in \R^d$, let $\| \bfz\|_{\max} = \max\{ |z_i|,i=1,\dotsc,d\}$ be the maximum absolute value of complements of $\bfz$.  We write
\begin{equation}
	\label{e:>>=}
	\bfz\geq {\bf 0}\  (\bfz>{\bf 0})
\end{equation}
to indicate that all the components of $\bfz$ are non-negative (positive), and
\begin{equation} \label{e:>Neq}
	\bfz \gneqq {\bf 0}
\end{equation}
stands for
$\bfz\geq {\bf 0}$ and $\bfz \neq {\bf 0}$, and analogously for a matrix in $\R^{d\times d}$.

For any matrices $\bfA=(A_{ij})_{d\times d}$ and $\bfB=(B_{ij})_{d\times d}$ in $\R^{d \times d}$, denote their Hilbert Schmidt inner product as $\langle \bfA,\bfB\rangle_{\rm HS}=\sum_{i,j=1}^d A_{ij} B_{ij}$ and denote the Hilbert Schmidt norm as $\| \bfA\|_{\rm HS}=\sqrt{ \langle \bfA,\bfA\rangle_{\rm HS} }$. The operator norm of matrix $\bfA$ is defined as $\|\bfA\|_{\rm op}=\sup_{\bfx\in \R^d, |\bfx|=1} |\bfA \bfx|$, and we have relation $\| \bfA \|_{\rm op} \leqslant \|\bfA\|_{\rm HS} \leqslant \sqrt{d} \|\bfA\|_{\rm op}$.
Besides, let $\lambda_{\min}(\bfA)$ and $\lambda_{\max}(\bfA)$ be the minimum and maximum eigenvalues for symmetric matrix $\bfA$ respectively.
On the other hand, a matrix $ \bfM \in\R^{d\times d}$ is called a M-matrix (see, \cite[Definition 1.2, p. 133]{Berman1994Nonnegative}) if its off-diagonal entries are less than or equal to zero and it can be expressed as $\bfM = s\mathbf{I}-\bfN$
for a constant $s>0$ and a matrix $\bfN\in\R^{d\times d}$ with its entries being all nonnegative and $\rho(\bfN) \leqslant s$, where $\bfI$ and $\rho(\bfN)$ denote the identity matrix in $\R^{d\times d}$ and spectral radius of $\bfN$ respectively.

For any two probability measures $\nu_1$ and $\nu_2$ defined on $\bbR^d$, the total variation distance $\| \nu_1-\nu_2\|_{\rm TV}$ is defined as below \cite[p. 10]{Villani2009Optimal}):
\begin{equation*}
	\| \nu_1-\nu_2\|_{\rm TV} =
	2\inf_{\Pi \in \mathcal{C}(\nu_1,\nu_2)}
	\big\{ \Pi(\{\bfX \neq \bfY \}): \mathcal{L}(\bfX)=\nu_1,\mathcal{L}(\bfY)=\nu_2 \big\},
\end{equation*}
where the infimum runs over the set $\mathcal{C}(\nu_1,\nu_2)$ of all probability measures $\Pi$ on with marginals $\nu_1$ and $\nu_2$,  and $\mathcal{L}(\bfX)$ is the law of random vector $\bfX$. An alternative characterisation of the total variation distance is as the dual norm to the supremum norm on the space of bounded Borel measurable functions (see, \cite[p. 57]{hairer2009introduction}):
\begin{equation*}
	\| \nu_1 - \nu_2 \|_{\rm TV} =
	\sup \left\{\int_{\R^d} \psi(\bfx) \nu_1(\dif \bfx)-\int_{\R^d} \psi(\bfx) \nu_2(\dif \bfx): \sup_{\bfx\in \R^d}|\psi(\bfx)|\leq 1 \right\}.
\end{equation*}
Given a weight function $V: \R^d\to [1,+\infty)$, define a weighted supremum norm on measurable functions by
\begin{equation*}
	\| \psi \|_{\rm V} = \sup_{\bfx\in \R^d} \frac{|\psi(\bfx)|}{V(\bfx)},
\end{equation*}
as well as the weighted total variation distance is defined by (see, \cite[(5.18)]{hairer2009introduction})
\begin{equation*} 
	\| \nu_1-\nu_2 \|_{{\rm TV},V} = \sup
	\left\{\int_{\R^d} \psi(\bfx) \nu_1(\dif \bfx)-
	\int_{\R^d} \psi(\bfx) \nu_2(\dif \bfx): \| \psi \|_{\rm V} \leq 1 \right\}.
\end{equation*}

The $L^p$-Wasserstein distance $\mathcal{W}_p(\nu_1, \nu_2)$ with $p\geq 1$ between $\nu_1$ and $\nu_2$  is defined as
\begin{equation*}
	\mathcal{W}_p(\nu_1,\nu_2) = \inf_{ \Pi \in \mathcal{C}(\nu_1,\nu_2) } \left( \int_{\R^{2d}}  |\bfx-\bfy|^p  \dif \Pi(\bfx,\bfy)  \right)^{{1}/{p}}.
\end{equation*}
Using the Kantorovich duality (see, \cite[Theorem 5.10]{Villani2009Optimal}), the Wasserstein-1 distance $\mathcal{W}_1(\nu_1, \nu_2) $ also can be defined as
\begin{equation}\label{e:W1}
	\mathcal{W}_1 (\nu_1,\nu_2) = \sup_{ h\in {\rm Lip}(1)} \left(\int_{\R^d} h(\bfx) \nu_1(\dif \bfx)-  \int_{\R^d} h(\bfx) \nu_2(\dif \bfx) \right),
\end{equation}
where ${\rm Lip}(1)$ is the set of Lipschitz functions with Lipschitz constant $1$, that is, ${\rm Lip}(1) = \{ f: |f(\bfx)-f(\bfy)| \leq |\bfx-\bfy|,\ \forall\ \bfx, \bfy \in \R^d \}.$ For a probability measure $\pi(\dif \bfx)$ and an integrable function $f$ on $\bbR^d$, we write $\pi(f)=\int_{\R^d} f( \bfx)\pi(\dif \bfx)$.  For a given probability measure $\nu$, we define $L^2(\nu)$ as the Hilbert space induced by $\nu$ with inner product
\begin{equation*}
	\langle f_1, f_2\rangle_{\nu}
	= \int f_1(\bfx)  f_2(\bfx)  \nu(\dif \bfx)
	\quad \text{for} \ f_1, f_2 \in L^2(\nu).
\end{equation*}

The notation $\mathcal{N}({\bf a},{\bf A})$  denotes the Gaussian distribution with mean ${\bf a}\in \R^d$ and covariance matrix ${\bf A}\in \R^{d\times d}$. A sequence of random variables $\{Y_n,n \geq 1\}$ is said to converge weakly or converge in distribution to a limit $Y_{\infty}$, denoted by $Y_n \Rightarrow Y_{\infty}$, if $\lim_{n\to \infty} \E f(Y_n)=\E f(Y_{\infty})$ for all bounded continuous function $f$. In addition, $Y_n \xrightarrow{p} Y_{\infty}$ means convergence in probability, namely $\lim_{n\to \infty} \PP(|Y_n-Y_{\infty}|>\delta)>0$ for all $\delta>0$.

We denote by $\mcl C(\R^d)$ all the continuous functions from $\R^d$ to $\R$, by
$\mathcal{C}_b(\R^d)$ all the bounded continuous functions from $\R^d$ to $\R$. For $f \in \mathcal{C}_b(\R^d)$, denote $\| f \|_{\infty} = \sup_{\bfx\in \R^d}|f(\bfx)|$. $\mathcal{C}^k(\R^d)$ denotes the collection of $k$-th continuously differentiable functions with integers $k\geq 1$. For any $f \in \mathcal{C}^2(\R^d)$, let $\nabla f$ and $\nabla^2 f $ be the gradient and Hessian matrix for $f$, and denote $\| \nabla f \|_{\infty} = \sup_{ \bfx\in \R^d} |\nabla f( \bfx)|$ and $\| \nabla^2 f \|_{{\rm HS},\infty}= \sup_{ \bfx\in \R^d} \| \nabla^2 f( \bfx) \|_{{\rm HS}}$. For any $f \in \mathcal{C}^1(\R)$, we often denote by $\dot{f}$ its derivative. We further introduce a function space $\mathcal{C}_{lin}(\R^d)$ as below:
\begin{equation*}
	\mathcal{C}_{lin}(\R^d)
	= \left\{f \in \mathcal C(\R^d): \sup_{\bfx \in \R^d} \frac{|f(\bfx)|}{1+|\bfx|}<\infty\right\}.
\end{equation*}

\section{Key ingredients for the proofs of the main results}  \label{sec:Ingredients}
In this section, we provide the main ingredients for the proofs of Theorems \ref{thm:XY-dis}, \ref{thm:SDECLT} and \ref{thm:SDEMDP}, which include
the mollification of $g$, Jacobi flow for mollified processes, the preliminaries for the semigroup theory with respect to Markov processes, and the ergodicities in the total variation type distance.

\subsection{Mollification of $g$ and properties of the mollifier $g_\e$}
Because $g( \bfx)$ is not differentiable, we need to consider its mollification in \cite{Jin2024An} as the following: for any $0<\e<1$,
\begin{equation*}
	g_{\e} ( \bfx) = \bmell- \left[ \bfM  \bfx+ ( \bmGamma- \bfM ) \bfv\rho_{\e}({\rm \bf e}_d^{\prime}  \bfx) \right], \quad \bfx \in \R^d,
\end{equation*}
where the function $\rho_{\e}(y): \bbR \to \bbR$ is defined as
\begin{equation*}
	\rho_{\e}(y)  =y \I_{\{y>\e\} } + \left(\frac{3\e}{16} - \frac{ 1 }{ 16 \e^3} y^4 + \frac{ 3 }{8 \e} y^2 + \frac{1}{2} y\right)\I_{\{|y|\leq \e \}}.
\end{equation*}
It is easy to check that $\lim_{\e \to 0}\rho_{\e}(y)=y^+$, $\forall \ y\in \R$, $\rho_{\e}\in \mathcal{C}^2(\R)$, and $\|\dot{\rho_{\e}}\|_{\infty} \leq 1$ for all $\e\in (0,1)$. Consequently, $g_{\varepsilon} \in \mathcal{C}^2(\mathbb{R}^d) $ and
\begin{equation}\label{e:nabge}
	\nabla g_{\e}( \bfx)
	=-\bfM-\dot{\rho_{\e}}({\rm \bf e}_d^{\prime} \bfx)(\bmGamma-\bfM)\bfv{\rm \bf e}_{d}^{\prime}.
\end{equation}
We shall consider the following auxiliary SDE:
\begin{equation}\label{e:SDEe}
	\dif \bfX^{\e}_t = g_\e (\bfX^{\e}_t) \dif t +\bmsigma \dif  \bfZ_t,
	\quad \bfX_0^{\e} = \bfX_0  =   \bfx \in \R^d.
\end{equation}
The EM scheme for the SDE \eqref{e:SDEe} with decreasing step-size $\Lambda=(\eta_n)_{n\in \mathbb{N}}$ is as the following:
\begin{equation}\label{e:EM}
	\bfY^{\e}_{t_{n+1}} = \bfY^{\e}_{t_n} + \eta_{n+1} g_{\e}(\bfY^{\e}_{t_n})+ \bmsigma   \Delta \bfZ_{\eta_{n+1}}, \quad  \bfY_0^{\e}  =   \bfx,
\end{equation}
where $t_n= \sum_{i=1}^n \eta_i$ with $t_0=0$, and $\Delta\bfZ_{\eta_{n+1}}=\bfZ_{t_{n+1}}-\bfZ_{t_{n}}$. Then $(\bfY_{t_n}^{\e})_{n\in \mathbb{N}_0}$ is a non-homogeneous Markov chain.
For any $0\leq s \leq t $, denote $\bfX_{s,t}^{ \bfx}$ as the value of process $( \bfX_t)_{t\geq 0}$ at the time $t$ with $\bfX_s= \bfx$.  Similarly for $\bfX_{s,t}^{\e, \bfx}$, $ \bfY_{s,t}^{ \bfx}$, and $ \bfY_{s,t}^{\e, \bfx}$. When $s = 0$, we will omit the subscript $s$ and write $\bfX_{t}^{\e, \bfx}=\bfX_{0,t}^{\e, \bfx}$, $\bfY_{t}^{ \bfx}=\bfY_{0,t}^{ \bfx}$, and $\bfY_{t}^{\e, \bfx}=\bfY_{0,t}^{\e, \bfx}$ if there is no ambiguity.

It is easy to show that the coefficient $g_{\e}( \bfx)$ is a global Lipschitz function, and its Lipschitz constant is independent of $\e$, that is, there exists some constant $C>0$ independent of $\varepsilon$, such that
\[
|g_{\e}( \bfx) - g_{\e}(\bfy) | \leqslant  C| \bfx-\bfy|, \quad \forall \ \bfx, \bfy \in \R^d.
\]
Additionally, there exists some positive constant $C$ independent of $\e$ such that
\begin{equation}\label{e:ge-g}
	|g_{\e}(\bfx)-g(\bfx)| \leq C \e \I_{ \{ |{\rm \bf e}_d^{\prime}  \bfx| \leq \e \} }, \quad \forall \ \bfx\in \R^d.
\end{equation}
Moreover, $\lim_{\e \rightarrow 0}g_{\e}( \bfx)=g( \bfx)$ for all $ \bfx\in \R^d$.
The coefficients $g_{\e}( \bfx)$ and $g( \bfx)$ do not satisfy a {dissipation} condition, that is, there exist some $ \bfx,\bfy\in \R^d$ such that for some constant $c>0$,
\begin{equation}\label{e:NonDis}
	\sup_{| \bfx-\bfy| \to \infty} \frac{
		\langle \bfx-\bfy, g_{\e}(\bfx)-g_{\e}(\bfy) \rangle}{|\bfx-\bfy|^2} \geq c,
	\sup_{|\bfx-\bfy| \to \infty} \frac{
		\langle \bfx-\bfy, g(\bfx)-g(\bfy) \rangle}{|\bfx-\bfy|^2} \geq c.
\end{equation}

The following lemma is from \cite{Arapostathis2019Ergodicity}.
\begin{lemma}{{\rm (}\cite[Theorem 3.5]{Arapostathis2019Ergodicity}{\rm )} }\label{lem:Q}
	Under Assumption \ref{assump:M}, there exists a positive definite matrix $ \bfQ $ such that matrices $ \bfM ^{\prime} \bfQ + \bfQ  \bfM $ and $( \bfM ^{\prime}-{\rm \bf e}_d \bfv^{\prime}( \bfM ^{\prime}- \bmGamma)) \bfQ
	+ \bfQ ( \bfM -( \bfM - \bmGamma) \bfv{\rm \bf e}_d^{\prime})$ are both positive definite.
\end{lemma}

For the positive definite matrix $ \bfQ $ in Lemma \ref{lem:Q}, denote
\begin{equation*}
	\begin{aligned}
		\lambda_1 &= \lambda_{\min}(\bfM ^{\prime} \bfQ + \bfQ  \bfM ), \\
		\lambda_2 &= \lambda_{\min}([\bfM ^{\prime}-{\rm \bf e}_d \bfv^{\prime}( \bfM ^{\prime}- \bmGamma)] \bfQ
		+ \bfQ [ \bfM -( \bfM - \bmGamma) \bfv{\rm \bf e}_d^{\prime}]),
	\end{aligned}
\end{equation*}
and
\begin{equation}\label{e:lambda}
	\lambda = \lambda_1 \wedge	\lambda_2, \quad \theta  = \frac{1}{2} \lambda \lambda^{-2}_{\max}( \bfQ ).
\end{equation}

Under Assumption \ref{assump:M}, it follows from Lemma \ref{lem:Q} that we can get the following kind of dissipation condition and its proof is in Appendix \ref{app_A1}.
\begin{lemma}\label{lem:QNge}
	Under Assumption \ref{assump:M} and  for the positive definite matrix $ \bfQ $ in Lemma \ref{lem:Q}, for any $ \bfx\in \R^d$ and $\e \in (0,1)$, the matrices
	\begin{equation*}
		[-\nabla g_{\e}( \bfx)]^{\prime}  \bfQ  +  \bfQ  [-\nabla g_{\e}( \bfx)]
	\end{equation*}
	are positive definite. We also have
	\begin{equation*}
		-\lambda = \sup_{ \bfx\in \R^d} \left\{ \lambda_{\max} \big([\nabla g_{\e}( \bfx)]^{\prime}  \bfQ  +  \bfQ  [\nabla g_{\e}( \bfx)] \big)  \right\}
		<  0,
	\end{equation*}
	where $\lambda$ is in \eqref{e:lambda}. In addition, for all $ \bfx,\bfy \in \R^d$,
	\begin{equation}\label{e:dis-con}
		\langle  \bfQ (\bfx-\bfy),g_{\e}(\bfx)-g_{\e}(\bfy)   \rangle
		\leq -\frac{\lambda}{2}|\bfx-\bfy|^2.
	\end{equation}
	Furthermore, for all $\bfx\in \R^d$,
	\begin{equation*}
		\langle g(\bfx),\bfQ \bfx \rangle \leq -\frac{\lambda}{2}|\bfx|^2+|\bmell| \lambda_{\max}(\bfQ)|\bfx|.
	\end{equation*}	
\end{lemma}

\subsection{Jacobi flow for the mollified process $( \bfX_t^{\e, \bfx})_{t\geq 0}$.}
Recall that for $\alpha\in(1,2]$,
\begin{equation*}
	\dif \bfX^{\e}_t = g_\e (\bfX^{\e}_t) \dif t +\bmsigma \dif \bfZ_t,
	\quad  \bfX_0^{\e}= \bfX_0= \bfx \in \R^d.
\end{equation*}
It can be solved as
\begin{equation*}
	\bfX^{\e, \bfx}_t =  \bfx+ \int_0^t g_{\e} (\bfX^{\e, \bfx}_s) \dif s +\bmsigma \bfZ_t.
\end{equation*}
For any $\bfu\in \R^d$, the corresponding Jacobi flow is defined as
\begin{equation*}
	\nabla_\bfu  \bfX_t^{\e, \bfx}
	= \lim_{\delta \to 0}\frac{\bfX^{\e, \bfx+\delta \bfu}_t - \bfX^{\e, \bfx}_t}{\delta},
\end{equation*}
which satisfies $\nabla_\bfu \bfX_0^{\e, \bfx}=\bfu$ and
\begin{equation*}
	\dif \nabla_\bfu  \bfX_t^{\e, \bfx}
	=  \nabla g_{\e}(\bfX^{\e, \bfx}_t) \nabla_\bfu  \bfX_t^{\e, \bfx} \dif t.
\end{equation*}
Then, we have the following estimate which is proven in Appendix \ref{app_A2}.
\begin{lemma}\label{lem:JF-est}
	For any $\bfu\in \R^d$, the following inequality holds
	\begin{equation*}
		|\nabla_\bfu  \bfX_t^{\e, \bfx}|^2
		\leq \frac{\lambda_{\max}( \bfQ )}{\lambda_{\min}( \bfQ )}
		\exp\left\{-\frac{\lambda}{\lambda_{\max}( \bfQ )}t \right\}
		|\bfu|^2,
	\end{equation*}
	where the matrix $\bfQ$ is in Lemma \ref{lem:Q} and $\lambda$ is in \eqref{e:lambda}.
\end{lemma}

Define
\begin{equation*}
	\bfJ_{s,t}^{\e,\bfx}:=\exp \left( \int_s^t \nabla g_{\e}(\bfX_{r}^{\e,\bfx}) \dif r \right), \quad 0 \leq s \leq t<\infty.
\end{equation*}
It is called the Jacobian between $s$ and $t$. For notational simplicity, denote $\bfJ^{\e,\bfx}_{t}=\bfJ^{\e,\bfx}_{0,t}$. Then we have
\begin{equation*}
	\nabla_{\bfu} \bfX^{\e,\bfx}_{t}
	= \bfJ_{t}^{\e,\bfx}\bfu.
\end{equation*}

Define
\begin{equation*}
	\nabla \widetilde{g}(\bfx)
	= -\bfM + \I_{\{{\bf e}_d^{\prime} \bfx>0\}} (\bfM-\bmGamma) {\bf v} {\bf e}_d^{\prime}.
\end{equation*}
It is easy to see that $\lim_{\e \rightarrow 0} \nabla g_{\e}(\bfx)=\nabla \widetilde{g}(\bfx)$ for all ${\bf e}_d^{\prime} \bfx \ne 0$.
Because $g(\bfx)$ is not differentiable for ${\bf e}_d^{\prime} \bfx=0$, it is necessary for us to define the above $\widetilde{g}(\bfx)$ which takes the same value as $\nabla g(\bfx)$ for ${\bf e}_d^{\prime} \bfx \neq 0$ and has a definition on ${\bf e}_d^{\prime} \bfx=0$. Define
\begin{equation*}
	\bfJ^{\bfx}_{s,t}
	= \exp\left(\int_{s}^{t} \nabla \widetilde{g}(\bfX^{\bfx}_{r}) \dif r\right), \quad \bfx \in \R^{d}, \ 0 \leq s \leq t<\infty.
\end{equation*}
We also denote $\bfJ^{\bfx}_{t}=\bfJ^{\bfx}_{0,t}$. For any $\bfu\in \R^d$, define $\nabla_{\bfu} \bfX_t^{\bfx}=\bfJ^{\bfx}_{t} \bfu.$

Then we have the following lemma which is proven in Appendix \ref{app_A3}.
\begin{lemma}\label{l:XeCon}
	For any $\bfx \in \R^{d}$, as $\e \rightarrow 0$, the following holds:
	\begin{equation*}
	\|\bfJ_{s,t}^{\e,\bfx}-\bfJ^{\bfx}_{s,t}\|_{{\rm op}}\ {\longrightarrow} \ 0, \quad  0 \leq s \leq t<\infty, \  \text{ a.s. }
	\end{equation*}
\end{lemma}

Now we give estimates for  $\|\bfJ^{\e,\bfx}_{s,t}\|_{ {\rm op} }$ and $\|\bfJ^{\bfx}_{s,t}\|_{ {\rm op} }$. By \eqref{e:nabge}, we can easily see that
\begin{equation}  \label{e:Ngeop}
	\|\nabla g_{\e}(\bfx)\|_{ {\rm op} } \ \le \ \|\bfM\|_{ {\rm op} }+\|(\bfM-\bmGamma) {\bf v} {\bf e}_d^{\prime} \|_{ {\rm op} } \ = \ C_{{\rm op}},
\end{equation}
from which we obtain
\begin{equation*}
	\|\bfJ^{\e,\bfx}_{s,t}\|_{ {\rm op} }
	\leq  \exp\left(\int_s^{t} \|\nabla g_{\e}(\bfX^{\e,\bfx}_{r})\|_{ {\rm op} } \dif r\right)
	\ \leq \  \rme^{C_{{\rm op}}(t-s)}.
\end{equation*}
So for all $0 \le s \le t<\infty$, we have
\begin{equation*}
	\|\bfJ^{\bfx}_{s,t}\|_{ {\rm op} }, \ \|\bfJ^{\e,\bfx}_{s,t}\|_{ {\rm op} } \  \le \  \rme^{C_{{\rm op}}(t-s)},
\end{equation*}
where the bound of $\|\bfJ^{\bfx}_{s,t}\|_{ {\rm op} }$ comes from the same argument since the bound in \eqref{e:Ngeop} also holds for $\nabla \widetilde{g}(\bfx)$. Observe that the above estimates immediately imply that for any $\bfu\in\R^d$,
\begin{equation*}
	|\nabla_{\bfu} \bfX_t^{\bfx}|, \ |\nabla_{\bfu} \bfX_t^{\e,\bfx}|
	\leq   \rme^{C_{{\rm op}}t}|\bfu|.
\end{equation*}

\subsection{Semigroups, generators and their properties}
In order to study the approximation in $\mathcal W_1$ distance, we consider the semigroups associated with the SDEs \eqref{e:SDE} and \eqref{e:SDEe} in $\mathcal{C}_{lin}(\R^d)$ defined as:
for a $f \in \mathcal{C}_{lin}(\R^d)$,
\begin{equation*}
	\begin{aligned}
		\opP_{s,t}f( \bfx)& = \E [f( \bfX_{t})| \bfX_{s}= \bfx], \quad 0\leq s \leq t,\\
		\opP^{\e}_{s,t}f( \bfx)& = \E [f( \bfX^{\e}_{t})| \bfX^{\e}_{s}= \bfx], \quad 0\leq s \leq t.
	\end{aligned}
\end{equation*}
By the same argument as in \cite{deng2023wasserstein}, we can show that $\{\opP_{s,t}\}_{0 \le s \le t<\infty}$ is a well defined semigroup in $\mathcal{C}_{lin}(\R^d)$.
It is easy to see that both $\opP_{s,t}$ and $\opP^{\e}_{s,t}$ are time homogeneous semigroup, i.e., $\opP_{s,t}=\opP_{0,t-s}$ and $\opP^{\e}_{s,t}=\opP^{\e}_{0,t-s}$ for any $0 \le s \le t$.
For simplicity, as $s=0$, we drop the script '$s$' and write $\opP_t=\opP_{0,t}$ and $\opP^{\e}_{t}=\opP^{\e}_{0,t}$.
The infinitesimal generators of $\opP_t$ and $\opP^\e_t$ are defined as
$$\mathcal{A}^{\alpha}f( \bfx)=\lim_{t \rightarrow 0} \frac{\opP_t f(\bfx)-f(\bf x)}{t}, \quad f \in \mathcal{D}(\mathcal{A}^{\alpha}),$$
$$\mathcal{A}_\e^{\alpha}f( \bfx)=\lim_{t \rightarrow 0} \frac{\opP^\e_t f(\bfx)-f(\bf x)}{t}, \quad f \in \mathcal{D}(\mathcal{A}_\e^{\alpha}),$$
where $\mathcal{D}(\mathcal{A}^{\alpha})$ and $\mathcal{D}(\mathcal{A}_\e^{\alpha})$ are the domains of $\mathcal{A}^{\alpha}$ and $\mathcal{A}_\e^{\alpha}$ respectively, which are both subsets of $\mathcal{C}_{lin}(\R^d)$. The forms of $\mathcal{D}(\mathcal{A}^{\alpha})$ and $\mathcal{D}(\mathcal{A}_
\e^{\alpha})$ are not necessary to be exactly specified in this paper. Let us now give the explicit forms of the infinitesimal generators $\mathcal{A}^{\alpha}$ and $\mathcal{A}_\e^{\alpha}$:

(1) When $\alpha=2$, i.e., the driven noise is Brownian motion, we have
$$\mathcal{A}_{\e}^{2}f( \bfx)  =  \langle g_{\e}( \bfx), \nabla f( \bfx)\rangle
+ \frac{1}{2} \langle \nabla^2 f( \bfx) , \bmsigma \bmsigma^{\prime} \rangle_{\rm HS},$$
$$\mathcal{A}^{2}f( \bfx)  =  \langle g( \bfx), \nabla f( \bfx)\rangle
+ \frac{1}{2} \langle \nabla^2 f( \bfx) , \bmsigma \bmsigma^{\prime} \rangle_{\rm HS}.$$

(2) When $\alpha \in (1,2)$, i.e., the driven noise is $\alpha$-stable noise, the generators $\mathcal{A}^\alpha$ can be written as the following according to the type of stable noise:
\begin{itemize}
	\item[(2.i)] When $\bfZ_t=(Z_t^1, \dotsc, Z_t^d)^{\prime}$ is the cylindrical stable noise,
	\begin{equation}\label{e:Ae}
		\mathcal{A}_{\e}^{\alpha}f( \bfx)  =  \langle g_{\e}( \bfx), \nabla f( \bfx)\rangle
		+ \mathcal{G}_{\rm iid}^{\alpha} f(\bfx),
	\end{equation}
	\begin{equation}\label{e:Aiid}
		\mathcal{A}^{\alpha}f( \bfx)  =  \langle g( \bfx), \nabla f( \bfx)\rangle  + \mathcal{G}_{\rm iid}^{\alpha} f(\bfx),
	\end{equation}
	where
	\begin{equation*}
		\mathcal{G}_{\rm iid}^{\alpha} f( \bfx) =  \sum_{i=1}^d \int_{\R_0} \big[ f(\bfx +\bmsigma_i z_i) - f( \bfx) -  z_i \langle \bmsigma_i ,  \nabla f( \bfx) \rangle  \I_{\{|z_i| \leq 1\}} \big] \frac{c_{\alpha}}{|z_i|^{1+\alpha}} \dif z_i,
	\end{equation*}
	where $\bmsigma_i$ is the $i$-th column of matrix $\bmsigma$.
	\item[(2.ii)] When $\bfZ_t$ is the $d$-dimensional rotationally symmetric $\alpha$ stable process, we have
	\begin{equation}\label{e:Ae rotational}
		\mathcal{A}_{\e}^{\alpha}f( \bfx)  =  \langle g_{\e}( \bfx), \nabla f( \bfx)\rangle
		+ \mathcal{G}_{\rm sym}^{\alpha} f(\bfx),
	\end{equation}
	\begin{equation}\label{e:Asym}
		\mathcal{A}^{\alpha}f( \bfx)  =  \langle g( \bfx), \nabla f( \bfx)\rangle
		+ \mathcal{G}_{\rm sym}^{\alpha} f(\bfx),
	\end{equation}
	where
	\[
	\mathcal{G}_{\rm sym}^{\alpha} f(\bfx)  =  \int_{\R^d_0} \big[ f(\bfx+\bmsigma \bfz) - f(\bfx) - \langle \bmsigma \bfz , \nabla f(\bfx) \rangle \I_{\{|\bfz| \leq 1\}} \big] \frac{C_{d,\alpha}}{|\bfz|^{d+\alpha}} \dif \bfz .
	\]
\end{itemize}

\subsection{Ergodicity in a weighted total variation distance and moment estimates}

Note that \cite[Theorem 3.5]{Arapostathis2019Ergodicity} has shown the exponentially ergodicity for the processes $(\bfX_t)_{t\geq 0}$ under the weighted total variation distance,  but the weight function therein is so large that we cannot directly use this ergodicity result to prove the regularities for our Stein's equation. To solve this problem, we establish an exponential ergodicity for these two processes in another weighted total variation distance, whose weight function has the form as \eqref{e:tiV} below. This is proven in Appendix \ref{app_A4}.

For the small enough constant $ \kappa >0$ satisfying $16 \kappa <1$, and the positive definite matrix $\bfQ$ in Lemma \ref{lem:Q}, we define a function $\widetilde{V}(\bfx): \R^d \to \R$ as
\begin{equation}\label{e:tiV}
	\widetilde{V}(\bfx) = \big( 1+\langle \bfx, \bfQ \bfx  \rangle \big)^{ \kappa }.
\end{equation}
\begin{lemma}\label{lem:eeSDEs}
	Under Assumption \ref{assump:M}, and for $\alpha\in (1,2]$, the processes $(\bfX_t)_{t\geq 0}$ and $(\bfX_t^{\e})_{t\geq 0}$ in \eqref{e:SDE} and \eqref{e:SDEe} are both exponential ergodicity under the weighted total variation distance $\| \cdot \|_{\rm TV,\widetilde{V}}$  with ergodic measures $\mu$ and $\mu_{\e}$ respectively. More precisely, there exist some positive constants $C$ and $c$, both independent of $\e$, such that
	\begin{equation*}
		\|\mathcal{L}(\bfX^{\bfx}_t) - \mu \|_{\rm TV,\widetilde{V}}
		\ \leq \ C\widetilde{V}(\bfx) \rme^{-ct}
		\quad \text{ and } \quad
		\|\mathcal{L}(\bfX^{\e,\bfx}_t) - \mu_{\e} \|_{\rm TV,\widetilde{V}}
		\ \leq \ C\widetilde{V}(\bfx)\rme^{-ct}.
	\end{equation*}
	In addition, it holds that for all $\bfx \in \bbR^d$,
	\begin{equation}\label{e:metlV}
		\E \widetilde{V}(\bfX_t^{\bfx})+
		\E \widetilde{V}(\bfX_t^{\e,\bfx}) \leq C(1+|\bfx|^{2 \kappa }).
	\end{equation}
\end{lemma}

The moment estimates are given below and the corresponding proofs are in Appendix \ref{app_A5}.
\begin{lemma}\label{lem:moment-est}
	For $\alpha \in (1,2]$ and $\varepsilon\in (0,1)$, under Assumptions \ref{assump:M} and \ref{assump-2}, there exists some positive constant $C$, independent of $\e$, such that for all $ \bfx\in \R^d$,
	\begin{align}
\label{e:Xmon} \E | \bfX_t^{\e, \bfx}| +  \E | \bfX_t^{\bfx}|
&\leq  C (1+| \bfx|), \quad \forall\ t\geq 0,   \\
		\label{e:X-Ymon} \E[|\bfX^{\e, \bfx}_{t_{n-1},t_n} - \bfY^{\e, \bfx}_{t_{n-1},t_n}|]
		&\leq C (1+| \bfx|)\eta_n^{ 1 + {1}/{\alpha} }, \quad \forall \ n\in \mathbb{N},    \\
		\label{e:Ymon} \E |\bfY^{\e, \bfx}_{t_n}|
		&\leq C (1+| \bfx|), \quad  \forall\ n\in \mathbb{N}_0,  \\
		\label{e:XeCon-1} \lim_{\e \to 0} \E|\bfX^{\e,\bfx}_{t}-\bfX^{\bfx}_{t}|
		&= 0,  \qquad  \forall t \geq 0.
	\end{align}
	In particular, we have $\mu_{\e}(|\bfx|) \leq C$,  where $\mu_{\e}$ is the ergodic measure of the process $(\bfX_t^{\e})_{t\geq 0}$.
\end{lemma}

\section{Proof of Theorem \ref{thm:XY-dis}} \label{sec:ProofThm1}

\subsection{The strategy of the proof}
Let us give a brief strategy for the proof of the Theorem \ref{thm:XY-dis}.
Under Assumption \ref{assump:M}, the processes  $( \bfX_t^{\e})_{t\geq 0}$ in \eqref{e:SDEe} are  exponential ergodicity (see, Lemma \ref{lem:eeSDEs} above) with ergodic measures  $\mu_{\e}$. By the triangle inequality,
\begin{equation}\label{e:tri}
	\mathcal{W}_1 \big( \mu,\mathcal{L}( \bfY_{t_n}^{ \bfx}) \big)
	\leq
	\mathcal{W}_1 (\mu,\mu_{\e}) +
	\mathcal{W}_1 \big(\mu_{\e},\mathcal{L}( \bfY_{t_n}^{\e, \bfx}) \big)+
	\mathcal{W}_1 \big(\mathcal{L}( \bfY_{t_n}^{\e, \bfx}),
	\mathcal{L}( \bfY_{t_n}^{ \bfx}) \big).
\end{equation}
To estimate $ \mathcal{W}_1 \big(\mu,\mathcal{L}( \bfY_{t_n}^{ \bfx}) \big)$, it is sufficient to give estimates for these three Wasserstein-1 distances on the right hand of \eqref{e:tri}. $\mathcal{W}_1 (\mu,\mu_{\e})$
is bounded in Lemma \ref{pro:mu-mue} by a generator comparison argument and solving a Stein's equation, while $ \mathcal{W}_1 \big(\mathcal{L}( \bfY_{t_n}^{\e, \bfx}), \mathcal{L}( \bfY_{t_n}^{ \bfx})\big)$ is the distance between the two Markov chains and is bounded in Lemma \ref{pro:2EM}. Finally, $\mathcal{W}_1 (\mu_{\e},\mathcal{L}\big( \bfY_{t_n}^{\e, \bfx}) \big)$ is bounded by a generalized Lindeberg principle developed in \cite{Pages2020Unajusted}.

The following contraction property of $(\bfX^{\e}_t)_{t\geq 0}$ will be proved in Appendix \ref{app_B1}.
\begin{lemma}\label{pro:erg-Xe}
	For $\alpha\in(1,2]$, under Assumption \ref{assump:M}, for any $\e \in (0,1)$,  the processes $( \bfX_t^{\e})_{t\geq 0}$ in \eqref{e:SDEe} satisfy that for all $\bfx,\bfy \in \R^d$ and $t\geq 0$,
	\begin{equation*}
		\mathcal{W}_1 \big( \mathcal{L}( \bfX_t^{\e,\bfx}),\mathcal{L}( \bfX_t^{\e,\bfy}) \big) \leq \frac{\lambda_{\max}( \bfQ )}{\lambda_{\min}( \bfQ )} |\bfx-\bfy| \rme^{-\theta t},
	\end{equation*}
	where the constant $\theta$ is given in \eqref{e:lambda} and the positive definite matrix $ \bfQ $ is given in Lemma \ref{lem:Q}.
\end{lemma}

Next, we give bounds for $\mathcal{W}_1 \big( \mu,\mu_{\e} \big)$ and $ \mathcal{W}_1 (\mcl{L}(\bfY^{\e,\bfx}_{t_{n}}),
\mcl{L}(\bfY^{\bfx}_{t_{n}}))$ in the following lemmas, and their proofs are postponed in Appendixes \ref{app_B2} and \ref{app_B3}.
\begin{lemma}\label{pro:mu-mue}
	For $\alpha\in(1,2]$ and $\varepsilon \in (0,1)$, under Assumption \ref{assump:M}, there exists some positive constant $C$ independent of $\e$ such that
	\begin{equation*}
		\mathcal{W}_1 (\mu,\mu_{\e})
		\leq C\e,
	\end{equation*}
	where $\mu$ and $\mu_{\e}$ are the ergodic measures of the processes $(\bfX_t)_{t\geqslant 0}$ and $(\bfX_{t}^{\e})_{t\geqslant 0}$ respectively.
\end{lemma}

\begin{lemma}\label{pro:2EM}
	For $\alpha\in(1,2]$ and $\varepsilon\in (0,1)$, under Assumptions \ref{assump:M} and \ref{assump-2}, for the two EM schemes $(\bfY^{\e, \bfx}_{t_{n}})_{n\in \mathbb{N}_0}$ and $(\bfY^{ \bfx}_{t_{n}})_{n\in \mathbb{N}_0}$ defined in \eqref{e:EM2} and \eqref{e:EM} respectively, there exists some positive constant $C$ independent of $\e$ such that for all $ \bfx\in \R^d$ and $n\in \mathbb{N}_0$,
	\begin{equation*}
		\mathcal{W}_1 \big( \mcl{L}(\bfY^{\e, \bfx}_{t_{n}}),
		\mcl{L}(\bfY^{ \bfx}_{t_{n}}) \big) \leq {C \e^{1/2}. }
	\end{equation*}
\end{lemma}

The following Lemma is similar to that of \cite[Lemma B.1]{Pages2020Unajusted} which is crucial to obtain the convergence rate, and we shall give the proof in Appendix \ref{app_B4}.
\begin{lemma}\label{lem:series}
	For $\alpha\in(1,2]$, given that the  step size $(\eta_n)_{n\in \mathbb{N}}$ satisfy Assumption \ref{assump-2},  there exists some positive constant $C$ such that
	\begin{equation*}
		\sum_{i=1}^{n}  \rme^{-\theta(t_n-t_i)} \eta_i^{1+ {1}/{\alpha}}
		\leq C\eta_n^{{1}/{\alpha}},
	\end{equation*}
	where the constant $\theta$ is in \eqref{e:lambda}.
\end{lemma}

\subsection{Proof of Theorem \ref{thm:XY-dis}}
The semigroups on $\mathcal{C}_{lin}(\R^d)$ for processes $( \bfY^{ \bfx}_{t_n})_{n\in \mathbb{N}_0}$, and $( \bfY^{\e, \bfx}_{t_n})_{n\in \mathbb{N}_0}$ can be defined as below: for any $f\in \mathcal{C}_{lin}(\R^d)$,
\begin{equation}\label{e:Pe}
	\begin{aligned}
		\opQ_{t_i,t_j}f( \bfx) &= \E [f( \bfY_{t_j})| \bfY_{t_i}= \bfx], \quad 0\leq t_i \leq t_j,  \\
		\opQ^{\e}_{t_i,t_j}f( \bfx) &= \E [f( \bfY^{\e}_{t_j})| \bfY^{\e}_{t_i}= \bfx], \quad 0\leq t_i \leq t_j.
	\end{aligned}
\end{equation}
\begin{proof}[Proof of Theorem \ref{thm:XY-dis}]	
	It follows from \eqref{e:tri} that
	\begin{equation*}
		\mathcal{W}_1 \big( \mu,\mathcal{L}( \bfY_{t_n}^{ \bfx}) \big)
		\leq
		\mathcal{W}_1 (\mu,\mu_{\e}) +
		\mathcal{W}_1 \big( \mu_{\e},\mathcal{L}( \bfY_{t_n}^{\e, \bfx}) \big)+
		\mathcal{W}_1 \big( \mathcal{L}( \bfY_{t_n}^{\e, \bfx}),
		\mathcal{L}( \bfY_{t_n}^{ \bfx}) \big).
	\end{equation*}
	The bounds for $ \mathcal{W}_1 (\mu,\mu_{\e})$ and $ \mathcal{W}_1 \big(\mathcal{L}( \bfY_{t_n}^{\e, \bfx}),	\mathcal{L}( \bfY_{t_n}^{ \bfx})\big)$ are given in Lemmas \ref{pro:mu-mue} and \ref{pro:2EM} respectively. Thus, it suffices to give {the bound for}  $ \mathcal{W}_1 \big(\mu_{\e},\mathcal{L}( \bfY_{t_n}^{\e, \bfx}) \big)$.
	
	By the generalized Lindeberg principle in \cite{Pages2020Unajusted}, it follows from \eqref{e:W1} that
	\begin{equation*}
		\begin{aligned}
			\mathcal{W}_1 (\mathcal{L}( \bfX^{\e, \bfx}_{t_n}), \mathcal{L}( \bfY^{\e, \bfx}_{t_n}))
			&= \sup_{h\in {\rm Lip}(1)} \left\{  \opP^{\e}_{t_n} h( \bfx) -  \opQ^{\e}_{t_n} h( \bfx) \right\} \\
			&= \sup_{h\in {\rm Lip}(1)}
			\left\{
			\sum_{i=1}^{n}  \opQ^{\e}_{0,t_{i-1}} \circ (  \opP^{\e}_{t_{i-1},t_{i}} -  \opQ^{\e}_{t_{i-1},t_{i}}  ) \circ  \opP^{\e}_{t_{i},t_n} h( \bfx)
			\right\}.
		\end{aligned}
	\end{equation*}
	
	We firstly give the estimate for $( \opP^{\e}_{t_{i-1},t_{i}} -  \opQ^{\e}_{t_{i-1},t_{i}})\circ  \opP^{\e}_{t_i,t_n} h( \bfx)$ for $i=1,2,\cdots,n$. Actually, for $\alpha \in (1,2]$, we have that
	\begin{equation*}
		\begin{aligned}
			|( \opP^{\e}_{t_{i-1},t_{i}}-  \opQ^{\e}_{t_{i-1},t_{i}} )\circ  \opP^{\e}_{t_i,t_n} h( \bfx)|
			&= | \opP^{\e}_{t_i,t_n} h( \bfX_{t_{i-1},t_i}^{\e, \bfx} ) -  \opP^{\e}_{t_i,t_n} h( \bfY_{t_{i-1},t_i}^{\e, \bfx} )|  \\
			&\leq C  \rme^{-\theta(t_n-t_i)} \E| \bfX_{t_{i-1},t_i}^{\e, \bfx}-  \bfY_{t_{i-1},t_i}^{\e, \bfx}| \\
			& \leq C  \rme^{-\theta(t_n-t_i)} \eta_i^{1+ {1}/{\alpha}}(1+| \bfx|),
		\end{aligned}
	\end{equation*}
	where the first inequality is due to Lemma  \ref{pro:erg-Xe} and the last is due to Lemma \ref{lem:moment-est}. Combining this with Lemmas \ref{lem:moment-est} and \ref{lem:series}, we can obtain that for $\alpha\in (1,2]$,
	\begin{equation*}
		| \opP^{\e}_{t_n} h( \bfx) -  \opQ^{\e}_{t_n} h( \bfx)|
		\leq C \sum_{i=1}^{n}   \rme^{-\theta(t_n-t_i)} \eta_i^{1+ {1}/{\alpha}} \E\left[1+| \bfY^{\e, \bfx}_{t_{i-1}}| \right]
		\leq C(1+| \bfx|)\eta_n^{ {1}/{\alpha}},
	\end{equation*}
	which immediately implies that
	\begin{equation*}
		\mathcal{W}_1 \big(\mathcal{L}( \bfX^{\e, \bfx}_{t_n}), \mathcal{L}( \bfY^{\e, \bfx}_{t_n})\big)
		\leq  C(1+| \bfx|)\eta_n^{ {1}/{\alpha}}.
	\end{equation*}
	Moreover, combining this with Lemmas  \ref{lem:moment-est} and \ref{pro:erg-Xe}, we know
	\begin{equation*}
		\begin{aligned}
			\mathcal{W}_1 \big(\mu_{\e}, \mathcal{L}( \bfY^{\e, \bfx}_{t_n}) \big)
			&\leq
			\mathcal{W}_1 \big(\mu_{\e},
			\mathcal{L}( \bfX^{\e, \bfx}_{t_n}) \big)
			+
			\mathcal{W}_1 \big(\mathcal{L}( \bfX^{\e, \bfx}_{t_n}), \mathcal{L}( \bfY^{\e, \bfx}_{t_n}) \big) \\
			&\leq C(1+| \bfx|) \rme^{-\theta t_n}+ C(1+| \bfx|)\eta_n^{ {1}/{\alpha}}
			\leq C(1+| \bfx|)\eta_n^{ {1}/{\alpha}}.
		\end{aligned}
	\end{equation*}
	This, together with Lemmas \ref{pro:mu-mue} and \ref{pro:2EM}, implies that for $\alpha\in (1,2]$,
	\begin{equation*}
		\begin{aligned}
			\mathcal{W}_1 \big(\mu,\mathcal{L}( \bfY_{t_n}^{ \bfx})\big)
			&\leq C\e+
			C(1+| \bfx|)\eta_n^{ {1}/{\alpha}}
			+C\e^{ {1}/{2}}
			\leq
			C(1+| \bfx|)\eta_n^{ {1}/{\alpha}},
		\end{aligned}
	\end{equation*}
	where the last inequality holds by choosing $\e\leqslant\eta_n^{{2}/{\alpha}}$. The proof is complete.
\end{proof}

\section{ Proofs of Theorems \ref{thm:SDECLT} and \ref{thm:SDEMDP}} \label{sec:SDECLTMDP}
We split this section into two subsections, one for the proof of Theorem \ref{thm:SDECLT} and the other for that of Theorem \ref{thm:SDEMDP}.

To investigate the CLT and MDP for the process $(\bfX_t)_{t\geq 0}$, we consider the following first kind of Stein's equation: for any $h\in \mathcal{B}_b(\R^d,\R)$,
\begin{equation}\label{e:steinbb}
	\mathcal{A}^{\alpha} f_h(\bfx)
	= h(\bfx)-\mu(h),
\end{equation}
where $\mathcal{A}^{\alpha}$ and $\mu$ are the infinitesimal generator and the ergodic measure for the process $(\bfX_t)_{t\geq 0}$.
Then, we can get the following regularities for the solution $f_h$ of Eq. \eqref{e:steinbb}.
\begin{lemma}\label{pro:steinbb}
	Let $h\in \mathcal{B}_b(\R^d,\R)$ and $f_h$ be the solution of Eq. \eqref{e:steinbb}. Then it holds
	\begin{equation*}
|f_h(\bfx)| \leq C(1+|\bfx|^{2 \kappa })
\quad \text{ and } \quad
|\nabla f_h(\bfx)| \leq C(1+|\bfx|^{2 \kappa }),
	\end{equation*}
	where the constant $ \kappa $ satisfies  $16 \kappa <1$.
\end{lemma}

We will prove this lemma in Appendix \ref{app_C}. With the help of this lemma, we can show the CLT and MDP for the process $(\bfX_t)_{t\geq 0}$ with $\alpha$ stable noise $(\bfZ_t)_{t\geq 0}$. Recall that the CLT and MDP for the process $(\bfX_t)_{t\geq 0}$ with $\alpha=2$ have been shown in \cite[Theorems 2 and 3]{Jin2024An}.

\subsection{Proof of Theorem
	\ref{thm:SDECLT}} We  prove Theorem \ref{thm:SDECLT} by It\^{o}'s formula and martingale CLT. The verification of the conditions in martingale CLT is highly non-trivial.
\begin{proof}[Proof of Theorem \ref{thm:SDECLT}]
	
	We shall split the proof according to the different $\alpha$ stable noises: one is the  $d$-dimensional rotationally symmetric $\alpha$ stable process, and the other is the cylindrical $\alpha$-stable noise.
	
	{\bf Case 1. For the rotationally symmetric stable noise.} For $(\bfX_t)_{t\geq 0} $ in SDE \eqref{e:SDE} with $\bfX_0=\bfx$, by using It\^{o}'s formula to $f_h$ which is the solution of  Stein's equation \eqref{e:steinbb}, we have
	\begin{equation*}
		\begin{aligned}
			f_h(\bfX_t^{\bfx})-f_h(\bfx)
			&= \int_0^t \mathcal{A}^{\alpha} f_h(\bfX_s^{\bfx})\dif s
			+\int_0^t \int_{\R_0^d} \left[f_h(\bfX_{s-}^{\bfx}+ \bmsigma \bfz) -f_h(\bfX_{s-}^{\bfx}) \right] \widetilde{N}(\dif s, \dif \bfz) \nonumber \\
			&= \int_0^t [h(\bfX_s^{\bfx})-\mu(h)] \dif s
			+\int_0^t \int_{\R_0^d} \left[f_h(\bfX_{s-}^{\bfx}+ \bmsigma \bfz)
			-f_h(\bfX_{s-}^\bfx) \right] \widetilde{N}(\dif s, \dif \bfz),
		\end{aligned}
	\end{equation*}
	where the infinitesimal operator $\mathcal{A}^{\alpha}$ is defined in \eqref{e:Asym}.
	It implies that
	\begin{equation*}
		\begin{aligned}
			&\mathrel{\phantom{=}}
			\sqrt{t} \left[\frac{1}{t}\int_0^t \delta_{\bfX_s^{\bfx}}(h)\dif s-\mu(h) \right]
			\ = \ \frac{1}{\sqrt{t}}\int_0^t [ h(\bfX_s^{\bfx}) -\mu(h)]\dif s  \\
			&= \frac{1}{\sqrt{t}} [ f_h(\bfX_t^\bfx)-f_h(\bfx)]-\frac{1}{\sqrt{t}} \int_0^t \int_{\R_0^d} \left[f_h(\bfX_{s-}^\bfx+\bmsigma \bfz)-f_h(\bfX_{s-}^\bfx) \right] \widetilde{N}(\dif s, \dif \bfz).
		\end{aligned}
	\end{equation*}
	It follows from $\E f_h(\bfX_t^\bfx)\leq C(1+|\bfx|^{2 \kappa })$ in \eqref{e:metlV} and Lemma  \ref{pro:steinbb}  that
	\begin{equation*}
		\E \left| \frac{1}{\sqrt{t}}
		\left[f_h(\bfX_t^{\bfx})-f_h(\bfx)
		\right] \right|
		\to 0 ,  \quad \text{ as}\ t \to \infty.
	\end{equation*}		
	It follows from Lemma  \ref{pro:steinbb} that for some $C>0$
	\begin{equation*}
		\begin{aligned}
			&\mathrel{\phantom{=}}
			\int_{\R_0^d} \left[f_h(\bfx+\bmsigma \bfz)
			-f_h(\bfx) \right]^2 \nu(\dif \bfz) \\
			&= \int_{0<|\bfz|\leq 1} \left| \int_0^1 \langle \nabla f_h(\bfx+r\bmsigma \bfz),\bmsigma \bfz \rangle \dif r  \right|^2 \nu(\dif \bfz)
			+\int_{|\bfz|>1} \left[f_h(\bfx+ \bmsigma \bfz)-f_h(\bfx) \right]^2 \nu(\dif \bfz) \\
			&\leqslant C \int_{0<|\bfz|\leq 1} (1+|\bfx|^{2 \kappa }+|\bfz|^{2 \kappa })^2 |\bfz|^2 \nu(\dif \bfz)
			+\int_{|\bfz|>1} (1+|\bfx|^{2 \kappa }+|\bfz|^{2 \kappa })^2 \nu(\dif \bfz) ,
		\end{aligned}
	\end{equation*}
	which implies that
	\begin{equation}\label{e:Vfest}
		\begin{aligned}
			&\mathrel{\phantom{=}}  \mathcal{V}(f_h)
			:= \int_{\R^d} \int_{\R_0^d} \left[f_h(\bfx+\bmsigma \bfz)
			-f_h(\bfx) \right]^2 \nu(\dif \bfz) \mu(\dif \bfx)  \\
			&\leq C\int_{\R^d} \left[\int_{0<|\bfz|\leq 1} (1+|\bfx|^{2 \kappa }+|\bfz|^{2 \kappa })^2 |\bfz|^2 \nu(\dif \bfz)
			+\int_{|\bfz|>1} (1+|\bfx|^{2 \kappa }+|\bfz|^{2 \kappa })^2 \nu(\dif \bfz) \right] \mu(\dif \bfx) \\
			&\leq  C\mu(|\bfx|^{4 \kappa })
			\ < \ \infty,
		\end{aligned}
	\end{equation}
	where the last two inequalities hold from the facts that $4 \kappa <1$ and $\int_{\R_0^d}(1\wedge |\bfz|^2) \nu(\dif \bfz)<\infty$.		
	
	To make notations simple, we denote
	\begin{equation*}
		\begin{aligned}
			U_i &= \int_{i-1}^i \int_{\R_0^d} \left[f_h(\bfX_{s-}^\bfx+\bmsigma \bfz)-f_h(\bfX_{s-}^\bfx) \right] \widetilde{N}(\dif s, \dif \bfz) \quad \text{for}\ i=1,2,\cdots,\lfloor t \rfloor, \\
			U_{\lfloor t \rfloor+1} &=  \int_{\lfloor t \rfloor}^t \int_{\R_0^d} \left[f_h(\bfX_{s-}^\bfx+\bmsigma \bfz)-f_h(\bfX_{s-}^\bfx) \right] \widetilde{N}(\dif s, \dif \bfz).
		\end{aligned}
	\end{equation*}
	Then, we have
	\begin{equation*}
		\frac{1}{\sqrt{t}} \int_0^t \int_{\R_0^d} \left[f_h(\bfX_{s-}^{\bfx}+\bmsigma \bfz)-f_h(\bfX_{s-}^{\bfx}) \right] \widetilde{N}(\dif s, \dif \bfz) = \frac{1}{\sqrt{t}} \sum_{i=1}^{\lfloor t \rfloor + 1} U_i.
	\end{equation*}
	We know that $U_i$'s are martingale differences and $\E U_i^2 <\infty$, which can be obtained similarly with inequality \eqref{e:Vfest} for all $i=1,2,\cdots,\lfloor t \rfloor+1$. We claim that
	\begin{gather}
		\lim_{t \to \infty} \E\left( \max_{1\leq i \leq \lfloor t \rfloor +1 } \frac{1}{\mathcal{V}(f_h) t} |U_i|^2 \right) \ = \ 0, \label{e:CLT1} \\
		\lim_{t \to \infty} \E \left| \sum_{i=1}^{\lfloor t \rfloor +1 } \frac{1}{\mathcal{V}(f_h) t } |U_i|^2 -1 \right|^2 \ = \ 0. \label{e:CLT2}
	\end{gather}		
	We know \eqref{e:CLT1} and \eqref{e:CLT2} imply
	\begin{equation*}
		\E\left( \max_{1\leq i \leq \lfloor t \rfloor +1 } \frac{1}{ \sqrt{\mathcal{V}(f_h)t}}|U_i|\right) \to 0
		\quad \text{ and } \quad
		\sum_{i=1}^{\lfloor t \rfloor +1 } \frac{1}{\mathcal{V}(f_h) t } |U_i|^2 \stackrel{p}{\longrightarrow} 1
	\end{equation*}
	as $t$ tends to infinity.  Applying the martingale CLT in \cite[Theorem 2]{sethuraman2002martingale}, due to \cite{mcleish1974dependent}, one has
	\begin{equation*}
		\frac{1}{\sqrt{\mathcal{V}(f_h) t}}\sum_{i=1}^{\lfloor t \rfloor+1} U_i
		\Rightarrow \mathcal{N} (0,1) \quad \text{as } \ t \to \infty.
	\end{equation*}		
	Then, we know
	\begin{equation*}
		\sqrt{t} \left[ \frac{1}{t}\int_0^t \delta_{\bfX_s^{\bfx}}(h)\dif s-\mu(h) \right]
		\ \Rightarrow \ \mathcal{N}(0,\mathcal{V}(f_h))
		\quad \text{ as  }\ t \to \infty.
	\end{equation*}
	
	It remains to show \eqref{e:CLT1} and \eqref{e:CLT2}. For \eqref{e:CLT1}, one has
	\begin{equation*}
		\begin{aligned}
			\E\left( \max_{1\leq i \leq \lfloor t \rfloor +1} |U_i|^2 \right)
			&= \E\left( \max_{1\leq i \leq \lfloor t \rfloor +1} (|U_i|^2  \I_{ \{ |U_i|^2\leq \sqrt{t} \} } + |U_i|^2  \I_{ \{|U_i|^2>\sqrt{t}  \}})  \right)  \nonumber \\
			&\leq \E\left( \max_{1\leq i \leq \lfloor t \rfloor+1} |U_i|^2  \I_{ \{ |U_i|^2\leq \sqrt{t} \} } \right)
			+ \E\left( \max_{1\leq i \leq \lfloor t \rfloor+1} |U_i|^2  \I_{ \{|U_i|^2> \sqrt{t}  \}}  \right) \nonumber \\
			&\leq \sqrt{t} +(\lfloor t \rfloor+1) \max_{1\leq i \leq \lfloor t \rfloor+1}
			\E\left[|U_i|^2  \I_{\{|U_i|^2>\sqrt{t}\}}
			\right] .
		\end{aligned}
	\end{equation*}
	Thus, we obtain that as $t$ tends to $\infty$
	\begin{equation*}
		\E \left( \max_{1\leq i \leq \lfloor t \rfloor+1 } \frac{1}{ \mathcal{V}(f_h) t} | U_i |^2 \right)
		\leq \frac{1}{ \mathcal{V}(f_h) t} \left(\sqrt{t}+(\lfloor t \rfloor+1) \max_{1\leq i \leq \lfloor t \rfloor +1} \E \left[|U_i|^2  \I_{\{|U_i|^2>\sqrt{t} \}}\right] \right)
		\to  0 .
	\end{equation*}
	
	It is easy to verify \eqref{e:CLT2} by proving the following equation,
	\begin{equation}\label{eq:32}
		\lim_{t \to \infty} \E \left| \sum_{i=1}^{ \lfloor t \rfloor +1 } \frac{1}{\mathcal{V}(f_h)t}|U_i|^2 -\frac{ \lfloor t \rfloor +1 }{t} \right|^2  \  =  \ 0.
	\end{equation}
	Observe that
	\begin{equation*}
		\begin{aligned}
			&\mathrel{\phantom{=}}
			\E \left| \sum_{i=1}^{ \lfloor t \rfloor +1 } \frac{1}{\mathcal{V}(f_h) t } |U_i|^2 -\frac{ \lfloor t \rfloor +1 }{t} \right|^2
			\  =  \ \E \left| \frac{1}{t} \sum_{i=1}^{ \lfloor t \rfloor +1 } \left( \frac{1}{\mathcal{V}(f_h)} |U_i|^2 - 1 \right) \right|^2 \\
			&= \frac{1}{t^2} \sum_{i=1}^{ \lfloor t \rfloor +1 }\E\left(\frac{1}{ \mathcal{V}(f_h)}|U_i|^2 - 1 \right)^2+\frac{2}{t^2}\sum_{i<j}\E\left[ \left( \frac{1}{ \mathcal{V}(f_h) } |U_i|^2 -1\right) \left(\frac{1}{ \mathcal{V}(f_h)}|U_j|^2-1\right) \right] \\
			& =: {\rm I} + {\rm II}.
		\end{aligned}
	\end{equation*}		
	By using Kunita's inequality (see, \cite[Theorem 4.4.23]{applebaum2009levy}), with similar calculations for \eqref{e:Vfest}, there exists some positive constant $C$ such that
	\begin{equation*}
		\begin{aligned}
			\E|U_i|^4
			&\leq \E \left| \int_{i-1}^i \int_{\R_0^d} \left[f_h(\bfX_{s-}^\bfx+\bmsigma \bfz)-f_h(\bfX_{s-}^\bfx) \right]^2 \nu(\dif \bfz) \dif s \right|^2  \\
			&\mathrel{\phantom{=}}
			+ \E\int_{i-1}^i \int_{\R_0^d} \left[f_h(\bfX_{s-}^\bfx+\bmsigma \bfz)-f_h(\bfX_{s-}^\bfx) \right]^4 \nu(\dif \bfz) \dif s  \\
			&\leq C\E(1+|\bfX_{i-1}^{\bfx}|^{8 \kappa })
			\ \leq \ C(1+|\bfx|^{8 \kappa }).
		\end{aligned}
	\end{equation*}
	Combining this with Lemma  \ref{pro:steinbb} and the fact $8 \kappa <1$, we know
	\begin{equation*}
		\begin{aligned}
			\E\left( \frac{1}{ \mathcal{V}(f_h)}|U_i|^2-1\right)^2
			& = \E \left[ \frac{1}{\mathcal{V}^2(f_h)}|U_i|^4 - \frac{2}{\mathcal{V}(f_h)}|U_i|^2+1 \right]
		\leq C(1+\E |\bfX_i^\bfx|^{8 \kappa }) \\
			& \leq \  C(1+|\bfx|^{8 \kappa }),
		\end{aligned}
	\end{equation*}
	which implies that
	\begin{equation*}
		{\rm I}
		= \frac{1}{t^2} \sum_{i=1}^{ \lfloor t \rfloor +1}\E\left( \frac{1}{\mathcal{V}(f_h)}|U_i|^2-1 \right)^2
		\to 0, \quad \text{as }\ t \to \infty.
	\end{equation*}		
	As for ${\rm II}$, for each $i$, it holds that for some constant $C>0$
	\begin{equation*}
		\begin{aligned}
			&\mathrel{\phantom{=}}
			\sum_{j=i+1}^{\lfloor t \rfloor +1} \E \left[ \left ( \frac{1}{ \mathcal{V}(f_h) } |U_j|^2 - 1 \right) \bigg|\mathcal{F}_i  \right] = \frac{1}{\mathcal{V}(f_h)} \sum_{j=i+1}^{\lfloor t \rfloor +1}\E\big[ (|U_j|^2 -\mathcal{V}(f_h)  ) |\mathcal{F}_i  \big] \\
			&= \frac{1}{\mathcal{V}(f_h)} \int_{i}^{\lfloor t \rfloor+1}
			\left( \E\int_{\R_0^d} \Big| f_h\big( \bfX_{s}^{\bfX_{i}^\bfx}+{\bf \sigma z} \big)-f_h\big( \bfX_{s}^{\bfX_{i}^\bfx} \big) \Big|^2 \nu(\dif \bfz)- \mathcal{V}(f_h) \right)
			\dif s \\
			&\leqslant \int_{0}^{\lfloor t \rfloor+1} (1+|\bfX_i^\bfx|^{4 \kappa })\rme^{-cs} \dif s
\leqslant C (1+|\bfX_i^\bfx|^{4 \kappa }),
		\end{aligned}
	\end{equation*}
	where the last second inequality is by Lemma  \ref{lem:eeSDEs} and the fact $8 \kappa <1$. Hence, we can obtain
	\begin{equation*}
		\begin{aligned}
			{\rm II}
			&= \frac{2}{t^2}\sum_{i=1}^{\lfloor t \rfloor } \sum_{j=i+1}^{\lfloor t \rfloor +1} \E \left\{ \left(\frac{1}{\mathcal{V}(f_h)}|U_i|^2-1\right) \E \left[ \left ( \frac{1}{ \mathcal{V}(f_h) } |U_j|^2 - 1 \right) |\mathcal{F}_i  \right]  \right\} \\
			&\leq \frac{C}{t^2}\sum_{i=1}^{\lfloor t \rfloor } \E \left\{ \Big| \frac{1}{\mathcal{V}(f_h)}|U_i|^2-1 \Big| \big(1+|\bfX_i^\bfx|^{4 \kappa }\big)  \right\} \\
			&\leq \frac{C}{t^2}\sum_{i=1}^{\lfloor t \rfloor } \left[ \E \Big|\frac{1}{\mathcal{V}(f_h)}|U_i|^2-1 \Big|^2 \right]^{{1}/{2}} \big[ 1+\E |\bfX_i^\bfx|^{8 \kappa } \big]^{{1}/{2}}
			\to 0, \quad \text{ as }\ t \to \infty.
		\end{aligned}
	\end{equation*}
	Combining the estimates for ${\rm I}$ and ${\rm II}$, \eqref{eq:32} holds.
	
	{\bf Case 2. For the  cylindrical  stable noise.} For $(\bfX_t)_{t\geq 0} $ in SDE \eqref{e:SDE} with $\bfX_0=\bfx$, by using It\^{o}'s formula for  $f_h$ which is the solution of Stein's equation \eqref{e:steinbb}, we have
	\begin{equation*}
		\begin{aligned}
			&\mathrel{\phantom{=}}
			f_h(\bfX_t^{\bfx})-f_h(\bfx)
			= \int_0^t \mathcal{A}^{\alpha} f_h(\bfX_s^{\bfx})\dif s
			+\sum_{i=1}^d \int_0^t \int_{\R_0} \left[f_h(\bfX_{s-}^{\bfx}+\bmsigma_i z_i) -f_h(\bfX_{s-}^{\bfx}) \right] \widetilde{N}(\dif s, \dif z_i) \nonumber \\
			&= \int_0^t [h(\bfX_s^{\bfx})-\mu(h)] \dif s
			+\sum_{i=1}^d \int_0^t \int_{\R_0} \left[f_h(\bfX_{s-}^{\bfx}+ \bmsigma_i z_i)-f_h(\bfX_{s-}^\bfx) \right] \widetilde{N}(\dif s, \dif z_i),
		\end{aligned}
	\end{equation*}
	where the infinitesimal operator $\mathcal{A}^{\alpha}$ here is defined in \eqref{e:Aiid}.
	It follows from the Stein's equation \eqref{e:steinbb} that
	\begin{equation*}
		\begin{aligned}
			&\mathrel{\phantom{=}}
			\sqrt{t} \left[\frac{1}{t}\int_0^t \delta_{\bfX_s^{\bfx}}(h)\dif s-\mu(h) \right] \\
			&= \frac{1}{\sqrt{t}}[f_h(\bfX_t^{\bfx})-f_h(\bfx)] +\sum_{i=1}^d \frac{1}{\sqrt{t}}\int_0^t \int_{\R_0} \left[f_h(\bfX_{s-}^{\bfx}+ \bmsigma_i z_i)-f_h(\bfX_{s-}^\bfx) \right] \widetilde{N}(\dif s, \dif z_i).
		\end{aligned}
	\end{equation*}
	With similar calculations for the {\bf Case 1}, we can get the following expression for $\mathcal{V}(f_h)$:
	\begin{equation*}
		\mathcal{V}(f_h)
		= \sum_{i=1}^d \int_{\R^d} \int_{\R_0} \left[f_h(\bfx+ \bmsigma_i z_i)-f_h(\bfx) \right]^2 \nu_1(\dif z_i) \mu(\dif \bfx).
	\end{equation*}
	The proof is complete.
\end{proof}

\subsection{Proof of Theorem \ref{thm:SDEMDP}}  We shall prove the MDP by the criterion in \cite[Theorem 2.6]{wu2001large}.
\begin{proof}[Proof of Theorem \ref{thm:SDEMDP}.]
	According to the criterion for MDP in \cite[Theorem 2.6]{wu2001large}, we need to show that the semigroup $\opP_t$ associated to the process $(\bfX_t)_{t\geqslant 0}$ is strong Feller, topological irreducible and satisfies a Lyapunov condition which is immediately implied by the second relation in \eqref{e:LyaCon}. The strong Feller property and topological irreducibility hold form the Harnack inequality and log-Harnack inequality from \cite[Theorem 2.1]{wang2014harnack} for the $d$ dimensional rotationally symmetric $\alpha$ stable noise and \cite[Theorems 1.1 and 1.2]{wang2015harnack} for the cylindrical stable noise respectively.
	
	For $h\in \mathcal{B}_b(\R^d, \R)$, it follows from Wu \cite[Theorem 2.1]{wu1995moderate} that
	$$\PP\left( \frac{1}{a_t\sqrt{t}}\int_0^t [ h(\bfX_s^\bfx) - \mu(h)  ]  \dif s \in  \,\, \cdot \,\, \right)$$
	satisfies the large deviation principle with speed $a_t^{-2}$ and rate function $I_h(z)=\frac{z^2}{2\mathcal{H}(h)}$ with
	$$\mathcal{H}(h) \ = \ 2\int_0^{\infty} \langle \opP_t h, h-\mu(h) \rangle_{\mu} \dif t,$$
	that is,
	\begin{equation*}
		\begin{aligned}
			-\inf_{z\in \mathring{A} }I_h(z)
			&\leq \liminf_{t\to\infty}\frac{1}{a_t^2}\log \mathbb{P}\left(\frac{\sqrt{t}}{a_t}  \left[\mcl E^{\bfx}_t(h)-\mu(h)\right] \in A \right)  \\
			&\leq \limsup_{t\to\infty}\frac{1}{a_t^2}\log \mathbb{P} \left( \frac{\sqrt{t}}{a_t}  \left[\mcl E^{\bfx}_t(h)-\mu(h)\right] \in A \right)
			\ \leq  \ -\inf_{z \in \bar{A}} I_h (z),
		\end{aligned}
	\end{equation*}
	where $\bar{A}$ and $\mathring{A}$ are the closure and  interior of set $A$ respectively.
	
	We claim that
	\begin{equation}\label{e:Vh}
		\mathcal{H}(h)  = \mathcal{V}(f_h),
	\end{equation}
	where $f_h$ is the solution to Stein's equation \eqref{e:steinbb}. Then the desired result holds.
	
	Now, we show that the claim \eqref{e:Vh} holds. With some calculations, one has
	\begin{equation*}
		\begin{aligned}
			\frac{1}{t}\E^{\mu} \left( \int_0^t [h(\bfX_s^\bfx) - \mu(h)] \dif s \right)^2
			&= \frac{1}{t}\E^{\mu} \left( \int_0^t \int_0^t [h(\bfX_s^\bfx)-\mu(h)][h(\bfX_u^\bfx)-\mu(h)] \dif u \dif s  \right) \\
			&= \frac{2}{t}\E^{\mu} \left( \int_0^t \int_0^u [h(\bfX_s^\bfx)-\mu(h)][h(\bfX_u^\bfx)-\mu(h)] \dif s \dif u \right) \\
			&= \frac{2}{t} \int_0^t \int_0^u \E^{\mu} \{ [h(\bfX_s^\bfx)-\mu(h)] \E \{ [h(\bfX_u^\bfx)-\mu(h)]| \bfX_s^\bfx \} \} \dif s \dif u \\
			&= \frac{2}{t} \int_0^t \int_0^u \E^{\mu} \{ h(\bfX_s^\bfx) \E \{ [h(\bfX_u^\bfx)-\mu(h)]| \bfX_s^\bfx \} \} \dif s \dif u,
		\end{aligned}
	\end{equation*}
	where the third equality holds from conditional probability and the last equality holds from
	\begin{equation*}
		\int_0^t \int_0^u \E^{\mu} \{ \mu(h) \E \{ [h(\bfX_u^\bfx)-\mu(h)]| \bfX_s^\bfx \} \} \dif s \dif u
		= \mu(h) \int_0^t \int_0^u \E^{\mu} \{  h(\bfX_u^\bfx)-\mu(h) \} \dif s \dif u
		=0.
	\end{equation*}
	Furthermore, for all $0\leq s \leq u<\infty$, one has
	\begin{equation*}
		\begin{aligned}
			\E^{\mu} \{ h(\bfX_s^\bfx) \E \{ [h(\bfX_u^\bfx)-\mu(h)]| \bfX_s^\bfx \} \}
			&= \E^{\mu} \{ h(\bfX_s^\bfx) [ \opP_{u-s}h(\bfX_s^\bfx) - \mu(h)] \} \\
			& = \int_{\R^d} h(\bfy)[ \opP_{u-s}h(\bfy)-\mu(h) ] \mu(\dif \bfy),
		\end{aligned}
	\end{equation*}
	which implies that
	\begin{equation*}
		\begin{aligned}
			&\mathrel{\phantom{=}}
			\frac{2}{t} \int_0^t \int_0^u \E^{\mu} \{ h(\bfX_s^\bfx) \E \{ [h(\bfX_u^\bfx)-\mu(h)]| \bfX_s^{\bfx} \} \} \dif s \dif u \\
			&= \frac{2}{t} \int_0^t\int_0^u \int_{\R^d} h(\bfy)[\opP_{u-s}h(\bfy)-\mu(h) ] \mu(\dif \bfy) \dif s \dif u \\
			&= \int_{\R^d} h(\bfy) \frac{2}{t} \int_0^t\int_0^u [ \opP_{u-s}h(\bfy)-\mu(h) ]  \dif s \dif u \mu(\dif \bfy) \\
			&= \int_{\R^d} h(\bfy) \frac{2}{t} \int_0^t\int_0^u[ \opP_{\tl{s}}h(\bfy)-\mu(h)]\dif \tl{s} \dif u \mu(\dif \bfy).
		\end{aligned}
	\end{equation*}
	Using the L'Hospital's rule, one has
	\begin{equation*}
		\lim_{t\to\infty}\int_{\R^d} h(\bfy) \frac{2}{t} \int_0^t\int_0^u [ \opP_{\tl{s}}h(\bfy)-\mu(h) ]  \dif \tl{s} \dif u \mu(\dif \bfy)
		= 2 \int_{\R^d} h(\bfy) \int_0^{\infty} [\opP_{\tl{s}} h(\bfy) - \mu(h)] \dif \tl{s} \mu(\dif \bfy).
	\end{equation*}
	From the expression of $\mathcal{H}(h)$, we know
	\begin{equation*}
		\mathcal{H}(h)
		= 2\int_0^{\infty} \langle \opP_t h, h-\mu(h) \rangle_{\mu} \dif t
		\ = \ 2\int_0^{\infty} \langle \opP_t h - \mu(h), h \rangle_{\mu} \dif t,
	\end{equation*}
	so that,
	\begin{equation*}
		\frac{1}{t}\E^{\mu} \left( \int_0^t [h(\bfX_s^\bfx) - \mu(h)] \dif s \right)^2
		\to \mathcal{V}(f_h) ,
		\quad \text{ as } \ t \to \infty.
	\end{equation*}
	
	{\bf Case 1. For the rotationally symmetric stable noise.} Using It\^{o}'s formula and Stein's equation \eqref{e:steinbb}, we have
	\begin{equation*}
		\begin{aligned}
			f_h(\bfX_t^\bfx)-f_h(\bfx)
			&=\int_0^t \mathcal{A}^{\alpha} f_h(\bfX_s^\bfx)\dif s
			+\int_0^t \int_{\R_0^d} \left[f_h(\bfX_{s-}^\bfx+\bmsigma \bfz)-f_h(\bfX_{s-}^\bfx) \right] \widetilde{N}(\dif s, \dif \bfz)  \nonumber \\
		&=\int_0^t [h(\bfX_s^\bfx)-\mu(h)] \dif s
			+\int_0^t \int_{\R_0^d} \left[f_h(\bfX_{s-}^\bfx+\bmsigma \bfz)-f_h(\bfX_{s-}^\bfx) \right] \widetilde{N}(\dif s, \dif \bfz).
		\end{aligned}
	\end{equation*}
	Combining this with ergodic theorem and the fact that $\E \big| [ f_h(\bfX_t^{\bfx})-f_h(\bfx)] / {\sqrt{t}} \big| \to 0$ as $t \to \infty$, one yields that
	\begin{equation*}
		\begin{aligned}
			&\mathrel{\phantom{=}}
			\frac{1}{t}\E^{\mu} \left( \int_0^t [h(\bfX_s^\bfx) - \mu(h)] \dif s \right)^2 \\
			&= \E^{\mu} \left( \frac{1}{\sqrt{t}} \left[f_h(\bfX_t^\bfx)-f_h(\bfx)\right] - \int_0^t \int_{\R_0^d} \left[f_h(\bfX_{s-}^\bfx+\bmsigma \bfz)-f_h(\bfX_{s-}^\bfx) \right] \widetilde{N}(\dif s, \dif \bfz) \right)^2 \\
			& \to \mathcal{V}(f_h) , \quad\text{ as }\ t \to \infty.
		\end{aligned}
	\end{equation*}
	Thus, we know $\mathcal{H}(h) = \mathcal{V}(f_h)$.
	
	{\bf Case 2. For the  cylindrical  stable noise.} Using It\^{o}'s formula and Stein's equation \eqref{e:steinbb},
	\begin{equation*}
		\begin{aligned}			
			&\mathrel{\phantom{=}}
			f_h(\bfX_t^{\bfx})-f_h(\bfx)
			= \int_0^t \mathcal{A}^{\alpha} f_h(\bfX_s^{\bfx})\dif s
			+\sum_{i=1}^d \int_0^t \int_{\R_0} \left[f_h(\bfX_{s-}^{\bfx}+\bmsigma_i z_i) -f_h(\bfX_{s-}^{\bfx}) \right] \widetilde{N}(\dif s, \dif z_i) \nonumber \\
			&= \int_0^t [h(\bfX_s^{\bfx})-\mu(h)] \dif s
			+\sum_{i=1}^d \int_0^t \int_{\R_0} \left[f_h(\bfX_{s-}^{\bfx}+ \bmsigma_i z_i)-f_h(\bfX_{s-}^\bfx) \right] \widetilde{N}(\dif s, \dif z_i),
		\end{aligned}
	\end{equation*}
	with similar discussions for {\bf Case 1}, we also know that $\mathcal{H}(h)=\mathcal{V}(f_h)$. Combining above two cases, the claim \eqref{e:Vh} holds. The proof is complete.
\end{proof}

\section{Optimal error bound for the classical Ornstein-Uhlenbeck process}\label{seclow}

In this section, we shall show that the convergence rate obtained by our method is optimal for the classical Ornstein-Uhlenbeck process.

We consider the following Ornstein-Uhlenbeck process defined as
\begin{equation}\label{e:OU}
\dif X_t  = -X_t \dif t + \dif Z_t, \quad  X_0 = 0,
\end{equation}
where $(Z_t)_{t\geq 0}$ is a one-dimensional symmetric $\alpha$-stable process with $\alpha\in(1,2]$. It is known that the solution to \eqref{e:OU} is
\begin{equation*}
	X_t =  \rme^{-t}X_0 + \rme^{-t} \int_0^t \rme^s \dif Z_s  =  \int_0^t \rme^{s-t} \dif Z_s, \quad \forall \  t\geq 0 .
\end{equation*}
It follows from \cite[Proposition 3.4.1]{Taqqu1994Stable} that the characteristic function for $X_t^0$ is
\begin{equation*}
	\begin{aligned}
		\E\left[\exp(i u  X_t^0)\right]
		&=  \E\left[\exp\left(i\int_0^t  u  \rme^{s-t} \dif Z_s\right)\right]
		=  \exp\left(-\int_0^t | u  \rme^{s-t}|^{\alpha} \dif s \right) \\
		& =   \exp\left(-\frac{1}{\alpha}| u |^{\alpha} [1-\rme^{-\alpha t}] \right),
	\end{aligned}
\end{equation*}
for any $ u  \in \bbR$. It immediately leads to that $\E[\exp(i u  X_t^0)] \to \exp(-| u |^{\alpha} / \alpha)$ as $t \to \infty$.

The EM scheme of \eqref{e:OU} is given as follows: $Y_0=X_0 = 0$ and
\begin{equation}\label{e-OU1}
	Y_{t_{n+1}}
	= Y_{t_{n}}-\eta_{n+1}Y_{t_{n}} + \Delta Z_{\eta_{n+1}},  \quad n\in \mathbb{N}_0,
\end{equation}
where $t_n= \sum_{i=1}^n \eta_i$ with $t_0=0$ and $\Delta Z_{\eta_{n+1}}=Z_{t_{n+1}}-Z_{t_{n}}$.

\begin{proposition}\label{lower bound}
	Denote the invariant measure of Ornstein-Uhlenbeck process $(X_{t}^0)_{t\geq0}$ in \eqref{e:OU} by $\mu$ and let $(Y_{t_{n}}^0)_{n\in\mathbb{N}_{0}}$
	be defined in \eqref{e-OU1}. Under Assumption \ref{assump-2}, if $\beta=\alpha^{-1}$ and  $\omega=\alpha$ in \eqref{e:Q} for any $\alpha\in(1,2]$, we have
	\begin{equation*}
		0
		< \liminf_{n\rightarrow\infty} \frac{ \mathcal{W}_1(\mu,\mathcal{L}(Y_{t_{n}}^{0}))}
		{\eta_n^{1/\alpha}}
		\ < \ \limsup_{n\rightarrow\infty} \frac{ \mathcal{W}_1(\mu,\mathcal{L}(Y_{t_{n}}^{0}))}
		{\eta_n^{1/\alpha}}
		\ < \ \infty.
	\end{equation*}
\end{proposition}
\begin{proof}[Proof of Proposition \ref{lower bound}]
	By induction, it follows from \eqref{e-OU1} that
	\begin{equation*}
		\begin{aligned}
			Y_{t_{n}}
			&= (1-\eta_n) Y_{t_{n-1}} + \Delta Z_{\eta_n}
			=  (1-\eta_{n-1})(1-\eta_n) Y_{t_{n-2}} + (1-\eta_n) \Delta Z_{\eta_{n-1}} + \Delta Z_{\eta_n} \\
			&= \sum_{j=1}^{n} \Delta Z_{\eta_{j}} \prod_{k=j+1}^{n}(1-\eta_{k}) .
		\end{aligned}
	\end{equation*}
	Hence, the characteristic function for $Y_{t_n}^0$ is given by
	\begin{equation}\label{com1}
		\mathbb{E}\left[ \exp\left( i u  Y_{t_{n}}^0 \right) \right]
		= \exp\left(-| u |^{\alpha} \sum_{j=1}^n \eta_j \prod_{k=j+1}^n (1-\eta_k)^{\alpha} \right)
	\end{equation}
	for any $ u  \in \R$. Besides, the fact $1-x \le \rme^{-x}$ and $t_n = \sum_{k=1}^n \eta_k$ yield that
	\begin{equation} \label{com3}
		\sum_{j=1}^{n}\eta_{j}\prod_{k=j+1}^{n}
		(1-\eta_{k})^{\alpha}
		\leq   \sum_{j=1}^{n}\eta_{j} \rme^{-\alpha(t_n-t_j)}
		\leq   \frac 1{\alpha}(1-\rme^{-\alpha t_n})
		+ C \eta_n
	\end{equation}
	holds for some constant $C>0$ independent of $n$. The last inequality above is obtained by
	\begin{equation*}
		\begin{aligned}
			\mathrel{\phantom{=}} \sum_{j=1}^{n}\eta_{j} \rme^{-\alpha(t_n-t_j)}- \frac{1}{\alpha}+\frac{1}{\alpha} \rme^{-\alpha t_n}
			& =  \rme^{-\alpha t_n} \left[ \sum_{j=1}^{n}\eta_{j} \rme^{\alpha t_j}- \int_0^{t_n} \rme^{\alpha s} \dif s \right] \\
			&= \rme^{-\alpha t_n} \sum_{j=1}^{n} \int_{t_{j-1}}^{t_j} (\rme^{\alpha t_j}
			- \rme^{\alpha s}) \dif s \\
			& \leq   \rme^{-\alpha t_n} \sum_{j=1}^{n} \int_{t_{j-1}}^{t_j} \alpha \rme^{\alpha t_j} (t_j-s) \dif s  \\
			&= \frac{\alpha}{2} \sum_{j=1}^{n}  \rme^{-\alpha (t_n - t_j)} \eta_j^2
			\leq  C \eta_n,
		\end{aligned}
	\end{equation*}
	where the first inequality holds from the mean value theorem and the last inequality holds by Lemma \ref{lem:series}.

	Combining \eqref{com1} and \eqref{com3}, we have
	\begin{equation*}
		\mathbb{E}\left[\exp\left(i u  Y_{t_{n}}^0 \right)\right]
		\geq \exp\left( -| u |^{\alpha} \Big[ \frac{1}{\alpha}\big(1-\rme^{-\alpha t_n} \big) + C\eta_n \Big]\right).
	\end{equation*}
	Hence, there exists a constant $C>0$ independent of $n$ such that
	\begin{equation}\label{com2}
		\begin{aligned}
			&\mathrel{\phantom{=}} \int_{-1}^1 \left[ \mathbb{E}\left[\exp\left(i u  Y_{t_{n}}^0
			\right)\right]
			-\mathbb{E}\big[\exp\big(i u  \alpha^{-{1}/{\alpha}}Z_1\big)\big] \right] \dif  u  \\
			&\geq \int_{-1}^1 \left[ \exp\left( -| u |^{\alpha}\Big[ \frac{1}{\alpha}\big(1-\rme^{-\alpha t_n}\big) + C\eta_n \Big]\right)
			-\exp\Big(-\frac{1}{\alpha}| u |^{\alpha} \Big)   \right] \dif  u   \\
			&= \int_{-1}^1 \exp\Big(-\frac{1}{\alpha}| u |^{\alpha} \Big)
			\left[
			\exp\left(\frac{1}{\alpha}| u |^{\alpha}
			\rme^{-\alpha t_n} + C| u |^{\alpha} \eta_n\right)-1
			\right] \dif  u   \\
			&\geq \int_{-1}^1 \exp\Big(-\frac{1}{\alpha}| u |^{\alpha} \Big)
			\left[\frac{1}{\alpha}| u |^{\alpha}
			\rme^{-\alpha t_n} + C| u |^{\alpha} \eta_n\right] \dif  u
			= C\big( \rme^{-\alpha t_n} + \eta_n \big).
		\end{aligned}
	\end{equation}
	
	For constants $\omega$ and $\beta$ in \eqref{e:Q}, we claim that
	\begin{equation}\label{e:claim-ss3}
		\frac{\rme^{-\alpha t_n}}{\eta_n^{\beta}}
		=
		\begin{cases}
			O(1),  & \text{if }  \omega=\alpha, \\
			o(1),  & \text{if }  \omega<\alpha,
		\end{cases}
	\end{equation}
	where $f(n)=O(g(n))$ as $n\rightarrow\infty$ means that $\lim_{n\rightarrow\infty} f(n)/g(n)
	<\infty$, $f(n)=o(g(n))$ as $n\rightarrow\infty$ means that $\lim_{n\rightarrow\infty} f(n)/g(n)=0$ for two sequences $\{f(n), n \in \mathbb{N}\}$ and $\{g(n), n \in \mathbb{N}\}$.

	Furthermore, by the similar argument as in the proof of \cite[Proposition B.1]{Chen2023Approximation} and taking
	\begin{equation*}
		h(y) = \frac{1}{M}\left(\frac{\sin y}{y}
		{\bf 1}_{\{y\neq 0\}}+{\bf 1}_{\{y=0\}}\right),\quad y\in\mathbb{R},
	\end{equation*}
	with constant  $M=\sup_{z\in\R_0}|{[z\cos z-\sin z]}/{z^{2}}|\in(0,\infty)$, we can obtain that
	\begin{equation*}
		\begin{aligned}
			\mathcal{W}_1(\mu,Y_{t_n}^0)
			&\geq \left|  \E[h(Y_{t_n}^0)]-\E[h(\alpha^{-1/\alpha} Z_1)]  \right|  \\
			&= \frac{1}{2M} \left| \int_{-1}^{1}\Big[
			\mathbb{E} \left[\exp\left(i u  Y_{t_{n}}^0 \right)\right]
			-\mathbb{E}\big[\exp\big(
			i u  \alpha^{-{1}/{\alpha}}Z_{1} \big)\big]\Big]\dif  u   \right|
		\geq  C (\rme^{-\alpha t_n} + \eta_n) \\
			& =  O(\eta_n^{{1}/{\alpha}}),
		\end{aligned}
	\end{equation*}
	where the inequality arises from \eqref{com2} and the last equality holds by taking $\beta = 1/\alpha$ and $\omega = \alpha$ in claim \eqref{e:claim-ss3}.
	
	The reminder is to verify the claim \eqref{e:claim-ss3}. Recall that
	\begin{equation*}
		\omega
		= \lim_{n\to \infty} \left\{ \frac{\eta_n^{\beta} -\eta_{n+1}^{\beta}}{\eta_{n+1}^{1+\beta}} \right\}
		< \infty ,
	\end{equation*}
	which leads to $1+\frac{\eta_{n-1}^{\beta}-\eta_n^{\beta}}
	{\eta_n^{1+\beta}} \eta_n \asymp \exp(\omega \eta_n)$, where $f(n)\asymp g(n)$ as $n\rightarrow\infty$ means that $\lim_{n\to \infty} f(n)/g(n)=1$ for two sequences $\{f(n), n \in \mathbb{N}\}$ and $\{g(n), n \in \mathbb{N}\}$. Let $\omega_n=\exp(-\alpha t_n)\eta_n^{-\beta}, \forall n\in \mathbb{N}_+$, then one has
	\begin{equation*}
		\omega_{n}
		= \omega_{n-1} \rme^{-\alpha \eta_n} \cdot \frac{\eta_{n-1}^{\beta}}{\eta_n^{\beta}}
		= \omega_{n-1} \rme^{-\alpha \eta_n} \left(1+ \frac{\eta_{n-1}^{\beta}-\eta_n^{\beta}}
		{\eta_n^{1+\beta}} \eta_n \right)
		\asymp  \omega_{n-1}\rme^{(-\alpha+\omega)\eta_n}.
	\end{equation*}
	
	By induction, one gets that
	\begin{equation*}
		\omega_{n} \asymp \omega_1 \rme^{(-\alpha+\omega)\eta_1} \rme^{(-\alpha+\omega)t_n}   =  O\left(\rme^{(-\alpha + \omega) t_n} \right).
	\end{equation*}
	Thus, the claim \eqref{e:claim-ss3} holds. The proof is complete.
\end{proof}

\begin{remark}
	As an example of the decreasing step size $\gamma = (\eta_n)_{n\in \mathbb{N}}$, we set
	\[
	\eta_n = \frac{1}{\alpha^2 n}, \quad \forall \ n \in \mathbb{N}.
	\]
	Then, taking $\beta = 1/\alpha$ in Assumption \ref{assump-2}, we can verify that
	\[
	\omega = \lim_{n\to \infty} \frac{\eta_n^{\beta} - \eta_{n+1}^{\beta}}{\eta_{n+1}^{1+\beta}}
	= \lim_{n\to \infty} \alpha^2 (n+1) \left[ \left( 1+\frac1n \right)^{1/\alpha} - 1 \right]
	= \alpha.
	\]
	Assumption \ref{assump-2} and $\rme^{-\alpha t_n} = \eta_n^{1/\alpha}$ hold, as long as $\theta$ in \eqref{e:lambda} is larger such that $\theta > \alpha$.
\end{remark}

\section{Sinkhorn-Knopp algorithm} In this section, we shall use the Sinkhorn-Knopp algorithm to verify our convergence results in $\mathcal W_1$ distance by several concrete examples in queue networks. Let us first recall some preliminary of the Sinkhorn-Knopp algorithm.
\subsection{Preliminary of Sinkhorn-Knopp algorithm}
Let $\mathcal{P}(\R^d)$ be the set of probability measures on $\R^d$. Let
\begin{equation*}
	\nu_1=\sum_{i=1}^n a_i \delta_{ \bfx_i}, \ \ \ \nu_2=\sum_{i=1}^m b_j \delta_{\bfy_j}
\end{equation*}
be two discrete probabilities in $\mathcal{P}(\R^d)$ centered on $\{\bfx_i: 1 \le i \le n\} \subseteq \R^d$ and $\{\bfy_j:1 \le j \le m\} \subseteq \R^d$ respectively, where $\bfa=(a_1,\cdots,a_n)^{\prime}\in \Delta_n$ and $\bfb=(b_1,\cdots,b_m)^{\prime} \in \Delta_m$ with
\begin{equation*}
	\Delta_n = \{\bfp\in \R^n_+: \bfp^{\prime}{\rm \bf e}_n =1  \},  \quad \Delta_m = \{{\bf q} \in \R^m_+: {\bf q}^{\prime}{\rm \bf e}_m =1  \},
\end{equation*}
where $\R^n_+$ is the set of vectors $\bfp\in \R^n$ with non-negative entries and ${\rm \bf e}_n=(1,...,1)' \in \R^n$ and $\Delta_m$ is similarly defined. In this setting, the evaluation of
the Wasserstein distance corresponds to solving a network flow problem in terms of the weight vectors $\bfa$ and $\bfb$ (see, \cite[Eq. 4]{Luise2018Differential} or \cite{Bertsimas1997Introduction})
\begin{equation}\label{e:sim-1}
	\mathcal{W}_1 (\nu_1,\nu_2) = \min_{ {\bf T} \in \mathcal{C}(\bfa,\bfb)} \langle {\bf T}, \bfM  \rangle_{\rm HS},
\end{equation}
where $ \bfM =(M_{ij})_{n\times m}$ is the cost matrix with entries $M_{ij}=|\bfx_i-\bfy_j|$.

A more practical approximation of \eqref{e:sim-1} is Sinkhorn distance $S^{ \tau }(\bfa, \bfb)$, given by
\begin{equation}\label{e:sim-2}
	S^{ \tau }(\bfa,\bfb) = \langle {\bf T}^{ \tau },  \bfM  \rangle_{\rm HS}
	{\ \ \rm with \ \ } {\bf T}^{ \tau }= \arg \min_{{\bf T}\in \mathcal{C}(\bfa,\bfb)} \left\{ \langle {\bf T}, \bfM  \rangle_{\rm HS}-\frac{1}{ \tau } h({\bf T}) \right\},
\end{equation}
where $ \tau >0$ is a regularization parameter, $h({\bf T})=-\sum_{i,j=1}^{n,m}T_{ij}(\log T_{ij}-1)$ and matrix ${\bf T}=(T_{ij})_{n\times m}$. Then, we have the following approximation between $S^{ \tau }(\bfa,\bfb)$ and $ \mathcal{W}_1 (\nu_1,\nu_2)$:
\begin{lemma}{{\rm(}\cite[Proposition 1]{Luise2018Differential}{\rm )}}
	For any pair of discrete measures $\nu_1,\nu_2\in \mathcal{P}(\R^d)$ with respective weights $\bfa\in \Delta_n$ and $\bfb\in \Delta_m$, we have
	\begin{equation*}
		|S^{ \tau }(\bfa,\bfb)- \mathcal{W}_1 (\nu_1,\nu_2)|
		\leq C \rme^{- \tau },
	\end{equation*}
	where the constant $C$ is independent of $ \tau $ but depends on the supports of $\nu_1$ and $\nu_2$.
\end{lemma}

It follows from \cite[Eq. (3)]{Browne1995Piecewise} that there exist vectors $\bfr=(r_1,\cdots,r_n)^{\prime}$ and $\bfc=(c_1,\cdots,c_m)^{\prime}$ such that the matrix ${\bf T}^{ \tau }=(T^{ \tau }_{ij})_{n\times m}$ satisfies
\begin{equation*}
	T^{ \tau }_{ij} = r_i \rme^{- \tau  M_{ij}} c_j, \quad 1\leqslant i \leqslant n, \  1\leqslant j\leqslant m.
\end{equation*}
To make notations simple, we write   $A^{ \tau }_{ij}=\rme^{- \tau  M_{ij}}$ for each $i,j$ and denote by $\bfA^{ \tau }$ the corresponding matrix. Then,
\begin{equation}\label{e:sim-3}
	{\bf T}^{ \tau } = {\rm diag}(\bfr) \bfA^{ \tau } {\rm diag} (\bfc).
\end{equation}
It follows from \cite[Section 2]{Knight2008The} that the vectors $\bfc$ and $\bfr$ can be solved by
\begin{equation*}
	\bfc  = \bfb./ ((\bfA^{ \tau })^{\prime} \bfr)
	\quad \text{ and } \quad
	\bfr  = \bfa./(\bfA^{ \tau } \bfc).
\end{equation*}
The above expressions for vectors $\bfr$ and $\bfc$  suggest a fixed point iteration as the following:
\begin{equation*}
	\bfc_{k+1}  =  \bfb./ ((\bfA^{ \tau })^{\prime} \bfr_k),
	\quad \text{ and } \quad
	\bfr_{k+1}  =  \bfa./(\bfA^{ \tau } \bfc_{k+1}), \quad k \in \mathbb{N}.
\end{equation*}
Furthermore, it follows from \eqref{e:sim-2} and \eqref{e:sim-3} that $ \mathcal{W}_1 (\nu_1,\nu_2)$ can be approximated by $S^{ \tau }(\bfa,\bfb)$.

\subsection{Simulations for two queue networks}
We point out that the simulation for one dimensional symmetric $\alpha$-stable random variables with $\alpha\in (1,2)$ can be found in \cite[Theorem 1.3]{Nolan2020Univariate} or \cite[Proposition 1.7.1]{Taqqu1994Stable}.

In the following examples, let the step size $\Lambda=(\eta_n)_{n\in \mathbb{N}}$ of the EM scheme  be $\eta_n=1/(10+n)$, $\forall\ n\in \mathbb{N}$, and the corresponding time be $t_n=\sum_{k=1}^n \eta_{k}$, $\forall\ n\in \mathbb{N}$ with $t_0=0$. Besides, let the regularization parameter $ \tau =20$ in Sinkhorn-Knopp algorithm for $1$-dimension. In the case of $2$-dimension, we let the regularization parameter in the Sinkhorn-Knopp algorithm $ \tau  = 80$ to obtain a better estimation.

\begin{example}[\bf{Stable piecewise Ornstein-Uhlenbeck process}]\label{ex stable POU}
	As shown in \cite{pang2010queueing}, {assuming that the arrival process is heavy-tailed and under the proper scaling, it converges to a symmetric $\alpha$-stable motion,  the scaling limit for the $G/M/N+M$ queues with such an arrival process}  is
	$$X(t) = X(0) -\beta \varrho t  - \int_0^t [\varrho( X(s) \wedge 0) + \gamma (X(s) \vee0) ] ds+Z_\alpha(t),$$
	where $Z_\alpha(t)$ is a symmetric $\alpha$-stable process with $\alpha \in (1,2)$. Here we take $\gamma =2$, $\varrho=1$, $\beta=1$, so that $\ell=-1$.
	The parameters in our SDE are: $\alpha \in (1,2)$, $d=1$,  $\ell=-1$, $M=1$, $(M+(\Gamma-M)v)=2$ and $\sigma=1$. Hence the coefficient $g$ is
	\begin{equation*}
		g(x)
		=
		\begin{cases}
			-1-2x,  & \text{if }  x >0, \nonumber \\
			-1-x,  & \text{if }  x \leq 0,
		\end{cases}
	\end{equation*}
	and the corresponding process $( X_t)_{t\geq 0}$ satisfies
	\begin{equation*}
		\dif  X_t = g( X_t) \dif t + \dif Z_t, \quad X_0  = x,
	\end{equation*}
	where $(Z_t)_{t\geq 0}$ is the symmetric $\alpha$-stable process with L\'{e}vy measure $\nu_1(\dif z)=c_{\alpha}|z|^{-1-\alpha} \dif z$.

	The ergodic measure $\mu$ in Theorem \ref{thm:XY-dis} does not have an explicit form, and we shall take the law of $Y_{t_{n}}^x$ with a sufficiently large $n$ as its estimation, denoted by $\hat \mu$. We shall use the following four steps to estimate $ \mathcal{W}_1 (\mu,\mathcal{L}(Y_{t_n}^{\bfx}))$. In our simulation, we take $\alpha\in \{1.25,1.5,1.75\}$.
	
	Step 1. Empirical distribution of $\mathcal{L}(Y_{t_n}^{\bfx})$. Let $N_1=1000$ be iteration for the EM scheme and generate $N_2=2000$ trajectories for the Markov chain $(Y_{t_{n}})_{n\in \mathbb{N}_0}$, denoted by $\{Y_{t_{n}}(\omega_i) \}_{i=1}^{N_2}$, that is, for any $i\in \{1,2,\cdots,N_2\}$ and $n\in \{0,1,2,\cdots,N_1-1\}$, one has
	\begin{equation*}
		Y_{t_{n+1}}(\omega_i) = Y_{t_n}(\omega_i) + \eta_{n+1} g(Y_{t_n}(\omega_i))+\frac{1}{\tilde{\sigma}}  \eta_{n+1}^{{1}/{\alpha}} \xi_{n+1}(\omega_i), \quad  Y_0(\omega_i) = x = 0,
	\end{equation*}
	where the constant $\tilde{\sigma}=(\frac{\alpha}{2c_{\alpha}})
	^{1/\alpha}$  and $(\xi_n)_{n\in \mathbb{N}}$ are stable random variables. Then, the law $\mathcal{L}(Y_{t_{N_1}}^0)$ can be approximated as
	$\mathcal{L}(Y_{t_{N_1}}^0)
	\approx N_2^{-1}\sum_{i=1}^{N_2} \delta_{Y^{0}_{t_{N_1}}(\omega_i)}.$
	
	Step 2. Estimation of $\mu$. Let $N_4=2N_1$ be iteration for the EM scheme and generate $N_2=2000$ trajectories for the Markov chain $(\tl{Y}_{t_{n}})_{n\in \mathbb{N}_0}$ with different initial values $\bfy={\rm linspace}(-20,20,N_2)$ which are all independent of those in Step 1, denote by $\{\tilde{Y}_{t_{n}}(\omega_j)\}_{j=1}^{N_2}$, that is, for any $j\in \{1,2,\cdots,N_2\}$ and $n\in \{0,1,2,\cdots,N_4-1\}$, one knows
	\begin{equation*}
		\tilde Y_{t_{n+1}}(\omega_j) = \tilde Y_{t_n}(\omega_j) + \eta_{n+1} g(\tilde Y_{t_n}(\omega_j))+ \frac{1}{\tilde{\sigma}} \eta_{n+1}^{{1}/{\alpha}} \xi_{n+1}(\omega_j), \quad  \tilde Y_0(\omega_j) = \bfy(j),
	\end{equation*}
	where $\bfy(j)$ means the $j$-th element of vector $\bfy$, and $(\xi_n)_{n\in \mathbb{N}}$ are $\alpha$ stable random variables. We use the empirical measures
	$\hat \mu=N_2^{-1}\sum_{j=1}^{N_2} \delta_{\tilde Y^{\bfy(j)}_{t_{N_4}}(\omega_j)}$ as an estimation of the ergodic measure $\mu$.
	
	Step 3. Run the Sinkhorn-Knopp algorithm. Let vectors ${\bf c}_0={\bf r}_0={\rm \bf e}_{N_2}$, $\bfa=\bfb=N_2^{-1} {\rm \bf e}_{N_2}$, two empirical measures
	$N_2^{-1}\sum_{i=1}^{N_2} \delta_{Y^{0}_{t_{N_1}}(\omega_i)}$ and
	$\hat \mu$, one obtains the value of $S^{\tau}$ which is the approximation of $ \mathcal{W}_1 (\mu, \mathcal{L}(Y_{t_{N_1}}^0))$.
	
	Step 4. Obtain a better approximation. Repeat the above three steps for $N_3=20$ times, and we can obtain the values of $ \mathcal{W}_1 ^{(i)}(\mu, \mathcal{L}(Y_{t_{N_1}}^0))$ with $i=1,\cdots,N_3$ in Step 3, and then take the mean of those $N_3=20$ distances to approximate the Wasserstein-1 distance.
	
	In addition, we can get the following figures. Figures \ref{figure S4}, \ref{figure S5} and \ref{figure S6} are the Wasserstein-1 distances for various iterations with $\alpha\in \{1.25,1.5,1.75\}$, that is, $ \mathcal{W}_1 (\mu,\mathcal{L}(Y_{t_{k}}^0))$ with  $k\in\{10,20,30,\cdots,1000\}.$
	\begin{figure*}[!httpb]
		\centering
		\subfigure[\label{figure S4}Wasserstein-1 distances with different iterations and  $\alpha=1.25$]{
			\begin{minipage}[t]{0.3\linewidth}
				\centering \includegraphics[width=\linewidth]{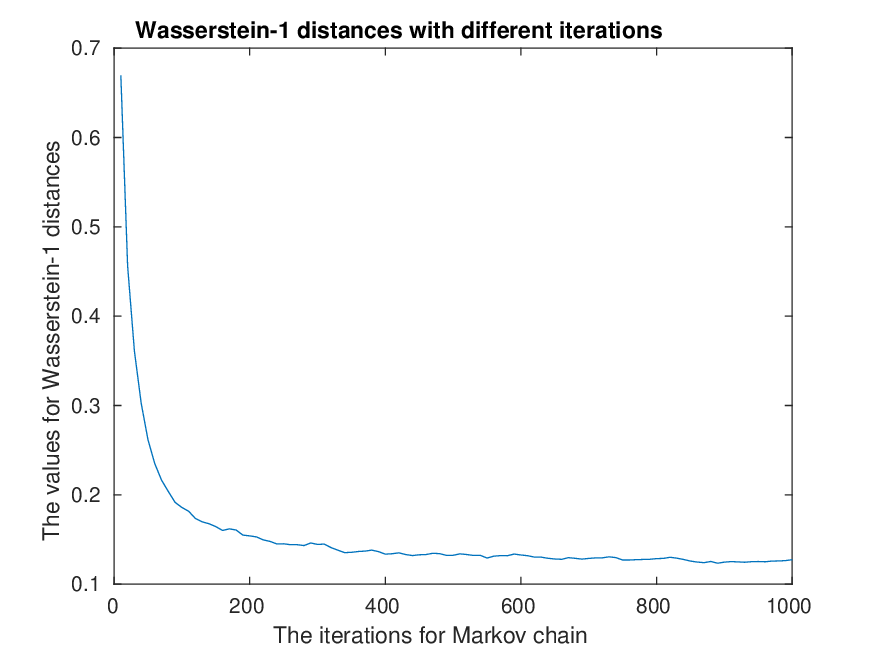}
			\end{minipage}
		}%
		\subfigure[\label{figure S5}Wasserstein-1 distances with different iterations and $\alpha=1.5$]{
			\begin{minipage}[t]{0.3\linewidth}
				\centering \includegraphics[width=\linewidth]{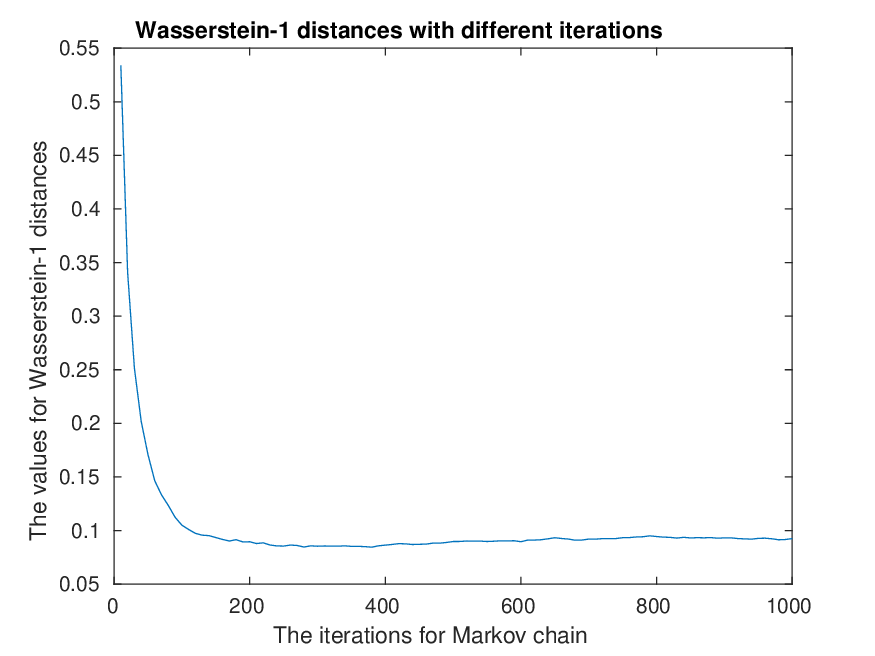}
			\end{minipage}
		}%
		\subfigure[\label{figure S6}Wasserstein-1 distances with different iterations and $\alpha=1.75$]{
			\begin{minipage}[t]{0.3\linewidth}
				\centering \includegraphics[width=\linewidth]{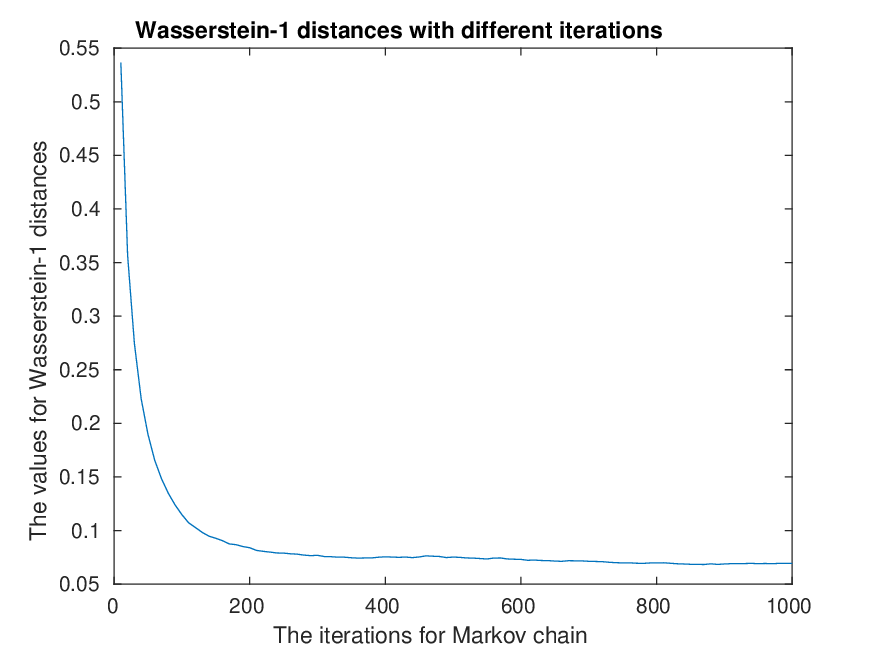}
			\end{minipage}
		}%
		\centering
		\caption{Wasserstein-1 distances for Example \ref{ex stable POU} with different iterations and $\alpha=1.25$, $1.5$ and $1.75$}
		\label{fig2}
	\end{figure*}
	
\end{example}

\begin{example}[\bf{The stable noise with independent components.}] \label{exam:2dim}
	
	We consider a two-class $G/M/N+M$ queues (the `$V$' network, see, \cite[pp. 1086-1087]{Arapostathis2019Ergodicity}), where the limit will have $\bfM={\rm diag}(m_1,m_2)$ with $m_1=2,m_2=1$ and $\bmGamma={\rm diag}(\gamma_1, \gamma_2)$ with $\gamma_1=1,\gamma_2=2$, we can choose the control $\bfv=(1/2,1/2)^{\prime}$, and $\bm{\ell}=(-1/2, -1/4)^{\prime}$, then the coefficient $g$ is given as
	\begin{equation*}
		g( \bfx)
		= \left\{
		\begin{aligned}
			\begin{pmatrix}
				-{1}/{2}  \\
				- {1}/{4}
			\end{pmatrix}  &-
			\begin{pmatrix}
				{3}/{2} &  - {1}/{2} \\
				{1}/{2} & {3}/{2}
			\end{pmatrix} \bfx ,  \quad \text{if }  x_1+x_2 >0,  \\
			\begin{pmatrix}
				-{1}/{2}  \\
				- {1}/{4}
			\end{pmatrix} &-
			\begin{pmatrix}
				2 & 0  \\
				0 & 1
			\end{pmatrix}\bfx,  \quad \text{if }  x_1+x_2 \leq 0,
		\end{aligned}
		\right.
	\end{equation*}
	where $\bfx = (x_1,x_2)^{\prime} \in \bbR^2$. The corresponding process $(\bfX_t)_{t\geq 0}$ satisfies that
	\begin{equation*}
		\dif \bfX_t = g(\bfX_t) \dif t + \dif \bfZ_t, \quad \bfX_0 = \bfx,
	\end{equation*}
	where $\bfZ_t=(Z_{1,t},Z_{2,t})^{\prime}$ and $Z_{i,t}$, $i=1,2$ are {independent} one-dimensional symmetric $\alpha$-stable processes with L\'{e}vy measure $\nu_1(\dif z)=c_{\alpha}|z|^{-1-\alpha} \dif z$. We shall use the following four steps to estimate $ \mathcal{W}_1 (\mu,\mathcal{L}(\bfY_{t_n}^x))$. In our simulation, we choose $\alpha=1.8$.
	
	Step 1. Empirical distribution of $\mathcal{L}(\bfY_{t_n}^x)$. Let $N_1 = 800$ be iteration for the EM scheme and generate $N_2=6000$ trajectories for the Markov chain $(\bfY_{t_{n}})_{n\in \mathbb{N}_0}$ denote by $\{\bfY_{t_{n}}(\omega_i) \}_{i=1}^{N_2}$, that is, for any $i\in \{1,2,\cdots,N_2\}$ and $n\in \{0,1,2,\cdots,N_1-1\}$, one has $\bfY_{0}(\omega_i)={\bf 0}$ and
	\begin{equation*}
		\begin{aligned}
			Y_{1,t_{n+1}}(\omega_i) &= Y_{1,t_n}(\omega_i) + \eta_{n+1} [g(\bfY_{t_n}(\omega_i))](1)+
			\frac{1}{\tilde{\sigma}}  \eta_{n+1}^{ {1}/{\alpha} } \xi_{1,n+1}(\omega_i), \\
			Y_{2,t_{n+1}}(\omega_i) &= Y_{2,t_n}(\omega_i) + \eta_{n+1} [g(\bfY_{t_n}(\omega_i))](2)
			+\frac{1}{\tilde{\sigma}}  \eta_{n+1}^{ {1}/{\alpha} } \xi_{2,n+1}(\omega_i),
		\end{aligned}
	\end{equation*}
	where $\bfY_{t_n}=(Y_{1,t_n},Y_{2,t_n})^{\prime}$,
	and $[{\bf a}](i)$ means the i-th element of vector ${\bf a}$ and the constant $\tilde{\sigma}=(\frac{\alpha}{2c_{\alpha}})^{1/\alpha}$  and  $(\xi_{i,n})_{n\in \mathbb{N}},i=1,2$ are stable random variables. Then, the law $\mathcal{L}(\bfY_{t_{N_1}}^{\bf 0})$ can be approximated as
	$\mathcal{L}(\bfY_{t_{N_1}}^{\bf 0})
	\approx N_2^{-1}\sum_{i=1}^{N_2} \delta_{\bfY^{{\bf 0}}_{t_{N_1}}(\omega_i)}.$
	
	Step 2. Estimation of $\mu$. Let $N_4= 10 N_1$ be iteration for the EM scheme and generate $N_2$ trajectories for the Markov chain $(\bfY_{t_{n}})_{n\in \mathbb{N}_0}$ with different initial values which are all independent of those in Step 1, denote by $\{\tilde{\bfY}_{t_{n}}(\omega_j)\}_{j=1}^{N_2}$, that is, for any $j\in \{1,2,\cdots,N_2\}$ and $n\in \{0,1,2,\cdots,N_4-1\}$, one knows $\tilde{\bfY}_0(\omega_j)=({\bfy}(j),{\bfy}(j))$ and
	\begin{equation*}
		\begin{aligned}
			\tilde Y_{1,t_{n+1}}(\omega_j) &= \tilde Y_{1,t_n}(\omega_j) + \eta_{n+1} [g(\tilde \bfY_{t_n}(\omega_j))](1)
			+\frac{1}{\tilde{\sigma}}  \eta_{n+1}^{ {1}/{\alpha} } \xi_{1,n+1}(\omega_j),  \\
			\tilde Y_{2,t_{n+1}}(\omega_j) &= \tilde Y_{2,t_n}(\omega_j) + \eta_{n+1} [g(\tilde \bfY_{t_n}(\omega_j))](2)
			+\frac{1}{\tilde{\sigma}}  \eta_{n+1}^{ {1}/{\alpha} } \xi_{2,n+1}(\omega_j),
		\end{aligned}
	\end{equation*}
	where ${\bfy}(j)$ means the $j$-th element of vector ${\bf y}={\rm linspace}(-20,20,N_2)$, and the random variables $(\xi_{i,n})_{n\in \mathbb{N}},i=1,2$ have stable distribution. We use the empirical measures
	$\hat \mu=N_2^{-1}\sum_{j=1}^{N_2} \delta_{\tilde \bfY^{({ \bfy}(j),{\bfy}(j))}_{t_{N_4}}(\omega_j)}$
	as an estimation of the ergodic measure $\mu$.
	
	Step 3. Run the Sinkhorn-Knopp algorithm. Let vectors ${\bf c}_0={\bf r}_0={\rm \bf e}_{N_2}$, $\bfa=\bfb=N_2^{-1} {\rm \bf e}_{N_2}$, two empirical measures
	$N_2^{-1}\sum_{i=1}^{N_2} \delta_{\bfY^{{\bf 0}}_{t_{N_1}}(\omega_i)}$ and $\hat \mu$,
	one obtains the value of $S^{\tau}$ which is the approximation of $ \mathcal{W}_1 (\mu, \mathcal{L}(\bfY_{t_{N_1}}^{\bf 0}))$.
	
	Step 4. Obtain a better approximation. Repeat the above three steps for $N_3=20$ times, and we can obtain the values of $ \mathcal{W}_1 ^{(i)}(\mu, \mathcal{L}(\bfY_{t_{N_1}}^{\bf 0}))$ with $i=1,\cdots,N_3$ in Step 3, and then take the mean of those $N_3=20$ distances to approximate the Wasserstein-1 distance.
	
	In addition, we can get the following figure for the Wasserstein-1 distances for various iterations with $\alpha=1.8$. Figure \ref{figure 2dS84} is for the number of sample $N_2=6000$ and the iteration of EM scheme $N_1=800$.
	\begin{figure*}[httpb]
		\centering
		\includegraphics[width=10cm,height=8cm] {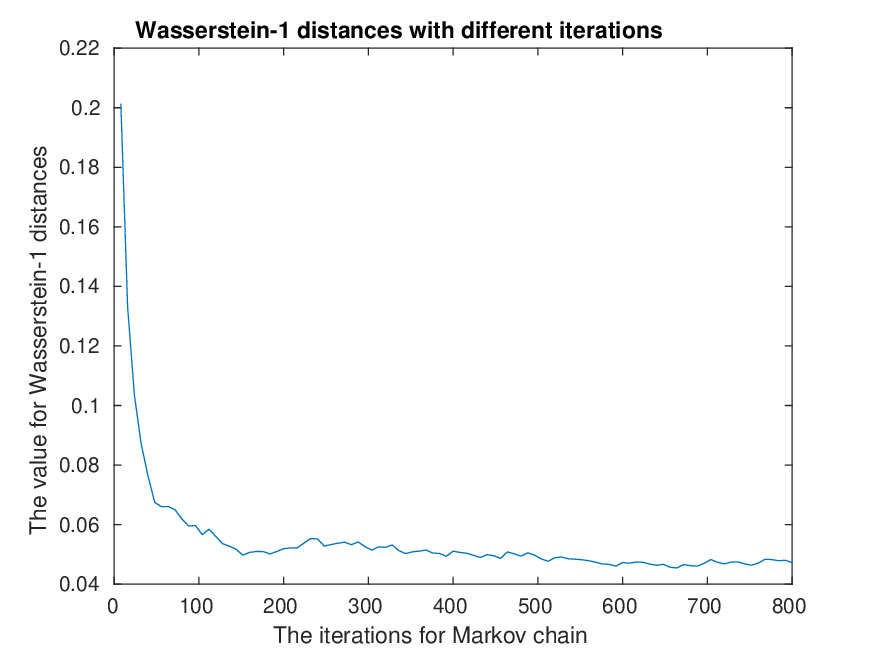}
		\caption{Wasserstein-1 distances for Example \ref{exam:2dim} with different iterations and $\alpha=1.8,N_2=6000$.}\label{figure 2dS84}
	\end{figure*}
	
\end{example}

\begin{appendix}
	
	\section{Proofs of the lemmas in Section \ref{sec:Ingredients}}\label{app_A}
	
	\subsection{Proof of Lemma \ref{lem:QNge}}\label{app_A1}
	
	\begin{proof}[Proof of Lemma \ref{lem:QNge}]
		(i) Recall that
		\begin{equation*}
			-\nabla g_{\e}( \bfx) =  \bfM  - ( \bfM - \bmGamma) \bfv{\rm \bf e}_d^{\prime} \dot{\rho_{\e}}({\rm \bf e}_d^{\prime} \bfx).
		\end{equation*}
		
		{\bf Case 1:} If ${\rm \bf e}_d^{\prime}  \bfx < -\e$, we know $\dot{\rho_{\e}}({\rm \bf e}_d^{\prime}  \bfx)=0$ and
		$-\nabla g_{\e}( \bfx) \equiv  \bfM ,$ it follows from Lemma \ref{lem:Q} that matrices
		\begin{equation*}
			[ -\nabla g_{\e}( \bfx)]^{\prime}  \bfQ  +  \bfQ  [-\nabla g_{\e}( \bfx)] \equiv   \bfM ^{\prime}  \bfQ  +  \bfQ   \bfM
		\end{equation*}
		are positive definite.
		
		{\bf Case 2:} If ${\rm \bf e}_d^{\prime}  \bfx > \e$, $\dot{\rho_{\e}}({\rm \bf e}_d^{\prime} \bfx) = 1$. Then one has
		\begin{equation*}
			-\nabla g_{\e}( \bfx) =  \bfM  - ( \bfM - \bmGamma) \bfv{\rm \bf e}_d^{\prime},
		\end{equation*}
		and the result can be obtained similarly from Lemma \ref{lem:Q}.
		
		{\bf Case 3:} If $|{\rm \bf e}_d^{\prime}  \bfx| \leq \e$, we know $\dot{\rho_{\e}}({\rm \bf e}_d^{\prime}  \bfx) \in [0,1]$ and
		\begin{equation}\label{e:Q-Q}
			\begin{aligned}
				&\mathrel{\phantom{=}}
				[-\nabla g_{\e}( \bfx)]^{\prime}  \bfQ  +  \bfQ  [-\nabla g_{\e}( \bfx)] \\
				&=
				[ \bfM  - ( \bfM - \bmGamma) \bfv{\rm \bf e}_d^{\prime} \dot{\rho_{\e}}({\rm \bf e}_d^{\prime}  \bfx)]^{\prime}  \bfQ  +  \bfQ  [ \bfM  - ( \bfM - \bmGamma) \bfv{\rm \bf e}_d^{\prime} \dot{\rho_{\e}}({\rm \bf e}_d^{\prime}  \bfx)]   \\
				&=
				\dot{\rho_{\e}}({\rm \bf e}_d^{\prime}  \bfx) \big\{ [ \bfM  - ( \bfM - \bmGamma) \bfv{\rm \bf e}_d^{\prime}]^{\prime}  \bfQ  +  \bfQ  [ \bfM  - ( \bfM - \bmGamma) \bfv{\rm \bf e}_d^{\prime}] \big\}   \\
				&\mathrel{\phantom{=}} +
				(1-\dot{\rho_{\e}}({\rm \bf e}_d^{\prime}  \bfx))[ \bfM ^{\prime} \bfQ  +  \bfQ  \bfM ].
			\end{aligned}
		\end{equation}
		It follows from Lemma \ref{lem:Q} that the matrices
		\begin{equation*}
			[ \bfM  - ( \bfM - \bmGamma) \bfv{\rm \bf e}_d^{\prime}]^{\prime}  \bfQ  +  \bfQ  [ \bfM  - ( \bfM - \bmGamma) \bfv{\rm \bf e}_d^{\prime}]
			\quad \text{ and } \quad
			[ \bfM ^{\prime} \bfQ  +  \bfQ  \bfM ]
		\end{equation*}
		are both positive definite. By the fact $\dot{\rho_{\e}}({\rm \bf e}_d^{\prime}  \bfx) \in [0,1]$, we know that the matrix $[-\nabla g_{\e}( \bfx)]^{\prime}  \bfQ  +  \bfQ  [-\nabla g_{\e}( \bfx)]$ with $|{\rm \bf e}_d^{\prime}  \bfx| \leq \e$ is also positive definite.
		
		Combining above {\bf Cases 1, 2} and {\bf 3}, one can derive the desired result.
		
		(ii) According to the definitions of $\lambda_1$ and $\lambda_2$,
		we know that
		\begin{equation*}
			\begin{aligned}
				-\lambda_1
				&= \lambda_{\max}([- \bfM ^{\prime}] \bfQ + \bfQ [- \bfM ]), \\
				-\lambda_2
				&= \lambda_{\max}((- \bfM ^{\prime}+{\rm \bf e}_d  \bfv^{\prime}( \bfM ^{\prime}- \bmGamma)) \bfQ  + \bfQ (- \bfM +( \bfM - \bmGamma) \bfv{\rm \bf e}_d^{\prime})).
			\end{aligned}
		\end{equation*}
		By the fact $\lambda=\lambda_1 \wedge \lambda_2$ and \eqref{e:Q-Q}, one has
		\begin{equation*}
			-\lambda = (-\lambda_1) \vee  (-\lambda_2)
			=  \sup_{ \bfx\in \R^d} \left\{ \lambda_{\max}([\nabla g_{\e}( \bfx)]^{\prime}  \bfQ  +  \bfQ  [\nabla g_{\e}( \bfx)])  \right\}.
		\end{equation*}
		
		(iii) Applying $-\lambda$ above,  one can obtain that for any $ \bfx,\bfy \in \R^d$,
		\begin{equation*}
			\begin{aligned}
				\langle  \bfQ ( \bfx-\bfy),g_{\e}(\bfx)-g_{\e}(\bfy)   \rangle
				&= \langle  \bfQ (\bfx-\bfy),\int_0^{1} \nabla g_{\e}(\bfy+r(\bfx-\bfy)) [\bfx-\bfy] \dif r \rangle \nonumber \\
				&= \int_0^1 [\bfx-\bfy]^{\prime}  \bfQ  \nabla g_{\e}(\bfy+r(\bfx-\bfy)) [\bfx-\bfy] \dif r
				\leq -\frac{\lambda}{2}|\bfx-\bfy|^2.
			\end{aligned}
		\end{equation*}
		
		(iv) Recall that
		\begin{equation*}
			g( \bfx) = \bmell- \bfM ( \bfx-\langle {\rm \bf e}_d, \bfx \rangle^+  \bfv)
			-\langle {\rm \bf e}_d, \bfx \rangle^+  \bmGamma  \bfv.
		\end{equation*}
		It follows from Lemma \ref{lem:Q} that for $ {\rm \bf e}_d^{\prime}  \bfx > 0$
		\begin{equation*}
			\langle g(\bfx),\bfQ \bfx  \rangle
			= \langle \bmell-( \bfM + ( \bmGamma- \bfM ) \bfv{\rm \bf e}_d^{\prime}) \bfx, \bfQ \bfx \rangle
			\ \leq \ -\frac{\lambda_2}{2}|\bfx|^2+|\bmell| \lambda_{\max}(\bfQ)|\bfx|.
		\end{equation*}
		In addition, it follows from Lemma \ref{lem:Q} that for $ {\rm \bf e}_d^{\prime}  \bfx \leq 0$
		\begin{equation*}
			\langle g(\bfx),\bfQ \bfx  \rangle
			= \langle \bmell-\bfM \bfx, \bfQ \bfx \rangle
			\ \leq \ -\frac{\lambda_1}{2}|\bfx|^2+|\bmell| \lambda_{\max}(\bfQ)|\bfx|.
		\end{equation*}
		Combining above two cases and the fact $\lambda=\lambda_1 \wedge \lambda_2$, the desired inequality holds. The proof is complete.
	\end{proof}

	\subsection{Proof of Lemma \ref{lem:JF-est}}\label{app_A2}
	
	\begin{proof}[Proof of Lemma \ref{lem:JF-est}.]
		The inequality holds obviously for $\bfu={\bf 0}$ due to $\nabla_\bfu  \bfX_t^{\e, \bfx}={\bf 0}$, $\forall\ t\geq 0$. Thus, it suffices to show the inequality for $\bfu\in \R^d_0$. For the matrix $ \bfQ $ in Lemma \ref{lem:QNge}, define a function $\widehat{V}(\bfx): \bbR^d \to \bbR$ as
		$$
		\widehat{V}( \bfx) = \frac{1}{2}\langle  \bfx, \bfQ  \bfx \rangle.
		$$
		Using It\^{o}'s formula, it follows from Lemma \ref{lem:QNge} that
		\begin{equation*}
			\begin{aligned}
				\dif \widehat{V}(\nabla_\bfu  \bfX_t^{\e, \bfx})
				&= \langle  \bfQ  \nabla_\bfu  \bfX_t^{\e, \bfx}, \nabla g_{\e}( \bfX^{\e, \bfx}_t) \nabla_\bfu  \bfX_t^{\e, \bfx} \rangle \dif t \\
				&=  [\nabla_\bfu  \bfX_t^{\e, \bfx}]^{\prime} \frac{ \bfQ  \nabla g_{\e}( \bfX^{\e, \bfx}_t)+[\nabla g_{\e}( \bfX^{\e, \bfx}_t)]^{\prime} \bfQ }{2} \nabla_\bfu  \bfX_t^{\e, \bfx}  \dif t  \\
				&\leq -\frac{\lambda}{2} \frac{ |\nabla_\bfu  \bfX_t^{\e, \bfx}|^2 }{\widehat{V}(\nabla_\bfu  \bfX_t^{\e, \bfx})} \widehat{V}(\nabla_\bfu  \bfX_t^{\e, \bfx}) \dif t \\
				&\leq -\frac{\lambda}{\lambda_{\max}( \bfQ )} \widehat{V}(\nabla_\bfu  \bfX_t^{\e, \bfx}) \dif t,
			\end{aligned}
		\end{equation*}
		where the first inequality holds from the fact $\nabla_\bfu  \bfX_t^{\e, \bfx}\neq {\bf 0}$, therefor $\widehat{V}(\nabla_\bfu  \bfX_t^{\e, \bfx})>0$ for all $\bfu\neq {\bf 0}$. Hence, one derives that
		\begin{equation*}
			\widehat{V}(\nabla_\bfu  \bfX_t^{\e, \bfx})
			\leq \exp\left\{ -\frac{\lambda}{\lambda_{\max}( \bfQ )} t \right\} \widehat{V}(\bfu).
		\end{equation*}
		The definition of $\widehat{V}( \bfx)$ leads that $\lambda_{\min}( \bfQ ) | \bfx|/2 \leq  \widehat{V}( \bfx) \leq \lambda_{\max}( \bfQ )| \bfx|/2$. Thus, one can obtain
		\begin{equation*}
			\frac{1}{2}\lambda_{\min}( \bfQ )|\nabla_\bfu  \bfX_t^{\e, \bfx}|^2
			\leq
			\frac{1}{2} \lambda_{\max}( \bfQ )
			\exp\left\{ -\frac{\lambda}{\lambda_{\max}( \bfQ )} t \right\} |\bfu|^2,
		\end{equation*}
		which implies the desired result immediately.
	\end{proof}
	
	\subsection{Proof of Lemma \ref{l:XeCon}}\label{app_A3}
	
	\begin{proof}[Proof of Lemma \ref{l:XeCon}]
		We claim that
		\begin{equation} \label{e:PN=0}
			\PP(N) = 0, \quad \text{with} \quad 	
			N =	\left\{
			\int_{0}^{\infty} \|\nabla g_{\e}(\bfX^{\e,\bfx}_{s})\|_{ {\rm op} } \I_{\{{\rm  \bf e}_d^{\prime}  \bfX^{\bfx}_{s}=0\}} \dif s
			\neq 0
			\right\}.
		\end{equation}
		Indeed, for any $T>0$, by $\|\nabla g_{\e}(\bfx)\|_{{\rm op}} \leq C_{{\rm op}}$ for all $\bfx \in \R^d$ and $\e\in(0,1)$,  we have
		\begin{equation}\label{e:ET}
			\begin{aligned}
				\E \int_{0}^{T}
				\|\nabla g_{\e}(\bfX^{\e,\bfx}_{s})\|_{ {\rm op} } \I_{\{{\rm \bf e}_d^{\prime}  \bfX_{s}^{\bfx} =0 \}} \dif s
				&= \int_{0}^{T} \E[\|\nabla g_{\e}(\bfX^{\e,\bfx}_{s})\|_{{\rm op}} \I_{\{{\rm \bf e}_d^{\prime}  \bfX_{s}^{\bfx} =0 \}}] \dif s   \\
				& \leq    \int_{0}^{T}  C_{{\rm op}}  \E  \I_{\{{\rm \bf e}_d^{\prime} \bfX_{s}^{\bfx} =0 \}} \dif s
				\ = \ 0.
			\end{aligned}
		\end{equation}
		Since \eqref{e:ET} holds for all $T>0$, we see that
		\begin{equation*}
			\E \int_{0}^{\infty} \|\nabla g_{\e}(\bfX^{\e,\bfx}_{s})\|_{ {\rm op} } \I_{\{{\rm \bf e}_d^{\prime} \bfX^{\bfx}_{s} =0 \}} \dif s
			= 0,
		\end{equation*}
		hence \eqref{e:PN=0} holds.
		
		Recall the definition of $\bfJ^{\e,\bfx}_{s,t}$ and define
		\begin{equation*}
			\hat{\bfJ}^{\e,\bfx}_{s,t}
			:= \exp \left(\int_s^t \nabla g_{\e} (\bfX_{r}^{\e,\bfx})  \I_{\{{\rm \bf e}_d^{\prime}  \bfX^{\bfx}_{r} \neq 0 \}}\dif r \right).
		\end{equation*}
		It is easy to verify that
		\begin{equation} \label{e:JstConN}
			\lim_{\e \rightarrow 0} \hat{\bfJ}^{\e,\bfx}_{s,t}
			=\bfJ^{\bfx}_{s,t}, \quad
			0 \leq s \leq t < \infty.
		\end{equation}
		On the event $N^c$, we know $\int_{0}^{\infty} \nabla g_{\e}(\bfX^{\e,\bfx}_{s}) \I_{\{{\rm \bf e}_d^{\prime} \bfX^{\bfx}_{s} =0 \}} \dif s={\bf 0}$ and thus
		\begin{equation*}
			\exp\left(\int_{s}^{t} \nabla g_{\e}(\bfX^{\e,\bfx}_{r}) \I_{\{{\rm \bf e}_d^{\prime} \bfX^{\bfx}_{r} = 0 \}} \dif r \right)={\bf I},
			\quad 0 \leq s \leq t < \infty.
		\end{equation*}
		Since ${\bf I}$ commutes with any matrix, on the event $N^c$, we get
		\begin{equation*}
			\hat{\bfJ}^{\e,\bfx}_{s,t}
			= \hat{\bfJ}^{\e,\bfx}_{s,t} \exp \left(\int_{s}^{t} \nabla g_{\e}(\bfX^{\e,\bfx}_{r})  \I_{\{{\rm \bf e}_d^{\prime} \bfX^{\bfx}_{r} = 0 \}} \dif r \right)
			\ = \ \bfJ^{\e,\bfx}_{s,t},
			\quad 0 \leq s \leq t < \infty.
		\end{equation*}
		This, combining with \eqref{e:JstConN}, implies that
		\begin{equation*}
			\lim_{\e \rightarrow 0}  \bfJ^{\e,\bfx}_{s,t} = \bfJ^{\bfx}_{s,t}, \quad 0 \leq s \leq t < \infty, \quad \text{On the event }\ N^c.
		\end{equation*}
		Note that $\bfJ^{\e,\bfx}_{s,t}$ and $\bfJ^{\bfx}_{s,t}$ are matrices, hence the above pointwise convergence implies the convergence in operator.
	\end{proof}
	
	\subsection{Proof of Lemma \ref{lem:moment-est}}\label{app_A5}
	
	\begin{proof}[Proof of Lemma \ref{lem:moment-est}.]
		(i) For the positive definite matrix $ \bfQ $ in Lemma \ref{lem:Q}, define a function $V: \R^d \to [1,\infty)$ as
		\begin{equation*}
			V( \bfx) = \big(1+\langle  \bfx,  \bfQ  \bfx \rangle \big)^{{1}/{2}}, \quad  \bfx\in \R^d .
		\end{equation*}
		Then, we obtain
		\begin{equation}\label{lem:moment-est-p1}
			\lambda^{1/2}_{\min}( \bfQ )| \bfx| \vee 1 \leq V( \bfx) \leq 1+ \lambda^{1/2}_{\max}( \bfQ )| \bfx|
		\end{equation}
		and
		\begin{equation*}
			\nabla V( \bfx)
			=  \frac{1}{V( \bfx)} \bfQ  \bfx \ ,
			\quad
			\nabla^2 V( \bfx)
			=  \frac{-1}{V^3( \bfx)} \bfQ  \bfx \bfx^{\prime} \bfQ
			+\frac{1}{V( \bfx)} \bfQ ,
			\quad
			\forall  \bfx\in \R^d.
		\end{equation*}
		It is easy to check that $\|\nabla V \|_{\infty} \leq 1$ and $\| \nabla^2 V \|_{\rm HS,\infty}\leq 1$. We will find a differential inequality of $\E V( \bfX_t^{\e, \bfx})$ firstly, and split the proof into three parts based on the type of the noises $(\bfZ_t)_{t\geq 0}$.
		
		We shall prove the following Lyapunov condition in three cases: there exists some positive constant $C$ independent of $\e$ such that
		\begin{equation}\label{e:Lya-Con}
			\mathcal{A}_{\e}^{\alpha}V( \bfx)
			\leq -\frac{\lambda}{2\lambda_{\max}( \bfQ )}V( \bfx)+C.
		\end{equation}	
		{\bf Case 1: $\alpha=2$.} Equation  \eqref{e:Ae} yields that
		\begin{equation}\label{e:AeV-BM}
			\mathcal{A}_{\e}^{\alpha} V( \bfx) = \langle g_{\e}( \bfx),\nabla V( \bfx) \rangle + \frac{1}{2} \langle \nabla^2 V( \bfx), \bmsigma \bmsigma^{\prime} \rangle_{\rm HS} .
		\end{equation}
		For the first term, by a straightforward calculation and  \eqref{e:dis-con}, there exists some positive constant $C$ independent of $\e$ such that
		\begin{equation}\label{e:gV-BM}
			\begin{aligned}
				\langle g_{\e}( \bfx),\nabla V( \bfx) \rangle
				&= \frac{1}{V( \bfx)} \langle  \bfQ  \bfx, g_{\e}( \bfx)-g_{\e}(\mathbf{0}) \rangle
				+\frac{1}{V( \bfx)} \langle  \bfQ  \bfx, g_{\e}(\mathbf{0}) \rangle   \\
				&\leq -\frac{\lambda}{2}V^{-1}( \bfx)| \bfx|^2
				+\lambda_{\max}( \bfQ )|g_{\e}(\mathbf{0})| \frac{| \bfx|}{V( \bfx)}  \\
				&\leq -\frac{\lambda}{2\lambda_{\max}( \bfQ )}V^{-1}( \bfx)\langle  \bfx,  \bfQ  \bfx \rangle+\lambda_{\max}( \bfQ )|g_{\e}(\mathbf{0})| \frac{| \bfx|}{V( \bfx)}  \\
				&\leq -\frac{\lambda}{2\lambda_{\max}( \bfQ )}V( \bfx)+C.
			\end{aligned}
		\end{equation}
		As for the second term, it is easy to check that by Cauchy inequality
		\begin{equation}\label{e:V_BM}
			\langle \nabla^2 V( \bfx), \bmsigma \bmsigma^{\prime} \rangle_{\rm HS}
			\leqslant \| \nabla^2 V \|_{\rm HS,\infty} \cdot \| \bmsigma\bmsigma^{\prime} \|_{\rm HS}
			\leqslant   \sqrt{d}\| \bmsigma \|_{\rm op}  \leqslant  C.
		\end{equation}
		Combining \eqref{e:AeV-BM}, \eqref{e:gV-BM} and \eqref{e:V_BM}, the Lyapunov condition \eqref{e:Lya-Con} immediately follows.
		
		{\bf Case 2: $(\bfZ_t)_{t\geq 0}$ being the cylindrical stable noise  with  $\alpha\in (1,2)$.} Combining with \eqref{e:Ae}, one knows
		\begin{equation}\label{e:AeV}
			\begin{aligned}
				\mathcal{A}_{\e}^{\alpha} V( \bfx)
				&= \langle g_{\e}( \bfx),\nabla V( \bfx) \rangle + \sum_{i=1}^d \int_{\R_0 } \big[ V( \bfx+\bmsigma_i z_i)-V( \bfx) - \langle \bmsigma_i , \nabla  V( \bfx) \rangle z_i \I_{ \{ |z_i| \leq 1 \} }   \big] \frac{c_{\alpha}}{|z_i|^{1+\alpha}}  \dif z_i   \\
				&= \langle g_{\e}( \bfx),\nabla V( \bfx) \rangle
				+ \sum_{i=1}^d \int_{|z_i|> 1} \big[ V( \bfx+\bmsigma_i z_i)-V( \bfx)  \big] \frac{c_{\alpha}}{|z_i|^{1+\alpha}} \dif z_i   \\
				&\mathrel{\phantom{=}} + \sum_{i=1}^d \int_{0<|z_i|\leq 1} \big[ V( \bfx+\bmsigma_i z_i)-V( \bfx) - \langle \bmsigma_i ,\nabla V( \bfx) \rangle z_i \big]  \frac{c_{\alpha}}{|z_i|^{1+\alpha}} \dif z_i.
			\end{aligned}
		\end{equation}
		The first term $\langle g_{\e}(\bfx),\nabla V(\bfx) \rangle$ is bounded by \eqref{e:gV-BM}. For the second term of equality \eqref{e:AeV} and any $i= 1,\dotsc,d$, by the Taylor's expansion and absolute value inequality, one has
		\begin{equation*}
			\begin{aligned}
				\left| \int_{|z_i|> 1} [ V( \bfx+\bmsigma_i z_i)-V( \bfx) ] \frac{c_{\alpha}}{|z_i|^{1+\alpha}} \dif z_i \right|
				&= \left| \int_{|z_i|> 1} \int_0^1 \langle \nabla V( \bfx+r\bmsigma_i z_i), \bmsigma_i z_i\rangle  \dif r  \frac{c_{\alpha}}{|z_i|^{1+\alpha}} \dif z_i \right| \nonumber \\
				&\leq  c_{\alpha} \|\nabla V \|_{\infty} \| \bmsigma_i\|_{\rm max} \int_{|z_i|>1} \frac{ |z_i| }{|z_i|^{1+\alpha}} \dif z_i   \leqslant   C.
			\end{aligned}	
		\end{equation*}
		Similarly, for the third term of equality \eqref{e:AeV}, we know
		\begin{equation*}
			\begin{aligned}
				&\mathrel{\phantom{=}} \left|\int_{0<|z_i|\leq 1} [ V( \bfx+\bmsigma_i z_i)-V( \bfx) - \langle \bmsigma_i ,\nabla V( \bfx) z_i \rangle ] \frac{c_{\alpha}}{|z_i|^{1+\alpha}} \dif z_i \right| \nonumber \\
				&= \left| c_{\alpha}\int_{0<|z_i|\leq 1}  \int_0^1 \int_0^r \langle  \bmsigma_i \bmsigma_i^{\prime} , \nabla^2 V( \bfx+ s\bmsigma_i z_i) \rangle_{\rm HS} \frac{ z_i^2}{|z_i|^{1+\alpha}} \dif s \dif r   \dif z_i \right| \\
				&\leq c_{\alpha} \| \nabla^2 V \|_{{\rm HS}, \infty} \| \bmsigma_i\|_{\rm max}^2 \int_{0<|z_i|\leq 1} |z_i|^{1-\alpha} \dif z_i  \leq  C \nonumber .
			\end{aligned}
		\end{equation*}
		Hence, we can obtain that
		\begin{equation} \label{e:V}
			\sum_{i=1}^d \left| \int_{\R_0} \big[ V( \bfx+\bmsigma_i z_i)-V( \bfx)  - \langle \bmsigma_i , \nabla  V( \bfx) \rangle z_i \I_{ \{ |z_i| \leq 1 \} }   \big] \frac{c_{\alpha}}{|z_i|^{1+\alpha}}  \dif z_i \right|   \leqslant  Cd .
		\end{equation}
		Combining  \eqref{e:gV-BM}, \eqref{e:AeV} and \eqref{e:V}, the Lyapunov condition \eqref{e:Lya-Con} immediately follows.
		
		{\bf Case 3:  $(\bfZ_t)_{t\geq 0}$ being the rotationally symmetric stable noise with $\alpha \in (1,2)$.} By \eqref{e:Ae rotational}, we obtain
		\begin{equation}\label{e:AeV rotational}
			\begin{aligned}
				\mathcal{A}_{\e}^{\alpha} V(\bfx)
				&= \langle g_{\e}(\bfx),\nabla V(\bfx) \rangle + \int_{\R^d_0 } [ V(\bfx+\bmsigma \bfz)-V(\bfx) - \langle \bmsigma \bfz, \nabla  V(\bfx) \rangle  \I_{ \{ |\bfz| \leq 1 \} }   ] \frac{C_{d,\alpha}}{|\bfz|^{d+\alpha}}  \dif \bfz   \\
				&= \langle g_{\e}(\bfx),\nabla V(\bfx) \rangle
				+ \int_{|\bfz|> 1} [ V(\bfx+\bmsigma \bfz)-V(\bfx)  ] \frac{C_{d,\alpha}}{|\bfz|^{d+\alpha}} \dif \bfz   \\
				&\mathrel{\phantom{=}} + \int_{0<|\bfz|\leq 1} [ V(\bfx+\bmsigma \bfz)-V(\bfx) - \langle \bmsigma \bfz,\nabla V(\bfx) \rangle ] \frac{C_{d,\alpha}}{|\bfz|^{d+\alpha}} \dif \bfz.
			\end{aligned}
		\end{equation}	
		The first term $\langle g_{\e}(\bfx),\nabla V(\bfx) \rangle$ is bounded by \eqref{e:gV-BM}. For the second term, by the Taylor's expansion, we have
		\begin{equation}\label{e:V-1}
			\begin{aligned}
				\left|\int_{|\bfz|> 1} [ V(\bfx+\bmsigma \bfz)-V(\bfx) ] \frac{C_{d,\alpha}}{|\bfz|^{d+\alpha}} \dif \bfz\right|
				&\leqslant \int_{|\bfz|> 1} \int_0^1 \big| \langle \nabla V(\bfx+r\bmsigma \bfz), \bmsigma \bfz\rangle \big| \dif r  \frac{C_{d,\alpha}}{|\bfz|^{d+\alpha}} \dif \bfz \\
				&\leq \|\nabla V \|_{\infty} \| \bmsigma\|_{\rm op} \int_{|\bfz|>1} \frac{ C_{d,\alpha} |\bfz| }{|\bfz|^{d+\alpha}} \dif \bfz   \leqslant   C.
			\end{aligned}
		\end{equation}
		Similarly, for the third term, we obtain
		\begin{equation}\label{e:V-2}
			\begin{aligned}
				&\mathrel{\phantom{=}}
				{\left|\int_{0<|\bfz|\leq 1} [ V(\bfx+\bmsigma \bfz)-V(\bfx) - \langle \bmsigma \bfz,\nabla V(\bfx) \rangle ] \frac{C_{d,\alpha}}{|\bfz|^{d+\alpha}} \dif \bfz\right|}  \\
				&\leqslant \int_{0<|\bfz|\leq 1}  \int_0^1 \int_0^r \big| \langle  \bmsigma \bfz, \nabla^2 V(\bfx+ s\bmsigma \bfz)\bmsigma \bfz \rangle \big| \frac{C_{d,\alpha}}{|\bfz|^{d+\alpha}} \dif s \dif r   \dif \bfz  \\
				&\leq \| \nabla^2 V \|_{{\rm HS},\infty} \| \bmsigma \|_{\rm op}^2  \int_{0<|\bfz|\leq 1} \int_0^1 \int_0^r \frac{C_{d,\alpha} |\bfz|^2}{|\bfz|^{d+\alpha}} \dif s \dif r \dif \bfz  \leq C.
			\end{aligned}
		\end{equation}
		Combining \eqref{e:gV-BM}, \eqref{e:AeV rotational}, \eqref{e:V-1} and \eqref{e:V-2}, the Lyapunov condition \eqref{e:Lya-Con} immediately follows.
		
		Using the It\^{o}'s formula, we have
		\begin{equation*}
			\E V( \bfX_t^{\e, \bfx})
			= V( \bfx) + \E \int_0^t  \mathcal{A}_{\e}^{\alpha} V( \bfX_s^{\e, \bfx}) \dif s
			\leq V( \bfx) + \int_0^t \left[-\frac{\lambda}{2\lambda_{\max}( \bfQ )} \E V( \bfX_s^{\e, \bfx}) + C \right]  \dif s.
		\end{equation*}
		It follows from \cite[pp. 32-33]{Gurvich2014Diffusion}  that
		\begin{equation}\label{e:EV}
			\E V( \bfX_t^{\e, \bfx})
			\leq   \rme^{-\frac{\lambda t}{2\lambda_{\max}( \bfQ )} } V( \bfx) +\frac{C \lambda_{\max}( \bfQ )}{\lambda} \Big(1- \rme^{-\frac{\lambda t}{2\lambda_{\max}( \bfQ )}} \Big).
		\end{equation}
		Combining this with the fact in \eqref{lem:moment-est-p1}, we  get the estimate \eqref{e:Xmon} for $\E| \bfX_{t}^{\e, \bfx}|$. With similar discussions, we also get the moment estimate for $  \E|\bfX_t^{\bfx}|$ in \eqref{e:Xmon}.
		
		Let $ \bfX_0^{\e}$ follow the ergodic measure $\mu_{\e}$, then $ \bfX_t^{\e}$ follows measure $\mu_{\e}$ for all $t\geq 0$ and  the inequality  \eqref{e:EV} implies that for all $t>0$,
		\begin{equation*}
	\mu_{\e}(V)
	\leq   \rme^{-\frac{\lambda t}{2\lambda_{\max}( \bfQ )} }  \mu_{\e}(V) +\frac{C \lambda_{\max}( \bfQ )}{\lambda} \Big(1- \rme^{-\frac{\lambda t}{2\lambda_{\max}( \bfQ )}} \Big) .
		\end{equation*}
		Thus, we obtain $\mu_{\e}(V)\leq C$, and so does $\mu_{\e}(|\bfx|)$, while the constant $C$ is independent of $\e$.
		
		(ii) For any $s\in(t_{i-1}, t_i]$ and $\alpha \in (1,2]$, one has
		\begin{equation*}
			\bfX_{t_{i-1},s}^{\e, \bfx} =  \bfx+ \int_{t_{i-1}}^s g_{\e}( \bfX_{t_{i-1},u}^{\e, \bfx}) \dif u+ \bmsigma  \bfZ_{s-t_{i-1}}.
		\end{equation*}
		It follows from the estimate for $\E| \bfX_t^{\e, \bfx}|$ above and the linear growth for function $g_{\e}$, we know for any $s\in(t_{i-1},t_i]$,
		\begin{equation}\label{e:EXs}
			\begin{aligned}
				\E| \bfX_{t_{i-1},s}^{\e, \bfx}- \bfx| &\leq C \int_{t_{i-1}}^s \E |  \bfX_{t_{i-1},u}^{\e, \bfx} | \dif u + \| \bmsigma \|_{\rm op} (s-t_{i-1})^{{1}/{\alpha}} \E| \bfZ_1| \\
				&\leqslant C(1+| \bfx|)[(s-t_{i-1})^{{1}/{\alpha}} \lor (s-t_{i-1})]
				\leqslant  C(1+| \bfx|) \eta_{i}^{1/\alpha},
			\end{aligned}
		\end{equation}
		where the last inequality holds since $s-t_{i-1} \le \eta_i< 1$.
		
		On the other hand, we have
		\begin{equation*}
			\bfY_{t_{i-1},t_i}^{\e, \bfx} =  \bfx+ \int_{t_{i-1}}^{t_i} g_{\e}( \bfx) \dif s+\bmsigma   \bfZ_{t_i-t_{i-1}},
		\end{equation*}
		which leads to
		\begin{equation*}
			\bfX_{t_{i-1},t_i}^{\e, \bfx}- \bfY_{t_{i-1},t_i}^{\e, \bfx}
			= \int_{t_{i-1}}^{t_i} [g_{\e}( \bfX_{t_{i-1},s}^{\e, \bfx})-g_{\e}( \bfx)] \dif s.
		\end{equation*}
		Combining \eqref{e:EXs} and the global Lipschitz property for function $g_{\e}$, we know
		\begin{equation}\label{e:EX-Y}
			\E | \bfX_{t_{i-1},t_i}^{\e, \bfx}- \bfY_{t_{i-1},t_i}^{\e, \bfx} |
			\leq  C \int_{t_{i-1}}^{t_i} \E | \bfX_{t_{i-1},s}^{\e, \bfx}- \bfx|\dif s
			\leq C (1+| \bfx|)\eta_i^{1+{1}/{\alpha}},
		\end{equation}
		which is the estimate \eqref{e:X-Ymon} for $\E | \bfX_{t_{n-1},t_n}^{\e, \bfx}- \bfY_{t_{n-1},t_n}^{\e, \bfx} |$.
		
		(iii) Similar to estimating $\E| \bfX_t^{\e, \bfx}|$, {we give the estimate for $\E| \bfY_{t_n}^{\e, \bfx}|$ with $n\in \mathbb{N}_0$}. Let us now estimate $\E V( \bfY_{t_k,t_{k+1}}^{\e, \bfy})$. We have
		\begin{equation*}
			\begin{aligned}
				\E V( \bfY_{t_k,t_{k+1}}^{\e, \bfy}) &= \E V( \bfY_{t_k,t_{k+1}}^{\e, \bfy}) - \E V( \bfX_{t_k,t_{k+1}}^{\e, \bfy}) +  \E V( \bfX_{t_k,t_{k+1}}^{\e, \bfy}) \\
				&\leq \| \nabla V\|_{\infty} \E | \bfX_{t_k,t_{k+1}}^{\e, \bfy}-  \bfY_{t_k,t_{k+1}}^{\e, \bfy}|
				+\E V( \bfX_{t_k,t_{k+1}}^{\e, \bfy}).
			\end{aligned}
		\end{equation*}
		Using \eqref{e:EV} and \eqref{e:EX-Y}, there exist some positive constants $c$ and $C$ such that
		\begin{equation}\label{e:EV0}
			\begin{aligned}
				\E V( \bfY_{t_k,t_{k+1}}^{\e, \bfy})
				&\leq
				C(1+|\bfy|) \eta_{k+1}^{1+{1}/{\alpha}} + {\rm e}^{-c \eta_{k+1}}V(\bfy) + C (1-{\rm e}^{-c \eta_{k+1}})   \\
				&\leq \big( C\eta_{k+1}^{1+{1}/{\alpha}} + {\rm e}^{-c \eta_{k+1}} \big) V(\bfy)  + C \eta_{k+1}.
			\end{aligned}
		\end{equation}
		
		Due to Assumption \ref{assump-2}, $ (\eta_n)_{n\in \mathbb{N}}$ decreases to $0$, and there exists a positive integer $k_0$ such that for all $k\geq k_0$, one has $ C\eta_{k+1}^{1+1/\alpha} + {\rm e}^{-c \eta_{k+1}}\leq 1-c\eta_{k+1}/2$, Thus,
		\begin{equation*}
			\E V( \bfY_{t_k,t_{k+1}}^{\e, \bfy})
			\leq \Big(1-\frac{c}{2}\eta_{k+1} \Big) V( \bfy) + C\eta_{k+1},
		\end{equation*}
		which implies that
		\begin{equation*}
			\E V( \bfY_{t_{k+1}}^{\e, \bfx})
			= \E\big\{ \E [V( \bfY_{t_{k+1}}^{\e, \bfx})| \bfY_{t_{k}}^{\e, \bfx} ] \big\}
			\leq  \Big(1-\frac{c}{2}\eta_{k+1} \Big) \E V(\bfY_{t_{k}}^{\e, \bfx}) + C\eta_{k+1}.
		\end{equation*}
		By induction, we obtain for any $k\geq k_0$, there exists some positive constant $C$ independent of $k$ satisfying
		\begin{equation}\label{e:Yn1}
			\begin{aligned}
				\E V( \bfY_{t_{k+1}}^{\e, \bfx})
				&\leq \Big( 1-\frac{c}{2}\eta_k \Big)( 1-\frac{c}{2}\eta_{k+1} ) \E V( \bfY_{t_{k-1}}^{\e, \bfx} ) + C \eta_k \Big( 1-\frac{c}{2} \eta_{k+1} \Big) + C \eta_{k+1}  \\
				&\leq \prod_{i=k_0}^{k} \Big(1-\frac{c}{2}\eta_i \Big)
				\E V( \bfY_{t_{k_0}}^{\e, \bfx}) + C\sum_{j=k_0}^{k} \eta_j \prod_{i=j+1}^{k} \Big(1-\frac{c}{2}\eta_i \Big)   \\
				&\leq C\E\big[1+V( \bfY_{t_{k_0}}^{\e, \bfx})\big],
			\end{aligned}
		\end{equation}
		where the last inequality holds from the facts that $\prod_{i=k_0}^{k}(1-c\eta_i/2)$ and $\sum_{j=k_0}^{k} \eta_j \prod_{i=j+1}^{k} (1-c\eta_i/2)$ are bounded.
		
		Due to Assumption \ref{assump-2} and for any $k\leq k_0$, we know $C\eta_{k}^{1+1/\alpha} + {\rm e}^{-c \eta_{k}} \leq 1-c \eta_{k}+C\eta_{k}^{3/2}$. Combining \eqref{e:EV0} and with a similar argument for \eqref{e:Yn1}, we obtain
		\begin{equation*}
			\begin{aligned}
				\E V( \bfY_{t_{k_0}}^{\e, \bfx})
				&\leq \left(1-c\eta_{k_0}
				+C\eta_{k_0}^{{3}/{2}}\right) \E V( \bfY_{t_{k_0-1}}^{\e, \bfx} ) + C\eta_{k_0} \\
				&\leq \left(1-c\eta_{k_0-1}
				+C\eta_{k_0-1}^{{3}/{2}}\right) \left(1-c\eta_{k_0}
				+C\eta_{k_0}^{{3}/{2}}\right)  \E V( \bfY_{t_{k_0-2}}^{\e, \bfx} )  \nonumber \\
				&\mathrel{\phantom{\leq}} + C \left(1-c\eta_{k_0-1}
				+C\eta_{k_0-1}^{{3}/{2}}\right) \eta_{k_0} + C \eta_{k_0} \nonumber \\
				&\leq C(1+V( \bfx)).
			\end{aligned}
		\end{equation*}
		Combining this with \eqref{e:Yn1}, for any $n\in \mathbb{N}_0$, there exists some positive constant $C$ independent of $n$ satisfying
		\begin{equation*}
			\E V( \bfY_{t_n}^{\e, \bfx})
			\leq C(1+V( \bfx)),
		\end{equation*}
		which implies the estimate \eqref{e:Ymon} for $\E| \bfY_{t_n}^{\e, \bfx}|$ by combining the relationship between $V(\bfx)$ and $|\bfx|$.
		
		(iv) Consider $\bfX^{\e,\bfx}_{t}-\bfX_{t}^{\bfx}$, which satisfies the following equation
		\begin{equation*}
			\begin{aligned}
		\frac{\dif}{\dif t} \left(\bfX^{\e,\bfx}_{t}-\bfX_{t}^{\bfx}\right)
				&= g_\e(\bfX^{\e,\bfx}_t)-g(\bfX^{\bfx}_t)
				= g_\e(\bfX^{\e,\bfx}_t)-g_{\e}(\bfX^{\bfx}_t)			+g_{\e}(\bfX^{\bfx}_t)-g(\bfX^{\bfx}_{t}) \\
				&= \nabla g_{\e}({\bmtheta}_{t})  \left(\bfX^{\e,\bfx}_{t}-\bfX_{t}^{\bfx}\right)			+g_{\e}(\bfX^{\bfx}_t)-g(\bfX^{\bfx}_{t}),
			\end{aligned}
		\end{equation*}
		where $\bmtheta_{t}$ is between $\bfX^{\bfx}_t$ and $\bfX^{\e,\bfx}_t$. The above equation can be solved by
		\begin{equation*}
			\bfX^{\e,\bfx}_{t}-\bfX_{t}^{\bfx}
			= \int_{0}^{t} \exp\left(\int_{s}^{t} \nabla g_{\e}(\bmtheta_{r}) \dif r \right) (g_{\e}(\bfX^{\bfx}_s)-g(\bfX^{\bfx}_{s})) \dif s.
		\end{equation*}
		Since $\|\nabla g_{\e}(\bfx)\|_{ {\rm op} } \le C_{{\rm op}}$ for all $\bfx \in \R^d$, the  relation $|g_{\e}(\bfx)-g(\bfx)| \le C_{{\rm op}}\e$ for all $\bfx \in \R^d$ immediately gives \eqref{e:XeCon-1}. The proof is complete.
	\end{proof}
	
	\subsection{Proof of Lemma \ref{lem:eeSDEs}}\label{app_A4}
	
	\begin{proof}[Proof of Lemma \ref{lem:eeSDEs}.]
		This proof is standard, and we just give some sketch of the proof here. With similar calculations for inequality \eqref{e:Lya-Con} in the proof of  Lemma \ref{lem:moment-est}, there exist some positive constants $c$ and $C$ both independent of $\e$ such that
		\begin{equation} \label{e:LyaCon}
			\mathcal{A}^{\alpha}_{\e} \widetilde{V}(\bfx)
			\leq  -c\widetilde{V}(\bfx)+C
			\quad \text{and} \quad
			\mathcal{A}^{\alpha} \widetilde{V}(\bfx)
			\leq  -c\widetilde{V}(\bfx)+C,
		\end{equation}
		which implies that the processes $(\bfX_t)_{t\geq 0}$ and $(\bfX_t^{\e})_{t\geq 0}$ are both exponential ergodicity under $\| \cdot \|_{\rm TV,\widetilde{V}}$. With similar discussions for inequality \eqref{e:EV}, one also gets the moment estimate in inequality \eqref{e:metlV}. The proof is complete.
	\end{proof}

	\section{Proofs of the lemmas in Section \ref{sec:ProofThm1}}
	
	\subsection{Proof of Lemma \ref{pro:erg-Xe}}
	\begin{proof}[Proof of Lemma \ref{pro:erg-Xe}]\label{app_B1}
	For $\alpha\in(1,2]$ and any $ \bfx,\bfy\in \R^d$, we have
		\begin{equation*}
			\dif ( \bfX_t^{\e, \bfx}- \bfX_t^{\e,\bfy}) =
			[g_{\e}( \bfX_t^{\e, \bfx})-g_{\e}( \bfX_t^{\e,\bfy})] \dif t, \quad
			\bfX_0^{\e, \bfx}- \bfX_0^{\e,\bfy} = \bfx-\bfy .
		\end{equation*}
		By the It\^{o}'s formula and \eqref{e:dis-con},  we obtain
		\begin{equation*}
			\begin{aligned}
				\dif | \bfQ ( \bfX_t^{\e, \bfx}- \bfX_t^{\e,\bfy})|^2
				&= 2\langle  \bfQ ( \bfX_t^{\e,\bfx}- \bfX_t^{\e,\bfy}), g_{\e}( \bfX_t^{\e, \bfx})-g_{\e}( \bfX_t^{\e,\bfy}) \rangle \dif t \\
				&\leq -\lambda | \bfX_t^{\e, \bfx}- \bfX_t^{\e,\bfy}|^2 \dif t
				= -\lambda | \bfQ ^{-1}  \bfQ ( \bfX_t^{\e,\bfx}- \bfX_t^{\e,\bfy})|^2 \dif t \\
				&\leq -\lambda \lambda^{-2}_{\max}( \bfQ )
				| \bfQ ( \bfX_t^{\e,\bfx}- \bfX_t^{\e,\bfy})|^2 \dif t ,
			\end{aligned}
		\end{equation*}
		which implies that
		\begin{equation*}
			| \bfQ ( \bfX_t^{\e,\bfx}- \bfX_t^{\e,\bfy})|^2
			\leq \exp\left\{-\lambda \lambda^{-2} _{\max}( \bfQ ) t \right\}
			| \bfQ (\bfx-\bfy)|^2.
		\end{equation*}
		Due to  $\lambda_{\min}( \bfQ )|\bfx-\bfy| \leqslant | \bfQ (\bfx-\bfy)| \leqslant \lambda_{\max}( \bfQ )|\bfx-\bfy|$ for any $\bfx,\bfy\in \R^d$, we get	\begin{equation}\label{e:E(X-Y)}
			| \bfX_t^{\e,\bfx}- \bfX_t^{\e,\bfy}|^2
			\leq
			\frac{\lambda^2_{\max}( \bfQ )}{\lambda^2_{\min}( \bfQ )}
			\exp\left\{-\lambda \lambda^{-2}_{\max}( \bfQ ) t \right\}
			|\bfx-\bfy|^2.
		\end{equation}
		From the definition of Wasserstein-1 distance in \eqref{e:W1} and \eqref{e:E(X-Y)}, we obtain
		\begin{equation*}
			\begin{aligned}
				\mathcal{W}_1 (\mathcal{L}( \bfX_t^{\e, \bfx}),
				\mathcal{L}( \bfX_t^{\e,\bfy}))
				&= \sup_{h\in {\rm Lip}(1)} \{ \E h( \bfX_t^{\e,\bfx})-\E h( \bfX_t^{\e,\bfy}) \}
				\leq  \E| \bfX_t^{\e,\bfx}- \bfX_t^{\e,\bfy}|  \\
				&\leq \frac{\lambda_{\max}( \bfQ )}{\lambda_{\min}( \bfQ )}
				\exp\left\{-\frac{1}{2}\lambda\lambda^{-2}_{\max}( \bfQ )t \right\}
				|\bfx-\bfy|.
			\end{aligned}
		\end{equation*}
		The proof is complete.
	\end{proof}
	
	\subsection{Proof of Lemma \ref{pro:mu-mue}}\label{app_B2}
	For $\alpha\in(1,2]$, $\e\in (0,1)$, and any function $h\in {\rm Lip}(1)$,  we consider the second kind of Stein's equations:
	\begin{equation}\label{e:Poi}
		\mathcal{A}_{\e}^{\alpha} f_{\e,h}(\bfx)
		= h( \bfx)-\mu_{\e}(h) ,
	\end{equation}
	where $\mathcal{A}_{\e}^{\alpha}$ and  $\mu_{\e}$ are the infinitesimal generators and ergodic measures for  the processes $( \bfX_t^{\e})_{t\geq 0}$ in \eqref{e:SDEe} respectively. Then, we can get the following expression and estimates for the solution to the Stein's equations in \eqref{e:Poi}.
	\begin{lemma} \label{pro:fh}
		Let $\alpha\in(1,2]$ and $\e\in (0,1)$. Under Assumptions \ref{assump:M} and \ref{assump-2}, the Stein's equations \eqref{e:Poi} admit a solution $f_{\e,h}$ with the following representation:
		\begin{equation}\label{e:fh}
			f_{\e,h}( \bfx) = -\int_0^{\infty}  \opP_{t}^{\e}[h( \bfx)-\mu_{\e}(h)] \dif t,
		\end{equation}
where $ \opP_{t}^{\e}$ are semigroups defined in \eqref{e:Pe}. Furthermore, there exists some positive constant $C$ independent of $\e$ such that
		\begin{equation*}
			|f_{\e,h}( \bfx)| \leq C(1+| \bfx|),
			\quad \text{and} \quad
			\|\nabla f_{\e,h}\|_{\infty} \leq C.
		\end{equation*}
	\end{lemma}
	
	We will prove this lemma at the end of this section. Applying this lemma, we can prove Lemma \ref{pro:mu-mue} as follow.
	
	\begin{proof}[Proof of Lemma \ref{pro:mu-mue}]
		In this proof, the expectation $\E$ is with respect to measure $\mu$, that is, $\E^{\mu}$, and we omit the superscript. For $g$ in the domain of $\mathcal{A}^{\alpha}$, i.e.,  $\mathcal{A}^{\alpha}g \in \mcl C_{lin}(\R^d)$, we have
		\begin{equation*}
			\begin{aligned}
				\mu( \mathcal{A}^{\alpha}g)& = \mu\left(\lim_{t \rightarrow 0} \frac{\opP_t g-g}{t}\right) = \lim_{t \rightarrow 0} \mu\left(\frac{\opP_t g-g}{t}\right) \\
				&= \lim_{t \rightarrow 0} \frac{\E[g(\bfX_t^{\bf x})]-\mu(g)}{t} = \lim_{t \rightarrow 0} \frac{\mu(g)-\mu(g)}{t} = 0,
			\end{aligned}
		\end{equation*}
		where the second equality is by dominated convergence theorem and the fourth one is by the stationarity of $\bfX_t$. Hence,
		\begin{equation} \label{e:ASta}
\E [\mathcal{A}^{\alpha} f_{\e,h}( \bfX_0)]  = 0.
		\end{equation}
		From the definition of Wasserstein-1 distance, Stein's equations in \eqref{e:Poi} and \eqref{e:ASta}, we obtain that for $\alpha \in (1,2]$,
		\begin{equation*}
			\begin{aligned}
				\mathcal{W}_1 (\mu,\mu_{\e})
				&= \sup_{h\in {\rm Lip}(1)} \{\E h( \bfX_0)-\mu_{\e}(h)  \}
				=  \sup_{h\in {\rm Lip}(1)} \{\E \mathcal{A}_{\e}^{\alpha} f_{\e,h}( \bfX_0) \} \\
				&= \sup_{h\in {\rm Lip}(1)} \{\E [(\mathcal{A}_{\e}^{\alpha}-\mathcal{A}^{\alpha}) f_{\e,h}( \bfX_0)] \} \\
				&= \sup_{h\in {\rm Lip}(1)} \{\E [\langle g_{\e}( \bfX_0)-g( \bfX_0),\nabla f_{\e,h}( \bfX_0)\rangle] \} \\
				&\leq C\e \mathbb{E} \I_{ \{ |{\rm \bf e}_d^{\prime}  \bfX_0| \leq \e \} }
				\leq  C\e,
			\end{aligned}
		\end{equation*}
		where the first inequality holds from the estimate for $\|\nabla f_{\e,h}\|_{\infty}$ in Lemma  \ref{pro:fh} and \eqref{e:ge-g}. The proof is complete.
	\end{proof}

	\begin{proof}[Proof of Lemma \ref{pro:fh}.]
		(i) For any $h \in {\rm Lip}(1)$, it follows from Lemma \ref{pro:erg-Xe} that
		\begin{equation*}
			\left| \int_{0}^{\infty}  \opP^{\e}_t[h( \bfx)-\mu_{\e}(h)] \dif t \right|
			\leqslant
			\int_{0}^{\infty}
			\mathcal{W}_1 ( \mcl{L}( \bfX_t^{\e, \bfx}),\mu_{\e}) \dif t
		\end{equation*}
		is well defined and
		\begin{equation*}
			|f_{\e,h}( \bfx)| \leq  \int_{0}^{\infty} \left|  \opP^{\e}_t[h( \bfx)-\mu_{\e}(h)] \right| \dif t
			\leq  \int_{0}^{\infty} C(1+| \bfx|)  \rme^{-\theta t}\dif t
			\leq   C(1+| \bfx|).
		\end{equation*}
		The expression for $f_{\e,h}$ in  \eqref{e:fh} can be proved similarly as that in \cite[Proposition 6.1]{Fang2019Multivariate}. We omit the details here.

		(ii) For any $\bfu\in \R^d_0$, we have
		\begin{equation*}
			\nabla_\bfu f_{\e,h}( \bfx)
			= \int_0^{\infty} \nabla_\bfu \E h( \bfX_t^{\e, \bfx}) \dif t
			=  \int_0^{\infty}  \E[\nabla h( \bfX_t^{\e, \bfx}) \nabla_\bfu  \bfX_t^{\e, \bfx}] \dif t.
		\end{equation*}
		It follows from the estimate for $|\nabla_\bfu  \bfX_t^{\e, \bfx}|$ in Lemma \ref{lem:JF-est} and $\| \nabla h \|_{\infty} \leq 1$ that
		\begin{equation*}
			|\nabla_\bfu f_{\e,h}( \bfx)|
			\leq \int_0^{\infty}   \left( \frac{\lambda_{\max}( \bfQ )}{\lambda_{\min}( \bfQ )} \right)^{{1}/{2}}		\exp\left\{-\frac{\lambda}{2\lambda_{\max}( \bfQ )}t \right\}
			|\bfu| \dif t
			\leq 2  \left( \frac{\lambda^3_{\max}( \bfQ )}{\lambda^2 \lambda_{\min}( \bfQ )} \right)^{{1}/{2}}|\bfu|,
		\end{equation*}
		which implies the desired result.
	\end{proof}

	\subsection{Proof of Lemma \ref{pro:2EM}}\label{app_B3}
	\begin{proof}[Proof of Lemma \ref{pro:2EM}.]
		
		For  positive definite matrix $ \bfQ $ in Lemma \ref{lem:Q}, recall that
		\begin{equation*}
			\widehat{V}( \bfx) = \frac{1}{2} \langle  \bfx,  \bfQ  \bfx \rangle, \quad \forall  \bfx\in \R^d.
		\end{equation*}
		Then, one has
		\begin{equation}\label{e:XV}
			\frac{1}{2}\lambda_{\min}( \bfQ )| \bfx|^2
			\leq  \widehat{V}( \bfx)
			\leq   \frac{1}{2} \lambda_{\max}( \bfQ )| \bfx|^2.
		\end{equation}
		
		It follows from those two EM schemes that for any $n\in \mathbb{N}_0$
		\begin{equation*}
			\bfY^{\e, \bfx}_{t_{n+1}}- \bfY^{ \bfx}_{t_{n+1}}
			=  \bfY^{\e, \bfx}_{t_n}- \bfY^{ \bfx}_{t_n} + \eta_{n+1}[g_{\e}( \bfY^{\e, \bfx}_{t_n})-g( \bfY_{t_n}^{ \bfx})] .
		\end{equation*}
		Then, by {a straightforward calculation,} one has
		\begin{equation}\label{e:VYe-Y}
			\begin{aligned}
				\mathrel{\phantom{=}} \widehat{V}( \bfY^{\e, \bfx}_{t_{n+1}}- \bfY^{ \bfx}_{t_{n+1}})
				&=  \frac{1}{2}\langle  \bfY^{\e, \bfx}_{t_{n+1}}- \bfY^{ \bfx}_{t_{n+1}},
				\bfQ [ \bfY^{\e, \bfx}_{t_{n+1}}- \bfY^{ \bfx}_{t_{n+1}}] \rangle   \\
				&= \frac{1}{2} \langle  \bfY^{\e, \bfx}_{t_n}- \bfY^{ \bfx}_{t_n},  \bfQ [ \bfY^{\e, \bfx}_{t_n}- \bfY^{ \bfx}_{t_n}] \rangle
				+\eta_{n+1} \langle  \bfQ [ \bfY^{\e, \bfx}_{t_n}- \bfY^{ \bfx}_{t_n}], g_{\e}( \bfY^{\e, \bfx}_{t_n})-g( \bfY^{ \bfx}_{t_n}) \rangle  \\
				&\mathrel{\phantom{=}} +\frac{1}{2} \eta_{n+1}^2 \langle g_{\e}( \bfY^{\e, \bfx}_{t_n})-g( \bfY_{t_n}^{ \bfx}),
				\bfQ [g_{\e}( \bfY^{\e, \bfx}_{t_n})-g( \bfY_{t_n}^{ \bfx})]
				\rangle.
			\end{aligned}
		\end{equation}
		It suffices to give estimates for the last two terms.
		
		For $\langle  \bfQ [ \bfY^{\e, \bfx}_{t_n}- \bfY^{ \bfx}_{t_n}], g_{\e}( \bfY^{\e, \bfx}_{t_n})-g( \bfY^{ \bfx}_{t_n}) \rangle$, it follows from \eqref{e:ge-g}, \eqref{e:dis-con} and \eqref{e:XV} that there exists some positive constant $C$ independent of $\e$ such that
		\begin{equation}\label{e:VYe-Y-1}
			\begin{aligned}
				&\mathrel{\phantom{=}}
				|\langle  \bfQ [ \bfY^{\e, \bfx}_{t_n}- \bfY^{ \bfx}_{t_n}], g_{\e}( \bfY^{\e, \bfx}_{t_n})-g( \bfY^{ \bfx}_{t_n}) \rangle|  \\
				&= |\langle  \bfQ [ \bfY^{\e, \bfx}_{t_n}- \bfY^{ \bfx}_{t_n}], g_{\e}( \bfY^{\e, \bfx}_{t_n})-g_{\e}( \bfY^{ \bfx}_{t_n})
				+g_{\e}( \bfY^{ \bfx}_{t_n})-g( \bfY^{ \bfx}_{t_n}) \rangle|   \\
				&\leq -\frac{\lambda}{2}| \bfY^{\e, \bfx}_{t_n}- \bfY^{ \bfx}_{t_n}|^2
				+C\e | \bfY^{\e, \bfx}_{t_n}- \bfY^{ \bfx}_{t_n}|  \\
				&\leq -\frac{\lambda}{2}| \bfY^{\e, \bfx}_{t_n}- \bfY^{ \bfx}_{t_n}|^2
				+C\e| \bfY^{\e, \bfx}_{t_n}- \bfY^{ \bfx}_{t_n}|^2
				+C\e  \\
				&\leq \left(-\frac{\lambda}{\lambda_{\max}( \bfQ )}
				+C\e \right)
				\widehat{V}( \bfY^{\e, \bfx}_{t_n}- \bfY^{ \bfx}_{t_n})+C\e,
			\end{aligned}
		\end{equation}
		where the second inequality holds from the Young's inequality, that is, $$\e | \bfY^{\e, \bfx}_{t_n}- \bfY^{ \bfx}_{t_n}|=
		\e^{1/2} \e^{1/2} | \bfY^{\e, \bfx}_{t_n}- \bfY^{ \bfx}_{t_n}| \leq (\e+\e | \bfY^{\e, \bfx}_{t_n}- \bfY^{ \bfx}_{t_n}|^2)/2.$$
		
		For $\langle g_{\e}( \bfY^{\e, \bfx}_{t_n})-g( \bfY^{ \bfx}_{t_n}),
		\bfQ [g_{\e}( \bfY^{\e, \bfx}_{t_n})-g( \bfY^{ \bfx}_{t_n})]
		\rangle$, it follows from \eqref{e:ge-g} and \eqref{e:XV} that there exists some positive constant $C$ independent of $\e$ such that
		\begin{equation}\label{e:VYe-Y-2}
			\begin{aligned}
				&\mathrel{\phantom{=}}
				 |\langle g_{\e}( \bfY^{\e, \bfx}_{t_n})-g( \bfY^{ \bfx}_{t_n}),
				\bfQ [g_{\e}( \bfY^{\e, \bfx}_{t_n})-g( \bfY^{ \bfx}_{t_n})]
				\rangle|  \\
				&\leq \lambda_{\max}( \bfQ )|g_{\e}( \bfY^{\e, \bfx}_{t_n})
				-g( \bfY^{ \bfx}_{t_n})|^2
				\\
				&= \lambda_{\max}( \bfQ )|g_{\e}( \bfY^{\e, \bfx}_{t_n})
				-g_{\e}( \bfY^{ \bfx}_{t_n})+g_{\e}( \bfY^{ \bfx}_{t_n})
				-g( \bfY^{ \bfx}_{t_n})|^2
				\\
				&\leq C| \bfY^{\e, \bfx}_{t_n}- \bfY^{ \bfx}_{t_n}|^2
				+C\e^2  \\
				&\leq C\widehat{V}( \bfY^{\e, \bfx}_{t_n}- \bfY^{ \bfx}_{t_n})+C\e^2.
			\end{aligned}
		\end{equation}
		
		Combining \eqref{e:VYe-Y},\eqref{e:VYe-Y-1} and \eqref{e:VYe-Y-2}, we obtain
		\begin{equation*}
			\widehat{V}( \bfY^{\e, \bfx}_{t_{n+1}}- \bfY^{ \bfx}_{t_{n+1}})
			\leq \left(
			1-\frac{\lambda}{\lambda_{\max}( \bfQ )}\eta_{n+1}
			+C\eta_{n+1} \e+C\eta_{n+1}^2  \right)\widehat{V}( \bfY^{\e, \bfx}_{t_n}- \bfY^{ \bfx}_{t_n})
			+C\eta_{n+1}\e.
		\end{equation*}
		
		Due to Assumption \ref{assump-2}, $(\eta_n)_{n\in \mathbb{N}}$ decreases to $0$. Hence, there exists some positive integer $k_0$ such that for any $n\leq k_0$, we obtain
		\begin{equation*}
			\widehat{V}( \bfY^{\e, \bfx}_{t_{n+1}}- \bfY^{ \bfx}_{t_{n+1}})
			\leq C\widehat{V}( \bfY^{\e, \bfx}_{t_n}- \bfY^{ \bfx}_{t_n})
			+C\eta_{n+1}\e.
		\end{equation*}
		Since $\widehat{V}( \bfY^{\e, \bfx}_{0}- \bfY^{ \bfx}_{0})
		=\widehat{V}({\bf 0})=0$,
		this inequality implies that there exists some positive constant $C$ (depending on $k_0$) independent of $n$ and $\e$ such that
		\begin{equation}\label{e:Vn-1}
			\widehat{V}( \bfY^{\e, \bfx}_{t_{n}}- \bfY^{ \bfx}_{t_{n}})
			\leq C\e, \quad\forall \ n\leq k_0.
		\end{equation}
		In addition, for all $n>k_0$, one has
		\begin{equation*}
			\widehat{V}( \bfY^{\e, \bfx}_{t_{n+1}}- \bfY^{ \bfx}_{t_{n+1}})
			\leq \left(
			1-\frac{\lambda}{2\lambda_{\max}( \bfQ )}
			\eta_{n+1}\right)
			\widehat{V}( \bfY^{\e, \bfx}_{t_n}- \bfY^{ \bfx}_{t_n})
			+C\eta_{n+1}\e.
		\end{equation*}
		By induction, for any $n\geq k_0$,
		\begin{equation*}
			\begin{aligned}
				\widehat{V}( \bfY^{\e, \bfx}_{t_{n+1}}- \bfY^{ \bfx}_{t_{n+1}})
				&\leq \prod_{i=k_0+1}^{n+1}\left(1-\frac{\lambda}
				{2\lambda_{\max}( \bfQ )}\eta_{i}\right)
				\widehat{V}( \bfY^{\e, \bfx}_{t_{k_0}}- \bfY^{ \bfx}_{t_{k_0}}) \\
				& + C\e \sum_{j=k_0+1}^{n+1} \eta_j \prod_{i=j+1}^{n+1}\left(1-\frac{\lambda}
				{2\lambda_{\max}( \bfQ )}\eta_i\right).
			\end{aligned}
		\end{equation*}
		Combining this with \eqref{e:Vn-1} and the fact $1-x\leqslant \exp(-x)$ for all $x\in \R$, there exists some positive constant $C$ independent of $n$ and $\e$ such that for $n>k_0$,
		\begin{equation*}
			\widehat{V}( \bfY^{\e, \bfx}_{t_{n}}- \bfY^{ \bfx}_{t_{n}})
			\leq \exp\left\{-\frac{\lambda}{2\lambda_{\max}( \bfQ )}(t_{n}-t_{k_0})\right\}
			\widehat{V}( \bfY^{\e, \bfx}_{t_{k_0}}- \bfY^{ \bfx}_{t_{k_0}})
			+C\e
			\leq  C\e.
		\end{equation*}
		
		Thus, there exists some positive constant $C$ independent of $n$ and $\e$ such that
		\begin{equation*}
			\widehat{V}( \bfY^{\e, \bfx}_{t_{n}}- \bfY^{ \bfx}_{t_{n}})
			\leq C\e, \quad \forall n\in \mathbb{N}_0,
		\end{equation*}
		which implies that
		\begin{equation*}
			\begin{aligned}
				\mathcal{W}_1 (\mcl{L}( \bfY^{\e, \bfx}_{t_{n}}),\mcl{L}( \bfY^{ \bfx}_{t_{n}})) &= \sup_{h\in {\rm Lip}(1)}  \{ \E h( \bfY^{\e, \bfx}_{t_{n}}) - \E h( \bfY^{ \bfx}_{t_{n}}) \} \\
				&\leq \E| \bfY^{\e, \bfx}_{t_{n}}- \bfY^{ \bfx}_{t_{n}}|
				\leq  C (\E \widehat{V}( \bfY^{\e, \bfx}_{t_{n}}- \bfY^{ \bfx}_{t_{n}}))^{{1}/{2}}
				\leq  C\e^{{1}/{2}}.
			\end{aligned}
		\end{equation*}
		The proof is complete.
	\end{proof}
	
	\subsection{Proof of Lemma \ref{lem:series}}\label{app_B4}
	\begin{proof}[Proof of Lemma \ref{lem:series}.]
		For convenience, we denote
		\begin{equation*}
			u_n = \sum_{i=1}^{n}  \rme^{-\theta(t_n-t_i)} \eta_i^{1+ {1}/{\alpha}} \eta_n^{-{1}/{\alpha}}.
		\end{equation*}
		Then, $u_{n+1}$ and $u_n$ have the following relationship
		\begin{equation*}
			u_{n+1} =  \lambda_n u_n + \eta_{n+1},
		\end{equation*}
		where
		\begin{equation*}
			\lambda_n
			=  \rme^{-\theta \eta_{n+1}} \left( \frac{\eta_n}{\eta_{n+1}} \right)^{{1}/{\alpha}}
			=   \rme^{-\theta \eta_{n+1}} \left(1+\frac{\eta_n^{\beta}
-\eta^{\beta}_{n+1}}{\eta^{1+\beta}_{n+1}} \eta_{n+1} \right)^{{1}/{(\alpha \beta)}}.
		\end{equation*}
		By Assumption \ref{assump-2}, we obtain
		\begin{equation*}
			\lambda_n
			\leq  \rme^{-\theta \eta_{n+1}}(1+\omega \eta_{n+1})^{{1}/{(\alpha \beta)}}
			\leq   \exp\left\{ \Big( -\theta+ \frac{\omega}{\alpha \beta} \Big) \eta_{n+1} \right\},
		\end{equation*}
		which implies that
		\begin{equation*}
			u_{n+1} \leq  \exp\left\{ \Big( -\theta+ \frac{\omega}{\alpha \beta} \Big) \eta_{n+1} \right\} u_n + \eta_{n+1}.
		\end{equation*}
		Thus, multiplying $ \rme^{(\theta-\frac{\omega}{\alpha \beta}) t_{n+1}}$ both sides, we obtain
		\begin{equation*}
			\exp\left\{ \Big(\theta-\frac{\omega}{\alpha \beta} \Big) t_{n+1} \right\} u_{n+1} \leq \exp\left\{ \Big(\theta-\frac{\omega}{\alpha \beta} \Big) t_n \right\} u_n + \exp\left\{ \Big(\theta-\frac{\omega}{\alpha \beta} \Big) t_{n+1} \right\} \eta_{n+1}.
		\end{equation*}
		By induction, we know
		\begin{equation*}
			\exp\left\{ \Big(\theta-\frac{\omega}{\alpha \beta} \Big) t_{n+1} \right\} u_{n+1} \leq \exp\left\{ \Big(\theta-\frac{\omega}{\alpha \beta} \Big) t_1 \right\} u_1 + \sum_{k=2}^{n+1} \exp\left\{ \Big(\theta-\frac{\omega}{\alpha \beta} \Big) t_k \right\} \eta_{k}.
		\end{equation*}
		To simplify notation, we denote $J=\theta-{\omega}/{(\alpha \beta)}$. Notice that $\omega < \alpha\beta\theta$ and then
		\begin{equation*}
			\begin{aligned}
				\sum_{k=2}^{n+1}  \rme^{J t_{k}} \eta_{k} &= \sum_{k=2}^{n+1}  \rme^{J  t_{k-1} + J  \eta_{k} } \eta_{k}
				\ \leqslant \  \rme^{ J  \eta_{1} } \sum_{k=2}^{n+1}  \rme^{J  t_{k-1}  } \eta_{k} \\
				&\leqslant \rme^{ J  \eta_{1} } \int_{t_1}^{t_{n+1}}  \rme^{J t}\dif t
				\ \leqslant \ J^{-1}  \rme^{ J (\eta_1 + t_{n+1})} .
			\end{aligned}
		\end{equation*}	
		Hence, there is a positive constant $C$ such that
		\begin{equation*}
			u_{n+1} \leq C, \quad \forall n\in \mathbb{N}_0.
		\end{equation*}
		Consequently, we complete the proof according to the definition of $u_n$.
	\end{proof}

	\section{Proofs of the lemmas in Section \ref{sec:SDECLTMDP}}\label{app_C}
	
	Recall that the Stein's equation \eqref{e:steinbb}, for any $h\in \mathcal{B}_b(\R^d,\R)$,
	\begin{equation*}
		\mathcal{A}^{\alpha}f_h(\bfx) = h(\bfx)-\mu(h).
	\end{equation*}
	It follows from \cite[Proposition 6.1]{Fang2019Multivariate} and Lemma \ref{lem:eeSDEs} that the solution $f_h$ can be expressed as
	\begin{equation}\label{e:tlf1}
		f_h(\bfx) = -\int_0^{\infty} \opP_t[h(\bfx)-\mu(h)] \dif t.
	\end{equation}
	The solution $f_h$  can be also expressed as (see, \cite[Lemma 2]{Jin2024An})
	\begin{equation}\label{e:tlf2}
		f_h(\bfx) = \int_0^{\infty}  \rme^{- a t}\opP_t[ a  f_h(\bfx)-h(\bfx)+\mu(h)] \dif t, \ \ \forall  a \ > \ 0.
	\end{equation}
	
	If the test function $h\in \mathcal{C}_b^{1}(\R^d,\R)$, denote $\widetilde{f}_h$ as the corresponding solution to the Stein's equation \eqref{e:steinbb} which also can be expressed as those in \eqref{e:tlf1} and \eqref{e:tlf2}.

	To discuss the regularities for the Stein's equation \eqref{e:steinbb}, we consider the third kind of Stein's equation: for $h\in \mathcal{C}_b^{1}(\R^d,\R)$,
	\begin{equation*}
		\mathcal{A}_{\e} \widetilde{f}_{\e,h}(\bfx)
		= h(\bfx)-\mu_{\e}(h),
	\end{equation*}
	where $\mathcal{A}_{\e}$ and $\mu_{\e}$ are the infinitesimal operator and the invariant measure for the process $(\bfX_t^{\e})_{t\geq 0}$ in \eqref{e:SDEe} respectively.
	It follows from \cite[Proposition 6.1]{Fang2019Multivariate} and Lemma  \ref{lem:eeSDEs} that the solution $\widetilde{f}_{\e,h}$ can be expressed as
	\begin{equation*} 
		\widetilde{f}_{\e,h}(\bfx) = -\int_0^{\infty} \opP^{\e}_{t}[h(\bfx)-\mu_{\e}(h)] \dif t,
	\end{equation*}
	which implies that there exist some positive constants $c$ and $C$, both independent of $\e$, such that
	\begin{equation}\label{e:tlfeest1}
		|\widetilde{f}_{\e,h}(\bfx)|
		\leq C(1+|\bfx|^{2 \kappa }).
	\end{equation}	
	In addition, it follows from Lemma  \ref{pro:mu-mue} that $\opP_t^{\e}h(\bfx) \to \opP_t h(\bfx)$ and
	$\mu_{\e}(h)\to \mu(h)$ as $\e \to 0$, hence, we obtain that
	\begin{equation}\label{e:tlff}
		\lim_{\e \to 0} \widetilde{f}_{\e,h}(\bfx)
		= -\lim_{\e \to 0}\int_0^{\infty} \opP^{\e}_{t}[h(\bfx)-\mu_{\e}(h)] \dif t
		= -\int_0^{\infty} \opP_{t}[h(\bfx)-\mu(h)] \dif t
		= \ \widetilde{f}_h(\bfx).
	\end{equation}	
	Furthermore, the solution $\widetilde{f}_{\e,h}$ also can be expressed as (see, \cite[Lemma 2]{Jin2024An})
	\begin{equation*}
		\widetilde{f}_{\e,h}(\bfx)
		= \int_0^{\infty}  \rme^{- a  t}\opP^{\e}_{t}[ a  \widetilde{f}_{\e,h}(\bfx)-h(\bfx)+\mu_{\e}(h)] \dif t, \quad  \forall \   a  >  0 .
	\end{equation*}
	It follows from Lemma \ref{pro:fh} that $\widetilde{f}_{\e,h} \in \mathcal{C}^1(\R^d,\R)$. Then it follows from \cite[Theorem 1.1]{zhang2013derivative} that for any $\bfu\in \R^d$,
	\begin{equation*}
		\nabla_{\bfu} \widetilde{f}_{\e,h}(\bfx)
		\ = \ \int_0^{\infty}  \rme^{- a  t}
		\E\left[ \frac{1}{S_t} \left( a  \widetilde{f}_{\e,h}(\bfX_s^{\e,\bfx})
		-h(\bfX_s^{\e,\bfx})+\mu_{\e}(h)\right) \right] \int_0^t \bmsigma^{-1} \nabla_{\bfu} \bfX_{s}^{\e,\bfx} \dif {\bf W}_{S_s} \dif s,
	\end{equation*}
	where the subordinated $({\bf W}_{S_t})_{t\geq 0}$ is the rotationally symmetric $\alpha$ stable process. Using \cite[Theorem 1.1]{zhang2013derivative} again, there exists some positive constant $C$ independent of $\e$ such that
	\begin{equation}\label{e:tlNufeest}
		\begin{aligned}
			&\mathrel{\phantom{=}}
			|\nabla_{\bfu} \widetilde{f}_{\e,h}(\bfx)| \\
			&= \int_0^{\infty}  \rme^{- a  t} \left[
			a  \left(\E |\widetilde{f}_{\e,h}(\bfX_s^{\e,\bfx})|^2\right)^{1/2}
			+\left(\E|h(\bfX_s^{\e,\bfx})-\mu_{\e}(h)|^2\right)^{1/2} \right] \|\bmsigma^{-1}\|_{\rm op} \rme^{\| \nabla g_{\e}\|_{\rm op} t} t^{-1/\alpha} \dif t  \\
			&\leq C(1+|\bfx|^{2\beta}),
		\end{aligned}
	\end{equation}
	where the last inequality holds from taking $ a >\|\nabla g_{\e}\|_{\rm op}+1$ and  moment estimate for $\E |\widetilde{f}_{\e,h}(\bfX_s^{\e,\bfx})|^2$, that is,
	\begin{equation*}
		\E|\widetilde{f}_{\e,h}(\bfX_s^{\e,\bfx})|^2 \leq C\E(1+|\bfX_s^{\e,\bfx}|^{4 \kappa })
		\ \leq \ C(1+|\bfx|^{4 \kappa }),
	\end{equation*}
	where the first inequality holds from the estimate for $\widetilde{f}_{\e,h}$ in  \eqref{e:tlfeest1}, the second inequality holds from the moment estimate for
	$\E|\bfX_s^{\e,\bfx}|$ in \eqref{e:Xmon} and the fact $4 \kappa <1$.

	By the convergence results in \eqref{e:tlff}, Lemmas \ref{lem:moment-est} and \ref{l:XeCon}, we know that
	\begin{equation}\label{e:tlfef}
		\begin{aligned}
			& \mathrel{\phantom{=}}
			\lim_{\e \to 0}\nabla_{\bfu} \widetilde{f}_{\e,h}(\bfx) \\
			&= \lim_{\e \to 0} \int_0^{\infty}  \rme^{- a  t}
			\E\left[ \frac{1}{S_t} \left( a  \widetilde{f}_{\e,h}(\bfX_s^{\e,\bfx})
			-h(\bfX_s^{\e,\bfx})+\mu_{\e}(h)\right) \right] \int_0^t \bmsigma^{-1} \nabla_{\bfu} \bfX_{s}^{\e,\bfx} \dif {\bf W}_{S_s} \dif t   \\
			&= \int_0^{\infty}  \rme^{- a  t}
			\E\left[ \frac{1}{S_t} \left( a  \widetilde{f}_h(\bfX_s^{\bfx})
			-h(\bfX_s^{\bfx})+\mu(h)\right) \right] \int_0^t \bmsigma^{-1} \nabla_{\bfu} \bfX_{s}^{\bfx} \dif {\bf W}_{S_s} \dif t.
		\end{aligned}
	\end{equation}	
	Since the operator $\nabla$ is closed (see Partington \cite[Theorem 2.2.6]{partington2004linear}), we know
	$$\lim_{\e \to 0}\nabla_{\bfu} \widetilde{f}_{\e,h}(\bfx)=\nabla_{\bfu} \widetilde{f}_h(\bfx).$$
	
	\begin{proof}[Proof of Lemma  \ref{pro:steinbb}.]
		(i) We shall use the expression for $f_h$ in \eqref{e:tlf1} and exponential ergodicity in Lemma  \ref{lem:eeSDEs} to get the estimate for $|f_h(\bfx)|$ for those two kinds of $\alpha\in(1,2)$ stable noises $(\bfZ_t)_{t\geq 0}$.
		
		Since $h\in \mathcal{B}_b(\R^d,\R)$, one knows that $|h(\bfx)|\leq C\widetilde{V}(\bfx), \forall  \bfx\in \R^d$. It follows from Lemma  \ref{lem:eeSDEs} that
		\begin{equation}\label{e:SE2}
			\begin{aligned}
				|f_h(\bfx)|
				&\leq  \int_0^{\infty} |\opP_t h(\bfx)-\mu(h)|  \dif t
				\ \leq \  \|h\|_{\infty} \int_0^{\infty} \| \opP^*_{t}\delta_{\bfx} - \mu \|_{ \rm TV,\widetilde{V} }   \dif t   \\
				&\leq C\|h\|_{\infty}(1+\widetilde{V}(\bfx)  )
				\  \leq \ C \|h\|_{\infty}(1+|\bfx|^{2 \kappa }).
			\end{aligned}
		\end{equation}
		
		(ii) We need to discuss the estimate for $|\nabla f_h(\bfx)|$ separately with respective to different $\alpha$ stable noises.
		
		{\bf Case 1. For the rotationally symmetric $\alpha$ stable noise}. Let $h \in \mathcal{C}_b^1(\R^d,\R)$. It follows from \eqref{e:tlNufeest} and \eqref{e:tlfef} that
		\begin{equation}\label{e:tlnuf}
|\nabla_{\bfu} \widetilde{f}_h(\bfx)|
\leq C \|h\|_{\infty}(1+|\bfx|^{2 \kappa }) |\bfu|.
		\end{equation}
		
		Secondly, we extend $h\in \mathcal{C}_b^1(\R^d,\R)$ to $h\in \mathcal{B}_b(\R^d,\R)$ above, by the standard approximation (see Fang et al. \cite[pp. 968-969]{Fang2019Multivariate}). Let $h\in \mathcal{B}_b(\R^d,\R)$, and define
		\begin{equation*}
			h_{\delta}(\bfx) = \int_{\R^d} \varphi_{\delta}(\bfy) h(\bfx-\bfy) \dif \bfy, \quad \delta>0,
		\end{equation*}
		where $\varphi_{\delta}$ is the density function of the normal distribution $\mathcal{N}(\mathbf{0},\delta^2 \mathbf{I})$. Thus $h_{\delta}$ is smooth, $\| h_{\delta} \|_{\infty}  \leq \| h\|_{\infty}$ and the solution to the Stein's equation \eqref{e:steinbb} with $h$ replaced by $h_{\delta}$, is
		\begin{equation*}
			f_{h,\delta}(\bfx)
			= -\int_0^{\infty} \opP_t [h_{\delta}(\bfx) - \mu(h_{\delta})] \dif t.
		\end{equation*}
		
		Denote $\hat{h}_{\delta}=-h_{\delta}+\mu(h_{\delta})$. Since $h_{\delta} \in \mathcal{C}_b^1(\R^d,\R)$, one has $h_{\delta}(\bfx) \leq \|h_{\delta}\|_{\infty}( 1+\widetilde{V}(\bfx))$ for all $\bfx\in \R^d$. From Lemma  \ref{lem:eeSDEs}, one has
		\begin{equation*}
			|\opP_t h_{\delta}(\bfx) - \mu(h_{\delta}) |
			\leq \|h_{\delta}\|_{\infty} \| \opP^*_{t}\delta_{\bfx} - \mu \|_{\rm TV,\widetilde{V}}
			\ \leq \ C \|h_{\delta}\|_{\infty}(1+\widetilde{V}(\bfx))  \rme^{-ct}.
		\end{equation*}
		With similar calculations for \eqref{e:SE2}, one has
		\begin{equation*}
			|f_{h,\delta}(\bfx)|
			\leq C \|h_{\delta}\|_{\infty}(1+|\bfx|^{2 \kappa }).
		\end{equation*}
		By the dominated convergence theorem, one has
		\begin{equation*}
			\lim_{\delta \to 0} f_{h,\delta}(\bfx)
			\ = \ - \int_0^{\infty} \opP_t[h(\bfx)-\mu(h)] \dif t
			\ = \ f_h(\bfx).
		\end{equation*}
		From \eqref{e:SE2} and $\| h_{\delta} \|_{\infty} \leq \| h \|_{\infty}$, one has
		\begin{equation*}
			|f_{h,\delta}(\bfx)|
			\leq C \|h_{\delta}\|_{\infty}(1+|\bfx|^{2 \kappa })
			\ \leq \ C \|h\|_{\infty}(1+|\bfx|^{2 \kappa }).
		\end{equation*}
		Letting $\delta \to 0$, one obtains
		\begin{equation}\label{e:fh1}
			|f_h(\bfx)|
			\ \leq \ C \|h\|_{\infty}(1+|\bfx|^{2 \kappa }).
		\end{equation}
		
		From calculations for \eqref{e:tlfef}, one has
		\begin{equation*}
			\nabla_{\bfu} f_{h,\delta}(\bfx)
			= \int_0^{\infty}  \rme^{- a  t}
			\E\left[ \frac{1}{S_t} \left( a  f_{h,\delta}(\bfX_s^{\bfx})
			-h_{\delta}(\bfX_s^{\bfx})+\mu(h_{\delta})\right) \right] \int_0^t \sigma^{-1} \nabla_{\bfu} \bfX_{s}^{\bfx} \dif {\bf W}_{S_s} \dif t
		\end{equation*}
		With similar calculations for the proof of \eqref{e:tlnuf}, one has
		\begin{equation}\label{e:fdh}
			|\nabla_{\bfu} f_{h,\delta}(\bfx)|
			\leq C\|h_{\delta}\|_{\infty}(1+|\bfx|^{2 \kappa }) |\bfu|
			\ \leq \ C\|h\|_{\infty}(1+|\bfx|^{2 \kappa }) |\bfu|.
		\end{equation}		
		Since the operator $\nabla$ is closed (see Partington \cite[Theorem 2.2.6]{partington2004linear}), it follows from the dominated convergence theorem that
		\begin{equation*}
			\begin{aligned}
				&\mathrel{\phantom{=}}
				\lim_{\delta \to 0} \nabla_{\bfu} f_{h,\delta}(\bfx) \\
				&= \lim_{\delta \to 0}  \int_0^{\infty}  \rme^{- a  t}
				\E\left[ \frac{1}{S_t} \left( a  f_{h,\delta}(\bfX_s^{\bfx})
				-h_{\delta}(\bfX_s^{\bfx})+\mu(h_{\delta})\right) \right] \int_0^t \bmsigma^{-1} \nabla_{\bfu} {\bfX}_{s}^{\bfx} \dif {\bf W}_{S_s} \dif t \\
	&= \int_0^{\infty}  \rme^{- a  t}
\E\left[ \frac{1}{S_t} \left( a  f_h(\bfX_s^{\bfx})
-h(\bfX_s^{\bfx})+\mu(h_{\delta})\right) \right] \int_0^t \bmsigma^{-1} \nabla_{\bfu} \bfX_{s}^{\bfx} \dif {\bf W}_{S_s} \dif t
				= \nabla_{\bfu} f_h(\bfx).
			\end{aligned}
		\end{equation*}
		Let $\delta \to 0$, it follows from \eqref{e:fdh} that
		\begin{equation*}
|\nabla_{\bfu} f_h(\bfx)|
\leq  C\|h\|_{\infty}(1+|\bfx|^{2 \kappa }) |\bfu|.
		\end{equation*}

		{\bf Case 2. For the  cylindrical  $\alpha$ stable noise.} Since $g$ is Lipschitz continuous, it follows from \eqref{e:tlf2} and \cite[Theorem 1.2]{wang2015harnack} that for any $\bfu\in \R^d$,
		\begin{equation*}
			\begin{aligned}
				&\mathrel{\phantom{=}}
				|\nabla_{\bfu} f_h(\bfx)|
				= \left| \int_0^{\infty}  \nabla_{\bfu} [ a   \rme^{- a  t} \E f_h(\bfX_t^\bfx) -  \rme^{- a  t} \E (h(\bfX_t^\bfx)-\mu(h))] \dif t \right|   \\
				&\leq C\int_0^{+\infty}  \rme^{- a  t}(1+t \rme^{d \|g\|_{\rm Lip}t})t^{1-1/\alpha} \left[  a (\E|f_h(\bfX_t^{\bfx})|^2)^{1/2}
+(\E|h(\bfX_t^{\bfx})-\mu(h)|^2)^{1/2}
				\right] \dif t  \\
				&\leq C(1+|\bfx|^{2 \kappa }),
			\end{aligned}
		\end{equation*}
		where the last inequality holds from taking $ a >d\|g\|_{\rm Lip}$ with $\|g\|_{\rm Lip}=\sup_{\bfx,\bfy \in \mathbb{R}^d}\frac{|g(\bfx)-g(\bfy)|}{|\bfx-\bfy|}$ and the moment estimate for $\E|f_h(\bfX_t^{\bfx})|^2$ as below
		\begin{equation*}
			\E|f_{h}(\bfX_s^{\bfx})|^2 \leq C\E(1+|\bfX_s^{\bfx}|^{4 \kappa })
			\ \leq \ C(1+|\bfx|^{4 \kappa }),
		\end{equation*}
		where the first inequality holds from the estimate for $|f_h|$ in \eqref{e:fh1}, the second  inequality holds from the moment estimate for $  \E|\bfX_t^{\bfx}|$ in \eqref{e:Xmon} and the fact $4 \kappa  <1$. The proof is complete.
	\end{proof}
	
\end{appendix}

\section*{Acknowledgements}
L. Xu is supported by the National Natural Science Foundation of China No. 12071499, the Science and Technology Development Fund (FDCT) of Macau S.A.R. FDCT 0074/2023/RIA2, and the University of Macau grants MYRG2020-00039-FST, MYRG-GRG2023-00088-FST.  G. Pang is supported in part by the US National Science Foundation grant DMS-2216765. X. Jin was supported in part by the National Natural Science Foundation of China No. 12401176, and the Fundamental Research Funds for the Central Universities grant No. JZ2023HGTA0170.

\bibliographystyle{acm}
\normalem
\bibliography{piecewise_reference}

\end{document}